%% file: main.tex
\documentclass[10pt, a4paper, USenglish, abstract=true]{scrartcl}
\input{preamble}

\input{makros}
\usepackage{geometry}
\usepackage{xfrac}
\setkomafont{disposition}{\normalcolor\bfseries}
 \DeclareNameAlias{default}{family-given}
 \DefineBibliographyStrings{english}{%
 page             = {},
 pages            = {},
 }  
 \DeclareFieldFormat[article]{citetitle}{#1}
 \DeclareFieldFormat[article]{title}{#1} 
\numberwithin{equation}{section}

\usepackage[normalem]{ulem} 
\author{\textsc{Christian Scharrer} and \textsc{Alexander West}}
\date{\today}
\title{Energy quantization for constrained Willmore surfaces}

\begin{document}
\maketitle 
\begin{abstract}
We establish an energy quantization for constrained Willmore surfaces, where the constraints are given by area, volume, and total mean curvature, assuming that the underlying conformal structures remain bounded. Furthermore, we show strong compactness of constrained Willmore surfaces under some energy threshold, proving in particular the strong compactness of minimizers of two previously studied problems.
\end{abstract}
\section{Introduction}\label{sec:Introduction}
Let $\Sigma$ be a closed, connected, and oriented genus $p$ surface and $\vec \Phi:\Sigma \to \R^3$ be a smooth immersion. We denote by $g_{\R^3}$ the standard Riemannian metric in $\R^3$ and by $g=g_{\vec \Phi} \cqq \vec \Phi^* g_{\R^3}$ the pull-back metric via $\vec \Phi$. The \emph{Gauss map} $\vec n$ of $\vec \Phi$ is defined in any positive chart as
\[\vec n  \cqq \frac{\parpar{x_1}{\vec \Phi}\times \parpar{x_2}{\vec \Phi}}{|\parpar{x_1}{\vec \Phi}\times \parpar{x_2}{\vec \Phi}|}.\]
$\dif{\mu}$ is the Riemannian measure induced by $g$. The \emph{second fundamental form} $\mb I$ of $\vec \Phi$ is defined at a point $x\in \Sigma$ as
\[\mb I_x:T_x \Sigma \times T_x \Sigma \to \R, \quad \mb I_x(v,w)\cqq -\langle d \vec n_x(v), d \vec \Phi_x(w) \rangle.\]
We define the \emph{scalar mean curvature} $H$, the \emph{mean curvature} $\vec H$, and the \emph{Gaussian curvature} $K$ by
\[H = \frac{1}{2} \tr_g(\mb I),\quad \vec H = H \vec n,\quad K = \det \hspace{-1pt}_g(\mb I).\]
 The \emph{Willmore energy} $\mc W(\vec \Phi)$ and the \emph{Dirichlet energy} $\mc E(\vec \Phi)$ of the Gauss map $\vec n$ are defined as
\[\mc W(\vec \Phi)\cqq \int_{\Sigma} |\vec H|^2 \dif{\mu},\quad \mc E (\vec \Phi) \cqq \int_{\Sigma} |d \vec n|_{g}^2 \dif{\mu}.\]
As a consequence of the Gauss-Bonnet theorem, these functionals are closely connected via
\[\mc E(\vec \Phi) = 4\mc W(\vec \Phi) - 4\pi \chi(\Sigma) = 4\mc W(\vec \Phi) - 8\pi(1-p),\]
where $\chi(\Sigma) = 2(1-p)$ is the Euler-characteristic of $\Sigma$. We define the \emph{area} $\mc A$, \emph{algebraic volume} $\mc V$, and \emph{total mean curvature} $\mc M$ of $\vec \Phi$ by 
\begin{equation}
\mc A(\vec \Phi) = \int _{\Sigma} 1 \dif{\mu}, \quad \mc V(\vec \Phi) = -\frac{1}{3} \int _{\Sigma} \langle \vec n, \vec \Phi\rangle \dif{ \mu} , \quad \mc M(\vec \Phi) = \int _{\Sigma} H \dif{ \mu}.\label{intro:area total mean curvature and area}
\end{equation}
If $\vec \Phi$ is an embedding, we choose the orientation of $\Sigma$ such that the Gauss map defined above is the \emph{inner} unit normal vector field. In this case, $\mc V$ corresponds to the volume enclosed by $\vec \Phi(\Sigma)$.

The Willmore energy and generalizations have been the subject of extensive research. Already appearing in the study of vibrations of thin plates in the early 19th century \cite{Germain, Poisson}, it nowadays plays a central role in many branches of physics and biology. One prominent example and the main motivation for this paper is the study of lipid bilayer membranes \cite{Seifert97}. Such a membrane is formed by two sheets of amphiphilic molecules, so-called lipids. They consist of a hydrophilic head and two hydrophobic tails. In an aqueous solution, these lipids self aggregate into two sheets in such a way that the hydrophobic tails are protected from the surrounding solution by the hydrophilic heads. Any border of these two sheets would again expose the hydrophobic tails and so the bilayer is naturally closed and forms a so-called \emph{vesicle}. The thickness of the resulting membrane is small compared to its global size. It is thus feasible to model them with smooth immersions $\vec \Phi: \Sigma \to \R^3$.

While the surface tension is the leading force that determines the shape of soap bubbles, the shape of a bilayer membrane is governed by its \emph{bending energy}. It was proposed by \textcite{Helfrich} that such an energy could be modelled by \cite[(12)]{Helfrich}
\[\frac{k_c}{2}\int _{\Sigma} \left (2H-c_0\right )^2\dif{\mu}  + \overline{k}_c \int _{\Sigma} K \dif{\mu}.\]
Here, $k_c$ and $\overline{k}_c$ are the curvature elastic moduli and $c_0$ is the \emph{spontaneous curvature}. Because of the Gauss-Bonnet theorem, the integral of $K$ is a topological constant and as we will be working with a fixed genus, this term does not play a role in variational problems. We thus define the \emph{Helfrich energy} $\mc H_{c_0}$
 \begin{equation}
 \mc H_{c_0}(\vec \Phi) \cqq \int _{\Sigma} \left (H-\frac{c_0}{2}\right )^2\dif{\mu} = \mc W(\vec \Phi) - c_0 \mc M(\vec \Phi) + \frac{c_0^2}{4}\mc A(\vec \Phi).\label{Helfrich energy}
 \end{equation}
The parameter $c_0$ is introduced in order to model the potential area deficit between the two layers and the chemical imbalance of the two regions separated by the membrane.

We may assume that the two layers of our vesicle are incompressible (but still bendable) and view the total surface area $\mc A$ as constant. Furthermore, we may also constrain the enclosed volume $\mc V$. The reason is that while the membrane is permeable to water, it is not permeable by other molecules, e.g., sugar or larger ions. Consequently, an exchange of water between the interior and exterior of the vesicle would change the osmotic pressure, leading to effects that are orders of magnitude larger than the bending energy \cite[Section 1.3]{Seifert97}.

To implement the additional constraints on area and enclosed volume, it is natural to consider the energy functional
\begin{equation}
\mc F(\vec \Phi) \cqq \mc W(\vec \Phi) + \alpha \mc A(\vec \Phi)+ \beta \mc V(\vec \Phi) + \gamma \mc M(\vec \Phi)\label{functional F definition}
\end{equation}
for $\alpha$, $\beta$, $\gamma \in \R$ fixed. By the Lagrange multiplier principle, minimizers of the Helfrich energy \eqref{Helfrich energy} subject to the given constraints are stationary points of $\mc F$. Denoting by $\delta$ the first variation, this means that we consider immersions satisfying
 \begin{equation}
 \delta\mc W(\vec \Phi) + \alpha \delta \mc A(\vec \Phi) + \beta \delta \mc V(\vec \Phi) + \gamma \delta \mc M(\vec \Phi)=0.\label{intro:EL equation}
 \end{equation}
We call an immersion satisfying \eqref{intro:EL equation} a \emph{constrained Willmore surface} or \emph{constrained Willmore immersion}. 
%

A large class of constrained Willmore immersions is given by minimizers of the following problems. We define the \emph{isoperimetric ratio} $\mc I(\vec \Phi)$ and the \emph{normalized curvature} $\mc T(\vec \Phi)$ as
\[\mc I(\vec \Phi ) \cqq \frac{\mc A(\vec \Phi)}{\mc V(\vec \Phi)^{2/3}},\quad \mc T(\vec \Phi) \cqq \frac{\mc M(\vec \Phi)}{\sqrt{\mc A(\vec \Phi)}}.\]
Definitions of $\mc I$ vary in literature. Notice that due to the isoperimetric inequality, $\mc I(\vec \Phi)\geq \sqrt[3]{36\pi}$. Fix a reference genus $p$ surface $\Sigma$ and consider the following minimization problems
\begin{equation}
\beta^{\mc I}_p(\sigma) \cqq \inf \{\mc W(\vec \Phi),\,  \vec \Phi:\Sigma \to \R^3\text{ is a smooth embedding and } \mc I(\vec \Phi) = \sigma\}\label{isoperimetric problem}
\end{equation}
and
\begin{equation}
\beta^{\mc T}_p(\tau) \cqq \inf \{\mc W(\vec \Phi),\,  \vec \Phi:\Sigma \to \R^3\text{ is a smooth embedding and }\mc T(\vec \Phi) = \tau\}.\label{total mean curvature problem}
\end{equation}

The global minimum $\bm{\beta}_p$ of the Willmore energy is defined as
\[\bm{\beta}_p \cqq \inf \{\mc W(\vec \Phi),\, \vec \Phi:\Sigma \to \R^3\text{ is a smooth embedding}\}. \]
$\bm{\beta}_p$ admits minimizers for all $p$ as proved in \cite{BauerKuwert}, building on the previous work \cite{Simon}. Existence of minimizers for \eqref{isoperimetric problem} for all $\sigma \in (\sqrt[3]{36\pi}, \infty)$ and arbitrary genus was shown to exist for genus $p=0$ in \cite{Schygulla}, whereas for arbitrary genus, the existence of minimizers was proved in \cite{KMR} under the condition that $\beta_p^{\mc I}(\sigma) < \min\{\bm{\beta}_p + \beta_0^{\mc I}(\sigma)-4\pi, 8\pi\}$. The bound $\beta_p^{\mc I}(\sigma) <\bm{\beta}_p + \beta_0^{\mc I}(\sigma)-4\pi$ was proved in \cite{MondinoScharrerInequality}, whereas the $8\pi$ bound was proved in \cite{ScharrerDelaunay} for $p=1$ and in \cite{KusnerMcGrath} for $p>1$. For \eqref{total mean curvature problem}, existence of minimizers for $\tau \in (0, \sqrt{8\pi})\setminus \{\sqrt{4\pi}\}$ and arbitrary genus was established by the authors in \cite{MasterThesis} using similar techniques. A quick calculation shows that minimizers $\vec \Phi$ of \eqref{isoperimetric problem} satisfy \eqref{intro:EL equation} with 
\begin{equation}
\alpha = \lambda \mc V(\vec \Phi)^{-2/3},\quad\beta = -\frac{2 \lambda}{3} \mc I(\vec \Phi) \mc V(\vec \Phi)^{-1},\quad  \gamma=0\label{explicit lagrange multipliers I}
\end{equation}
 for some $\lambda \in \R$, whereas minimizers of \eqref{total mean curvature problem} satisfy \eqref{intro:EL equation} with
\begin{equation}
\alpha = -\frac{\lambda}{2} \mc T(\vec \Phi) \mc A(\vec \Phi)^{-1} ,\quad\beta = 0,\quad\gamma = \lambda \mc A(\vec \Phi)^{-1/2}.\label{explicit lagrange multipliers T}
\end{equation}

Vesicles are observed to undergo various shape transformations when certain parameters like temperature or osmotic pressure are changed. One notable example is the so-called \emph{budding transition} described in \cite{KaesSackmann}. In this process, an increase in temperature causes the vesicle to adopt a more prolate shape. It then becomes pear-shaped in such a way that the surface forms two spherical bubbles connected by a small catenoidal neck. A further increase in temperature eventually closes the catenoidal bridge and splits the surface into two closed vesicles, touching at a point. 

In the scenario above, three limiting surfaces form. Two macroscopic vesicles which touch at a point, and a microscopic vesicle which converges to a catenoid after suitable rescaling. We call these regions of energy concentration, microscopic or macroscopic, a \emph{bubble}. A natural question is whether there is any energy lost in the limit, which was needed to connect these bubbles to one another. We will answer this question in \Cref{thm:Energy quantization} as the main result of this paper.

Mathematically, we want to study limits of sequences $\vec \Phi_k:\Sigma \to \R^3$ of constrained Willmore immersions from a genus $p$ surface $\Sigma$ into $\R^3$ with uniformly bounded Willmore energy. To understand why a bubbling phenomenon as described above occurs, the following $\eps$-regularity result needs to be established.

\begin{thm} \label{thm:eps regularity}
There exists $\eps_0>0$ such that for any conformal, constrained Willmore immersion $\vec \Phi:B_1\to \R^3$ with coefficients $\alpha$, $\beta$, $\gamma$ in the sense of \eqref{intro:EL equation} such that
\begin{equation}
\int _{B_1} |\nabla \vec n _{\vec \Phi} |^2 \dif{x} < \eps_0,\label{small dirichlet energy}
\end{equation}
it holds for any $r\in (0,1)$ 
\begin{align}
\|\nabla ^l \vec n\|_{L^\infty(B_{r})} &\leq C(r, l, \|\nabla \lambda \|_{L^1(B_1)}, e^{2\overline{\lambda}}\alpha,e^{\overline{\lambda}}\gamma) \bigg (\left ( \int_{B_1} |\nabla \vec n |^2 \dif{x} \right )^{1/2}+ e^{3\overline{\lambda}}|\beta| \bigg ) \label{eps regularity 1}\\
\|e^{-\lambda} \nabla ^ l \vec \Phi \|_{L^\infty(B_{r})} &\leq C(r, l, \|\nabla \lambda \|_{L^1(B_1)}, e^{2\overline{\lambda}}\alpha,e^{3\overline{\lambda}}\beta, e^{\overline{\lambda}}\gamma),\label{eps regularity 2}
\end{align}
for all integers $l\geq 1$. Here, $\lambda$ is the conformal factor of $\vec \Phi$ and $\overline{\lambda} = \frac{1}{|B_{1}|} \int _{B_{1}} \lambda \dif{x}$ its average. Furthermore, there exists $\eps_1 =\eps_1(\|\nabla \lambda\|_{L^1(B_1)})> 0$ depending only on $\|\nabla \lambda \|_{L^1(B_1)}$ such that if
\begin{equation}
\int _{B_1} |\nabla \vec n _{\vec \Phi}|^2 \dif{x} < \eps_1(\|\nabla \lambda\|_{L^1(B_1)}),\label{Dirichlet energy smaller than eps depending on lambda}
\end{equation}
then
\begin{equation}
e^{3\overline{\lambda}}|\beta| \leq C(\|\nabla \lambda\|_{L^1(B_1)}, e^{2\overline{\lambda}}\alpha,e^{\overline{\lambda}}\gamma) \left (\int_{B_1} |\nabla \vec n|^2 \dif{x}\right )^{1/2},\label{c3 estimate}
\end{equation}
and thus
\begin{equation}
\|\nabla ^l \vec n\|_{L^\infty(B_{r})} \leq C(r, l, \|\nabla \lambda \|_{L^1(B_1)}, e^{2\overline{\lambda}}\alpha, e^{\overline{\lambda}}\gamma) \left ( \int_{B_1} |\nabla \vec n |^2 \dif{x} \right )^{1/2}. \label{eps regularity 3}
\end{equation}
\end{thm}

An $\eps$-regularity result for unconstrained (i.e., $\alpha=\beta=\gamma=0$) Willmore surfaces was already proved in the extrinsic setting by \textcite{WillmoreFlowSmallInitialEnergy}. It was discovered by \textcite{RiviereAnalysisAspects} that in conformal coordinates, the gradient of the Willmore energy can be put into divergence form. This weak formulation gave rise to an intrinsic $\eps$-regularity for unconstrained Willmore immersions \cite[Theorem I.5]{RiviereAnalysisAspects}, using compensated integration techniques and the special form of the Willmore gradient. \textcite{BernardNoether} showed that this divergence form was really a consequence of Noether's theorem, making use of the invariance of the Willmore energy under dilations, isometries and inversions. He further showed that these divergence formulations also exist for our constraints, see \eqref{First variation of Area}, \eqref{First variation of Volume}, and \eqref{First variation of Total Mean Curvature}. 

A related $\eps$-regularity result for a large class of constrained Willmore surfaces was established by \textcite[Theorem 1.1]{BernardWheelerWheeler}. The reason we cannot apply their theorem is that they obtain an estimate for which the right-hand side of \eqref{eps regularity 1} takes the form $C(\|\nabla \vec n\|_{L^2(B_1)}+1)$. However, in order for us to prove small $L^{2,\infty}$-estimates for $\nabla \vec n$ in neck regions later on, we need that the right-hand side depends linearly on $\|\nabla \vec n\|_{L^2(B_1)}$, see \eqref{eps regularity 3} and \eqref{pointwise nabla n bound}. This is achieved at the cost of letting the threshold $\eps$ depend on the gradient of the conformal factor.

Such an $\eps$-regularity leads to a concentration of energy phenomenon, initially observed by \textcite{SacksUhlenbeck} in the context of harmonic maps. We can reparametrize the immersions such that they are conformal with respect to a constant scalar curvature metric of unit volume. If we assume that the induced conformal structures remain compact in the moduli space, we obtain local $L^\infty$-bounds of the conformal factor up to rescaling, away from finitely many points. Then, the immersions smoothly converge, again away from finitely many points, where bubbles may develop. These bubbles are parametrized on a smaller scale than the remaining surface. The regions between the different scales are called \emph{neck regions}. Due to the $\eps$-regularity, the bubble regions converge smoothly and no energy is lost there. Consequently, the question whether any energy is lost in the limit amounts to answering whether the neck regions carry any non-zero energy in the limit. This question was answered by \textcite{BernardRiviere} in the case of unconstrained Willmore immersions with bounded Willmore energy and a bounded conformal structure, showing that no energy is lost. They proved $L^{2,\infty}$- and $L^{2,1}$-bounds for the mean curvature in neck regions, which together with the $L^{2,\infty}$ and $L^{2,1}$ duality yields estimates for the Willmore energy in neck regions. This approach was already used by \textcite{LinRiviere} in the context of harmonic maps. In the case of diverging conformal structures, \textcite{LaurainRiviereEnergyQuantization} give an explicit description for the amount of energy lost in neck regions, depending on a certain residue. 

With this, we are now ready to state the main theorem.
%
%
%
\begin{thm}[Energy quantization for constrained Willmore surfaces] \label{thm:Energy quantization}
Let $\Sigma$ be a genus $p$ surface and $\vec \Phi_k$ be a family of constrained Willmore immersions with coefficients $\alpha_k,\beta_k,\gamma_k$ as in \eqref{intro:EL equation} satisfying
\begin{equation}
\limsup_{k\to \infty} \Big[\mc W(\vec \Phi_k) + \mc A(\vec \Phi_k) +|\alpha_k| + |\beta_k| + |\gamma_k|\Big]< \infty.\label{Bounded willmore, area and coefficients}
\end{equation}
We assume that the conformal classes of $g_{\vec \Phi_k}$ remain in a compact subset of the moduli space of $\Sigma$. Then, up to a subsequence, there are immersions $\vec \zeta_\infty:\Sigma \to \R^3$ and for $i=1,\ldots, n$ and $j=1, \ldots, Q^i$ immersions $\vec \zeta^{i,j}:\mb S^2\to \R^3$ which are constrained Willmore, possibly singular, branched and containing ends, such that the energy quantization property
\begin{equation}
\lim _{k \to \infty} \mc W(\vec \Phi_k) = \mc W(\vec \zeta_\infty) + \sum _{i=1}^n \sum _{j=1}^{Q^i} \mc W(\vec \zeta^{i,j})\label{energy quantization result}
\end{equation}
holds. If $\vec \zeta_\infty$ or one of the $\vec \zeta^{i,j}$ contain ends, then it is an unconstrained Willmore immersion and may be inverted to a possibly singular, branched unconstrained Willmore surface without ends. 
\end{thm}

\begin{thm}\label{thm:continuation of energy quantization}
Continuing \Cref{thm:Energy quantization}, these  maps are obtained in the following way: There exists a sequence of diffeomorphisms $f_k$ and constant curvature metrics $h_k$ of unit volume on $\Sigma$ such that $\vec \xi_k \cqq \vec \Phi_k \circ f_k$ are conformal with respect to $h_k$ and $h_k$ converge in $C^\infty(\Sigma)$ to a limiting metric $h_{\infty}$. There are finitely many points $a_1,\ldots, a_n\in \Sigma$ and maps $\Xi_k$ which are compositions of dilations and isometries in $\R^3$ such that
\begin{equation}
\Xi_k \circ \vec \xi_k \to \vec \zeta_\infty \quad \text{in}\quad C^l_{\loc}(\Sigma \setminus \{a_1,\ldots, a_n\})\label{limit immersion on Sigma}
\end{equation}
for all $l\in \N$. Let $\phi^i_k$ be converging conformal coordinates around $a_i$. There are $z^{i,j}_k \in \Sigma$ converging to $a_i$, radii $\rho^{i,j}_k$ converging to 0, and $\Xi^{i,j}_k$ compositions of dilations and isometries in $\R^3$ such that
\[\Xi^{i,j}_k \circ \vec \xi_k \circ \phi_k(\rho^{i,j}_k y + {\phi^{i}_k}^{-1}(z^{i,j}_k)) \to \vec \zeta^{i,j} \circ \pi^{-1}(y) \quad \text{in}\quad C^l_{\loc}(\C \setminus \{a^{i,j}_1,\ldots, a^{i,j}_{n_{i,j}}\}).\]
Here, $\pi:\mb S^2 \to \C \cup \{\infty\}$ is the stereographic projection from $\mb S^2$ to the extended complex plane and $\{a^{i,j}_1,\ldots, a^{i,j}_{n_{i,j}}\}$ are finitely many points in $\C$. Finally, it holds 
\begin{equation}
\Xi_k \circ \vec \xi_k \to \vec \zeta_\infty\quad\text{in $C^l(\Sigma)$ for all $l\in \N$} \iff \lim _{k\to \infty}\mc W(\vec \xi_k) = \mc W(\vec \zeta_\infty).\label{energy quantization: strong convergence everywhere}
\end{equation}
\end{thm}

This theorem allows us to conclude that the minimizers given in \eqref{isoperimetric problem} and \eqref{total mean curvature problem} are compact in the $C^l$-topology for all $l\in \N$.
\begin{thm}[Compactness of constrained Willmore surfaces]
 \label{thm:Compactness of constrained Willmore surfaces}
Let $\Sigma$ be a genus $p$ surface and $0<\delta<1$. The family of immersions $\vec \Phi:\Sigma \to \R^3$ that are constrained Willmore surfaces with coefficients $\alpha,\beta,\gamma$ satisfying
\begin{equation}
\max\{|\alpha|,|\beta|,|\gamma|\}<\frac{1}{\delta},\label{intro:Lagrange multiplier bounded corollary}
\end{equation}
\[\delta \leq \mc A(\vec \Phi)\leq 1/\delta\]
and 
\begin{equation}
\mc W(\vec \Phi)< \begin{cases}
\min\{8\pi,\bm{\beta}_p + \max\{\beta^{\mc T}_0(\mc T(\vec \Phi)), \beta_0^\iso(\iso (\vec \Phi))\} - 4\pi\}-\delta &\quad \text{if $p\geq 1$},\\
8\pi - \delta & \quad \text{if $p=0$}
\end{cases}\label{intro:W bound}
\end{equation}
is compact in $C^l(\Sigma)$ for any $l\in \N$, modulo reparametrization.

\end{thm}
A corollary of the previous theorem is that the minimizers of \eqref{isoperimetric problem} and \eqref{total mean curvature problem} are also compact in the $C^l$-topology. \eqref{intro:W bound} is satisfied using \cite[Corollary 1.5]{MondinoScharrerInequality} and \cite[Theorem 1.2, Proposition 1.3]{MasterThesis}. \eqref{intro:Lagrange multiplier bounded corollary} is a consequence of \eqref{explicit lagrange multipliers I}, \eqref{explicit lagrange multipliers T}, \Cref{lem:boundedness of Lagrange multiplier for minimizers}, and \Cref{rem:iso lagrange multipliers remain bounded}.

\begin{cor}[See \Cref{thm:Compactness of constrained Willmore surfaces}, \Cref{lem:boundedness of Lagrange multiplier for minimizers}, and \Cref{rem:iso lagrange multipliers remain bounded}] \label{cor:Compactness minimizers}
Suppose that $p\in \N_0$, $\{\tau_k, \, k\in \N\}\subset\subset (0, \sqrt{8\pi})\setminus \{\sqrt{4\pi}\}$ and $\{\sigma_k, \, k \in \N\} \subset\subset (\sqrt[3]{36\pi},\infty)$. Denote by $\vec \Phi^{\mc I}_k$ minimizers of \eqref{isoperimetric problem} satisfying $\mc I(\vec \Phi^{\mc I}_k) = \sigma_k$ and $\mc A(\vec \Phi^{\mc I}_k) = 1$. Similarly, denote by $\vec \Phi^{\mc T}_k$ minimizers of \eqref{total mean curvature problem} satisfying $\mc T(\vec \Phi^{\mc T}_k) = \tau_k$ and $\mc A(\vec \Phi^{\mc T}_k) = 1$. Then, after a potential reparametrization, $\{\Phi^{\mc I}_k,\, k\in \N\}$ and $\{\Phi^{\mc T}_k,\, k\in \N\}$ are precompact in $C^l(\Sigma, \R^3)$ for all $l\in \N$.
\end{cor}


The structure of this paper is as follows. In \Cref{sec:Notations}, we introduce necessary notation, specify the notion of constrained Willmore surfaces and briefly recall the notion of Lorentz spaces. In \Cref{sec:eps regularity}, we prove the $\eps$-regularity result \Cref{thm:eps regularity}. In \Cref{sec:Energy quantization}, we recall the conservation laws originally derived in \cite{BernardNoether}, prove $L^p$-estimates for the conformal factor in neck regions and derive Lorentz-type estimates on annuli independent of the conformal class. In \Cref{sec:L21 estimates on mean curvature in neck regions}, we derive $L^{2,1}$-estimates for the mean curvature in neck regions. In \Cref{sec:L2weak quantization in neck regions}, we prove small $L^{2,\infty}$-bounds for the mean curvature in the neck regions. In \Cref{sec:Proof main result}, we prove the main results \Cref{thm:Energy quantization}, \Cref{thm:continuation of energy quantization} and \Cref{thm:Compactness of constrained Willmore surfaces}. In \Cref{sec:Other stuff}, we prove bounds for the Lagrange multipliers for minimizers of \eqref{isoperimetric problem} and \eqref{total mean curvature problem} in order to prove \Cref{cor:Compactness minimizers}. Finally, in Appendix \hyperref[sec:Appendix Bubble neck decomposition]{A}, we recall the bubble-neck decomposition from \cite{BernardRiviere} and in Appendix \hyperref[sec:Appendix Derivation Conservative System]{B}, we rederive the conservation laws for constrained Willmore immersions.

We want to mention that the boundedness of the Lagrange multipliers $\alpha$, $\beta$, $\gamma$ is a necessary condition throughout: An example of constrained Willmore surfaces converging to a round sphere in $W^{2,2}(\mb S^2)$, with diverging Lagrange multipliers, where $C^2$-convergence of the sequence fails, is given in \cite[Section 7]{MasterThesis}. Finally, the analysis for the case when the conformal structures diverge is expected to be analogous to \cite{LaurainRiviereEnergyQuantization} and this paper.
\section{Notation and definitions}\label{sec:Notations}
The vector symbol on a function indicates that the function takes values in $\R^3$. For $A\subset X$, we write $\chi_A:X\to \R$ for the characteristic function on $A$. If $\mu$ is a measure on $X$, $A$ is $\mu$-measurable such that $0<\mu(A)<\infty$ and $f:X\to \R$ is measurable, we set
\[\dashint_A f\dif{\mu} \cqq \frac{1}{\mu(A)} \int_A f\dif{\mu}.\]
We denote for $f:\C\setminus \{0\} \to \R$
\[\partial_r f(x) \cqq \frac{\dd}{\dd t} \bigg \vert _{t=0} f\left (x + t \frac{x}{|x|}\right ), \quad \partial_{\theta} f(x) \cqq \frac{\dd}{\dd t} \bigg \vert _{t=0} f\left (x +t ix \right ),\]
which is equivalent to expressing $x = r e^{i\theta}$ in polar coordinates and taking derivatives with respect to this coordinate system. Notice that $|\partial_r f(x)|\leq |\nabla f(x)|$ and $\frac{1}{|x|}  |\partial _{\theta} f(x)| \leq |\nabla f(x)|$. If we do not specify the center of a ball $B_r$, we always mean the ball centered around the origin.

When working with vector valued maps in a flat chart $B_1\subset \C$, we will adopt the notation from \cite{RiviereLectureNotes}. That is, for maps $\vec e$, $\vec f:B_1 \to \R^3$ and $\lambda:B_1\to \R$, we denote for example\enlargethispage{0.5cm}
\begin{alignat*}{2}
\nabla \vec e &\cqq \vecto{\parpar{x_1}{\vec e}}{\parpar{x_2}{\vec e}},& \nabla ^\perp \vec e &\cqq \vecto{-\parpar{x_2}{\vec e}}{\parpar{x_1}{\vec e}},\\
\langle \vec e, \nabla \vec f\rangle&\cqq \vecto{\langle \vec e, \parpar{x_1}{\vec f}\rangle}{\langle \vec e, \parpar{x_2}{\vec f}\rangle},&  \langle \vec e, \nabla^\perp \vec f\rangle&\cqq \vecto{-\langle \vec e, \parpar{x_2}{\vec f}\rangle}{\langle \vec e, \parpar{x_1}{\vec f}\rangle}, \\
\vec e\times \nabla \vec f&\cqq \vecto{ \vec e\times \parpar{x_1}{\vec f}}{ \vec e\times \parpar{x_2}{\vec f}},\quad & \vec e\times \nabla ^\perp \vec f&\cqq \vecto{ -\vec e\times \parpar{x_2}{\vec f}}{ \vec e\times \parpar{x_1}{\vec f}},\\
\nabla \lambda \vec e &\cqq \vecto{\parpar{x_1}{\lambda}\vec e}{\parpar{x_2}{\lambda} \vec e},&\nabla \lambda \cdot \nabla \vec e &\cqq \parpar{x_1}{\lambda} \parpar{x_1}{\vec e} + \parpar{x_2}{\lambda} \parpar{x_2}{\vec e},\\
\langle \nabla \vec e, \nabla \vec f\rangle &\cqq \langle \parpar{x_1}{\vec e},\parpar{x_1}{\vec f}\rangle + \langle\parpar{x_2}{\vec e}, \parpar{x_2}{\vec f}\rangle ,\quad & \langle \nabla \vec e, \nabla ^\perp \vec f\rangle &\cqq -\langle \parpar{x_1}{\vec e},\parpar{x_2}{\vec f}\rangle + \langle\parpar{x_2}{\vec e}, \parpar{x_1}{\vec f}\rangle ,\\
\nabla \vec e \times \nabla \vec f &\cqq \parpar{x_1}{\vec e}\times \parpar{x_1}{\vec f} +\parpar{x_2}{\vec e}\times \parpar{x_2}{\vec f},\quad &  \nabla \vec e \times \nabla^\perp \vec f &\cqq -\parpar{x_1}{\vec e} \times \parpar{x_2}{\vec f} + \parpar{x_2}{\vec e} \times \parpar{x_1}{\vec f}. 
\end{alignat*}
If $\vec X = \vecto{\vec X_1}{\vec X_2}$, we denote
\[\Div \vec X \cqq \parpar{x_1}{\vec X_1} + \parpar{x_2}{\vec X_2}\quad\text{and}\quad\Delta \vec e \cqq \Div(\nabla \vec e) = \partial^2_{x_1}\vec e+ \partial^2_{x_2}\vec e.\]

\subsection{Constrained Willmore surfaces}\label{subsec:Constrained Willmore Surfaces}
We suppose that $\Sigma$ is a two-dimensional, connected, oriented, closed and smooth manifold and we let $\vec \Phi:\Sigma \to \R^3$ be a smooth immersion.
Recall from \eqref{intro:EL equation} that we defined constrained Willmore surfaces to be stationary points of the functional $\mc F$ defined in \eqref{functional F definition}. To make this more explicit, we consider the first variation of the terms involved:
\begin{align}
\delta \mc W(\vec \Phi)(\vec w) &= \int_{\Sigma} (\Delta_g H + 2H(H^2-K))\langle \vec n, \vec w\rangle \dif{\mu} ,\label{extrinsic Willmore gradient}\\
\delta \mc A(\vec \Phi)( \vec w) &= \int_{\Sigma} -2 H\langle\vec n , \vec w\rangle \dif{\mu} ,\label{extrinsic area gradient}\\
 \delta \mc V(\vec \Phi)(\vec w)&=  \int_{\Sigma}- \langle \vec n , \vec w\rangle \dif{\mu} , \label{extrinsic volume gradient}\\
\delta \mc M(\vec \Phi)(\vec w) &=  \int_{\Sigma} -K \langle \vec n , \vec w\rangle \dif{\mu} \label{extrinsic total mean curvature gradient}.
\end{align}
Here, $\Delta_g$ is the Laplace-Beltrami operator induced by the metric $g=g_{\vec \Phi}$ and $\vec w:\Sigma \to \R^3$ is smooth. In general, the Willmore gradient $\delta \mc W$ poses analytical challenges as it requires $H$ to be in $L^3(\Sigma,\mu)$, whereas the Willmore energy only gives $H\in L^2(\Sigma,\mu)$. In \cite{RiviereAnalysisAspects}, it was shown that in conformal coordinates, the Willmore gradient may be recast in divergence form. Conformal coordinates are charts in which the metric takes the form $g = e^{2\lambda} \id$ for some smooth function $\lambda$, called the \emph{conformal factor}. In these coordinates, the induced area measure is $\dif{\mu} = e^{2\lambda} \dif{x}$. If $\vec \Phi:B_1 \to \R^3$ is a conformal immersion, then \enlargethispage{1cm}
\begin{align}
\delta \mc W(\vec \Phi)(\vec w)  &=  \frac{1}{2} \int _{B_1} \left \langle \Div\left (2\nabla \vec H -3 H \nabla \vec n + \vec H \times \nabla^\perp \vec n\right ) , \vec w\right \rangle \dif{x}\label{First variation of Willmore 1} \\
&=-\frac{1}{2}\int _{B_1}\left \langle \Div\left (\nabla \vec H - 3 \pi_{\vec n} (\nabla \vec H) + \nabla ^\perp \vec n \times \vec H\right ) , \vec w \right \rangle\dif{x}. \label{First variation of Willmore 2}
\end{align}
for all $\vec w \in C_c^\infty(B_1, \R^3)$. Here, $\pi_{\vec n}$ denotes the orthogonal projection onto $\vec n$. \cite{BernardNoether} showed that the other functionals may also be recast in a divergence form as a consequence of Noether's theorem, namely\footnote{Note that since we work in our chart parametrization, we integrate with respect to the Lebesgue measure $\mc L^2$, not with respect to $\mu$ and these distributional variations have to be understood in this sense as well.}
\begin{align}
\delta \mc A &= -\Delta \vec \Phi, \label{First variation of Area}\\
\delta \mc V &= \frac{1}{2}\Div(\vec \Phi \times \nabla^\perp \vec \Phi),\label{First variation of Volume}\\
\delta \mc M &= - \frac{1}{2} \Div(\nabla \vec n  + 2H \nabla \vec \Phi).\label{First variation of Total Mean Curvature}
\end{align}
We define 
\begin{equation}
\ms L_{\vec n} (\vec w) = \Div(\nabla \vec w -3\pi_{\vec n}(\nabla \vec w) + \nabla^\perp \vec n \times \vec w).\label{L operator associated to Willmore variation}
\end{equation}
$\ms L_{\vec n}$ was originally introduced in \cite{RiviereAnalysisAspects}, where it was shown that this is a locally invertible, self-adjoint elliptic operator. This allows $\delta \mc W$ to be defined in a weak sense. We say that $\vec \Phi$ is a \emph{Willmore immersion} if it is a stationary point of $\mc W$, i.e., $\ms L_{\vec n} \vec H = 0$ in every conformal parametrization. It is called \emph{constrained Willmore immersion} if it is a stationary point of $\mc F$, i.e.,
\begin{equation}
-\frac{1}{2}\ms L_{\vec n}(\vec H) + \Div \left (-\alpha \nabla \vec \Phi + \frac{\beta}{2} \vec \Phi \times \nabla ^\perp \vec \Phi- \frac{\gamma}{2} (\nabla \vec n + 2H \nabla \vec \Phi) \right ) = 0.\label{constrained Willmore surface}
\end{equation}
We define 
\begin{equation}
\vec T \cqq -\alpha \nabla \vec \Phi + \frac{\beta}{2} \vec \Phi \times \nabla ^\perp \vec \Phi- \frac{\gamma}{2} (\nabla \vec n + 2H \nabla \vec \Phi).\label{definition T}
\end{equation}
The divergence term of \eqref{constrained Willmore surface} is the linear combination for the first variations of area, total mean curvature and (algebraic) volume. This is equivalent to the scalar equation in non-divergence form
\begin{equation}
\Delta_g H + 2H (H^2-K) - 2\alpha H - \beta - \gamma K =0.
\label{extrinsic constrained system}
\end{equation}
Notice that when $\tilde{\vec \Phi} = r \vec \Phi$ is a rescaled immersion, then $\tilde{\vec \Phi}$ is still a constrained Willmore immersion with coefficients
\begin{equation}
\tilde{\alpha} = r^{-2} \alpha, \quad \tilde{\beta} = r^{-3} \beta, \quad \tilde{\gamma} = r^{-1} \gamma.\label{rescalings of constants}
\end{equation}
\subsection{Lorentz spaces}\label{subsec:Lorentz spaces}
For a detailed introduction to Lorentz spaces, see \cite{Grafakos}. Suppose $\Omega\subset \R^n$ is measurable. For $1\leq p, q<\infty$ define the Lorentz space $L^{p,q}(\Omega)$ as the set of all measurable functions $f$ such that the quasi-norm
\[|f|_{L^{p,q}(\Omega)} \cqq p^{1/q}\left (\int_0^\infty t^q \mc L^n(\{x,\, |f(x)| > t\})^{\frac{q}{p}}  \frac{\dif{t}}{t} \right )^{\frac{1}{q}} \]
is finite. For $q=\infty$, we define
\[|f|_{L^{p,\infty}(\Omega)} \cqq \left (\sup_{t>0} (t^p \mc L^n(\{x,\,|f(x)|>t\}))\right )^{\frac{1}{p}}.\]
For $p>1$, $|\cdot|_{L^{p,q}(\Omega)}$ is equivalent to a norm $\|\cdot \|_{L^{p,q}(\Omega)}$ on $L^{p,q}(\Omega)$ with which this space becomes a Banach space, see \cite[Exercise 1.4.3, 1.1.12]{Grafakos}. We have the inclusions 
\[L^{p,1}(\Omega)\subset L^{p,q}(\Omega)\subset L^{p,q'}(\Omega)\subset L^{p,\infty}(\Omega)\]
for $1< q< q' < \infty$. Furthermore, $L^{p}(\Omega) = L^{p,p}(\Omega)$ for $1\leq p < \infty$. Finally, $L^{\frac{p}{p-1}, \frac{q}{q-1}}(\Omega)$ is the dual space of $L^{p,q}(\Omega)$ for $1<p<\infty$ and $1\leq q< \infty$\footnote{However, the dual space of $L^{p,\infty}(\Omega)$ is in general strictly larger than $L^{\frac{p}{p-1},1}(\Omega)$.} with the $L^2$-pairing. We will later use the duality between $L^{2,1}(\Omega)$ and $L^{2,\infty}(\Omega)$. In the special case $q=1$, $\|\cdot \|_{L^{p,1}(\Omega)}$ is given by
\begin{equation}
\|f\|_{L^{p,1}(\Omega)} = \frac{p^2}{p-1} \int_0^\infty \mc L^n(\{x,\,|f(x)|>t\})^{1/p}\dif{t},\label{Lp1 norm definition}
\end{equation}
see \cite{MichelatPointwiseExpansion}. A quick calculation in the case $n=2$ then shows that for $R>2r>0$
\begin{equation}
c \log(R/r) \leq \left \|\frac{1}{|z|}\right \|_{L^{2,1}(B_{R}\setminus B_r)} \leq C \log(R/r),\label{L21 norm for gradient of log}
\end{equation}
whereas
\begin{equation}
\left \|\frac{1}{|z|}\right \|_{L^{2,\infty}(\C)}\leq C.\label{L2infty norm for gradient of log}
\end{equation}
Suppose that $\Omega\subset \R^n$ is open and bounded and $p>2$. Then
\[\mc L^n(\{x,\,|f(x)|>t\})^{1/2} \leq t^{-p/2} \|f\|_{L^p}^{p/2}\]
and so for any $s>0$
\[\|f\|_{L^{2,1}(\Omega)}\leq 4 \int_0^s \mc L^n(\Omega)^{1/2} \dif{t} + 4\|f\|_{L^p(\Omega)}^{p/2}\int_s^\infty t^{-p/2} \dif{t} .\]
Minimizing the right-hand side over $s$ yields $s=\|f\|_{L^p(\Omega)} \mc L^n(\Omega)^{-1/p}$ and thus
\begin{equation}
\|f\|_{L^{2,1}(\Omega)} \leq 4\frac{p}{p-2} \|f\|_{L^p(\Omega)} \mc L^n(\Omega)^{\frac{1}{2}-\frac{1}{p}}. \label{Lp L21 inequality}
\end{equation} 
\section{\texorpdfstring{$\eps$}{epsilon}-regularity for constrained Willmore surfaces}\label{sec:eps regularity}
In this section, we want to establish the $\eps$-regularity result \Cref{thm:eps regularity}. Before proving it, we make the following remark:
\begin{remark}\label{rem:various things} \hfill
\begin{enumerate}[label = \arabic*)]
\item \label{rem:typical argument harmonic functions} The following argument is frequently used for applying estimates with prescribed zero boundary condition: Suppose that $\Delta f = g $ in $B_1$. If $f$ is locally integrable in $B_1$, then using the mean value theorem, we choose $r\in (3/5, 4/5)$ such that $\|f\|_{L^1(\partial B_r)} \leq C \|f\|_{L^1(B_{4/5})}$. Denote by $\nu$ the solution to
\begin{equation}
\begin{cases}
\Delta \nu = 0 & \text{ in }B_r,\\
\nu = f&\text{ on }\partial B_r.
\end{cases}\label{harmonic solution nu}
\end{equation} 
Using the Poisson representation formula for the solution of \eqref{harmonic solution nu}, we obtain that 
\begin{equation}
\|\nu\|_{C^k(B_{1/2})}\leq C_k \|f\|_{L^1(\partial B_r)} \leq C_k \|f\|_{L^1(B_{4/5})}.\label{harmonic part is good inside half ball}
\end{equation}
Then we apply suitable estimates with zero boundary conditions, e.g. from \cite{Helein}, to $f - \nu$ in the ball $B_r$. 
\item Suppose that $\vec \Phi: \Omega\to \R^3$ is a conformal immersion. Then it follows by Jensen's inequality that
\begin{equation}
\mc A(\vec \Phi) = \int _{\Omega} e^{2\lambda} \dif{x} \geq |\Omega| e^{2\overline{\lambda}}, \label{Jensen application}
\end{equation}
where $\overline{\lambda} = \dashint_{\Omega} \lambda \dif{x}$ is the average conformal factor. \label{rem:average conformal factor bounded by area} 
\item \label{rem:bounded oscillation for conformal factor} Suppose $\vec \Phi:B_1\to \R^3$ is a conformal immersion with conformal factor $\lambda$. Suppose further that $\vec \Phi$ satisfies
\[\int _{B_1} |\nabla \vec n_{\vec \Phi} |^2 \dif{x} < \frac{8\pi}{3}.\]
	Doing the same arguments as in \cite[Theorem 5.5]{RiviereLectureNotes} and replacing the bound $\|\nabla \lambda\|_{L^{2,\infty}(B_1)}\leq C$ by a bound on $\|\nabla \lambda\|_{L^1(B_1)}$, one can show that for $r\in (0,1)$ and $\overline{\lambda} = \dashint_{B_1} \lambda \dif{x}$
\begin{equation}
\|\lambda - \overline{\lambda }\|_{L^\infty(B_{r})}\leq  C(r, \|\nabla \lambda\|_{L^1(B_1)}),\label{lambda l infinity estimate up to difference}
\end{equation}
where $C(r, \|\nabla \lambda\|_{L^1(B_1)})$ is a constant only depending on $\|\nabla \lambda\|_{L^1(B_1)}$ and $r$.
\item \label{rem: eps regularity} Notice that due to \eqref{rescalings of constants}, all the quantities appearing in \Cref{thm:eps regularity} are invariant under rescalings of $\vec \Phi$. 
\end{enumerate}
\end{remark}

\begin{proof}[Proof of \Cref{thm:eps regularity}]
We prove the claims for some choice of $r$, the general result follows by the same arguments, just with different choices of radii throughout. By \cite[Theorem 4.3]{MondinoScharrer}, it is already known that $\vec \Phi$ is a smooth immersion. We will follow \cite[Section III.2]{RiviereAnalysisAspects} and mostly focus on the differences.  Thanks to \Cref{rem:various things} \ref{rem: eps regularity}, we may rescale $\vec \Phi$ by the factor $e^{-\overline{\lambda}}$ to assume $\overline{\lambda} = 0$. Working under a $\frac{8\pi}{3}$ threshold for the energy allows to apply \Cref{rem:various things} \ref{rem:bounded oscillation for conformal factor} and by \eqref{lambda l infinity estimate up to difference}, it holds $\|\lambda\|_{L^\infty(B_{1/2})}\leq C$. Here and in the following, all constants $C$ can depend on $\|\nabla \lambda\|_{L^1(B_1)}$, $\alpha$, and $\gamma$. The dependence on $\beta$ however will be tracked explicitly for now.
Let $\omega\in C_c^\infty(B_{1/2}, \R)$ such that $\omega \vert_{B_{1/4}} = 1$. Then it holds that
\begin{equation}
\ms L_{\vec n} (\omega \vec H) = \vec g_1 + \vec g_2 +\vec g_3 + \vec g_4,\label{decomposition in g1 to g4}
\end{equation}
where $\vec g_1$ to $\vec g_4$ are defined as 
\begin{align}
\vec g_1 &\cqq 2 \Div( \vec H\nabla \omega) - \vec H \Delta \omega - 6 \Div(\vec H \nabla \omega) +3\vec H\Delta \omega ,\\
\vec g_2 &\cqq 3 H \nabla \vec n \cdot \nabla \omega - \vec H \times \nabla^\perp \vec n \cdot \nabla \omega,\\
\vec g_3 &\cqq 2\omega \Div (-\alpha \nabla \vec \Phi - \frac{\gamma}{2} (\nabla \vec n + 2H \nabla \vec \Phi)) =- \omega (\gamma \Div (\nabla \vec n + 2H \nabla \vec \Phi) +4\alpha e ^{2\lambda} \vec H) \label{g3 definition},\\
\vec g_4 &\cqq \beta \omega \Div(\vec \Phi \times \nabla ^\perp \vec \Phi) = -2\beta e^{2\lambda} \omega\vec n.\label{g4 definition}
\end{align}
$\vec g_1$ and $\vec g_2$ were already defined in \cite[(III.7), (III.8)]{RiviereAnalysisAspects} up to a sign mistake.
%
Here, we used \eqref{constrained Willmore surface}. By \cite[(III.9), (III.10)]{RiviereAnalysisAspects}, the uniform bound for $\lambda$ allows us to conclude
\[\|\vec g_1\|_{H^{-1}(B_{1/2})}^2 +\|\vec g_2 \|_{L^1(B_{1/2})}\leq C \int _{B_{1/2}\setminus B_{1/4} }|\nabla \vec n|^2 \dif{x}.\]
We use here that $|\vec H| \leq e^{-\lambda} |\nabla \vec n|$. Similarly, $\vec g_3$ and $\vec g_4$ satisfy
\begin{equation}
\|\vec g_3\|_{H^{-1}(B_{1/2})}^2 \leq C \int _{B_{1/2}} |\nabla \vec n|^2 \dif{x}, \quad \|\vec g_4 \|_{L^\infty(B_{1/2})} \leq C |\beta|.
\end{equation}
Let us first prove \eqref{c3 estimate}. We will choose $\eps_1$ later. Notice that by the Poincar\'e inequality, it holds for $\overline{\vec n} \cqq \dashint _{B_{2^{-3}}} \vec n \dif{x}$
\[\|\vec n - \overline{\vec n}\|_{L^2(B_{2^{-3}})} \leq C \|\nabla \vec n\|_{L^2(B_{2^{-3}})}.\]
We let $\tilde{\omega} \in C_c^\infty(B_1, \R)$ such that $\chi_{B_{2^{-3}}}\leq \tilde{\omega}\leq \chi_{B_{\rho}}$, where $2^{-3} < \rho = \rho(\|\nabla \lambda\|_{L^1(B_1)}) \leq 1/4$ will be chosen later. We obtain, using the self-adjointness of $\ms L_{\vec n}$ proved in \cite{RiviereAnalysisAspects},
\begin{align}
\left |\int _{B_1} \langle \tilde{\omega}  \overline{\vec n} , \ms L_{\vec n}(\omega \vec H)\rangle \dif{x} \right |&=\left | \int _{B_{1/2}} \langle \omega \vec H , \ms L_{\vec n}(\tilde{\omega} \overline{\vec n})\rangle \dif{x} \right |\notag\\
&= \left |\int _{B_{1/2}} \langle\omega \vec H , \Delta \tilde{\omega}\overline{\vec n}-3 \Div(\vec n  \langle \vec n , \overline{\vec n}\rangle \nabla \tilde{\omega} ) + \nabla^\perp \vec n \times  \overline{\vec n} \nabla \tilde{\omega} \rangle \dif{x}\right |\notag\\
&\leq C\int _{B_{1/2}}  |\vec H|(1 + |\nabla \vec n|)\dif{x} \leq C \|\nabla \vec n\|_{L^2(B_{1/2})}.\label{first estimate for estimate of integral tested by average of n}
\end{align} 
On the other hand, \eqref{decomposition in g1 to g4} together with the bounds on $\vec g_1$, $\vec g_2$ and $\vec g_3$ yield
\begin{align}
\left |\int _{B_{1/2}} \langle\tilde{\omega}  \overline{\vec n} , \ms L_{\vec n}(\omega \vec H)\rangle \dif{x} \right |&\geq -C\|\nabla \vec n\|_{L^2(B_{1/2})} + 2|\beta| \left |\int_{B_{1/2}}  e^{2\lambda} \omega \tilde{\omega} \langle \vec n , \overline{\vec n}\rangle\dif{x}\right |. \label{lower bound for the testing with average n}
\end{align}
We estimate
\begin{align}
\int _{B_{1/2}} e^{2\lambda} \omega \tilde{\omega} \langle \vec n , \overline{\vec n}\rangle \dif{x}& = \int _{B_{2^{-3}}} e^{2 \lambda} \dif{x} + \int_{B_{2^{-3}}} e^{2\lambda} \langle \vec n , \overline{\vec n} - \vec n\rangle \dif{x}  + \int _{B_{\rho}\setminus B_{2^{-3}}} e^{2\lambda} \tilde{\omega} \langle \vec n , \overline{\vec n}\rangle \dif{x} \notag\\
&\geq C e^{-2\|\lambda\|_{L^\infty}} - C e^{2\|\lambda\|_{L^\infty}} \sqrt{\eps_1} - C(\rho^2 - 2^{-6}) e^{2\|\lambda\|_{L^\infty}}.\label{estimate integral tested by average of n}
\end{align}
Using \eqref{lambda l infinity estimate up to difference}, we may first choose $\eps_1$ small and then $\rho$ close to $2^{-3}$ such that the right-hand side of \eqref{estimate integral tested by average of n} can be bounded from below by $C=C(\|\nabla \lambda\|_{L^1})>0$. This is the point where we use explicitly that $\eps_1$ depends on $\|\nabla \lambda\|_{L^1}$. Together with \eqref{first estimate for estimate of integral tested by average of n} and \eqref{lower bound for the testing with average n}, \eqref{estimate integral tested by average of n} implies 
\[|\beta| \leq C \|\nabla \vec n\|_{L^2(B_{1})},\]
which is \eqref{c3 estimate}.

\medskip

Let us now prove \eqref{eps regularity 1} and \eqref{eps regularity 2}, for which we only assume the bound \eqref{small dirichlet energy} to hold. Let $\vec v_1$ to $\vec v_4$ be defined through \cite[Lemma A.1, Lemma A.3]{RiviereAnalysisAspects} as the unique solutions of 
\begin{equation}
\begin{cases}
\ms L_{\vec n} \vec v_i =  \vec g_i& \text{ in } B_{1/2},\\
\vec v_i =0 &\text{ on } \partial B_{1/2},
\end{cases} \quad\text{for $i=1,\ldots, 4$.}
\end{equation}
\cite[Lemma A.1, Lemma A.3]{RiviereAnalysisAspects} yields the bounds
\begin{equation}
\|\nabla \vec v_1\|_{L^{2}(B_{1/2})}+  \|\nabla \vec v_2\|_{L^{2,\infty}(B_{1/2})}\leq   C \|\nabla \vec n\|_{L^2(B_{1})},\label{v1 and v2 bounds}
\end{equation}
\begin{equation}
\|\nabla \vec v_3\|_{L^2(B_{1/2})} \leq C\|\nabla \vec n\|_{L^2(B_{1})}, \quad \|\nabla \vec v_4\|_{L^2(B_{1/2})} \leq C |\beta|. \label{v3 and v4 bounds}
\end{equation}
By \cite[Lemma A.1]{RiviereAnalysisAspects}
Then $\vec v \cqq \omega \vec H - \vec v_1 - \vec v_2 -\vec v_3-\vec v_4\in L^2(B_{1/2})$ satisfies $\ms L_{\vec n}(\vec v) =0$ and $\vec v \vert _{\partial B_{1/2}} = 0$. Using \cite[Lemma A.8]{RiviereAnalysisAspects}, we obtain $\vec v=0$. 
We obtain from \eqref{v1 and v2 bounds} and \eqref{v3 and v4 bounds} that
\begin{equation}
\|\nabla (\omega \vec H)\|_{L^{2,\infty}(B_{1/2})} \leq  C \|\nabla \vec n\|_{L^2(B_{1})}+C|\beta|.\label{eps reg: H L2 weak estimate}
\end{equation}
We do the Hodge decomposition from \cite[Proposition 3.3.9]{Helein}
\begin{equation}
\nabla \vec H - 3\pi_{\vec n} (\nabla \vec H) = \nabla \vec C + \nabla ^\perp \vec D + \nabla \vec r\label{eps reg: Hodge decomposition}
\end{equation}
on $B_{1/4}$, where $\vec r$ is harmonic and $\vec C$ and $\vec D$ satisfy 
\begin{equation}
\begin{cases}
\Delta \vec C = \Div(\vec H \times \nabla ^\perp \vec n ) + \vec g_3 + \vec g_4  &\text{ on }B_{1/4},\\
\vec C=0 &\text{ on }\partial B_{1/4},
\end{cases}\label{C-equation}
\end{equation}
and
\begin{equation}
\begin{cases}
\Delta \vec D = -3\Div(\pi_{\vec n}(\nabla ^\perp \vec H)) &\text{ on }B_{1/4},\\
\vec D=0 &\text{ on }\partial B_{1/4}.
\end{cases}\label{D-equation}
\end{equation}
Using again the Wente inequality \cite[Theorem 3.4.5]{Helein} for \eqref{D-equation} and splitting the right-hand side of \eqref{C-equation}, we obtain the estimates
\begin{align}
\|\nabla \vec C\|_{L^2(B_{1/4})}+\|\nabla \vec D\|_{L^2(B_{1/4})} &\leq C(\|\nabla \vec n\|_{L^2(B_{1})}\|\nabla \vec H\|_{L^{2,\infty}(B_{1/4})} +  \|\nabla \vec n\|_{L^2(B_{1})}+ |\beta|)\notag\\
&\leq C(\|\nabla \vec n\|_{L^2(B_{1})} +|\beta|).\label{eps reg: L2 estimate C and D}
\end{align}
Together with the bound for the harmonic term $\nabla \vec r$ \cite[(III.18), (III.19)]{RiviereAnalysisAspects}, \eqref{eps reg: H L2 weak estimate}, \eqref{eps reg: Hodge decomposition}, and \eqref{eps reg: L2 estimate C and D}, we obtain 
\[\|\nabla \vec H\| _{L^2(B_{2^{-3}})}\leq C(\|\nabla \vec n\|_{L^2(B_{1})} +|\beta|).\]
By the same argument as in \cite{RiviereAnalysisAspects}, we obtain by the means of another Hodge decomposition the estimate
\[\|\nabla \vec H\| _{L^{2,1}(B_{2^{-4}})}\leq C(\|\nabla \vec n\|_{L^2(B_{1})} + |\beta|).\]
From this, the uniform $L^2$ bound of $\vec H$, and \cite[Theorem 3.3.4]{Helein}, we obtain the uniform bound $\|\vec H\|_{L^\infty(B_{2^{-4}})}\leq C(\|\nabla \vec n\|_{L^2(B_{1})} + |\beta|)$. Differentiating $\Delta \vec \Phi  = e^{2\lambda} \vec H$ yields
\begin{equation}
\Delta \nabla \vec \Phi =\nabla \Delta \vec \Phi = 2\nabla \lambda e^{2\lambda} \vec H + e^{2\lambda} \nabla \vec H.\label{Delta Phi equation}
\end{equation}
By \cite[(5.21)]{RiviereLectureNotes}, there are $\vec e_1, \vec e_2\in W^{1,2}(B_1)$ with 
\begin{equation}
\Delta \lambda = \langle \nabla^\perp \vec e_1, \nabla \vec e_2\rangle\label{Poisson equation for lambda}
\end{equation}
satisfying $\|\nabla \vec e_1\|_{L^{2}(B_1)}+\|\nabla \vec e_1\|_{L^{2}(B_1)}\leq C \|\nabla \vec n\|_{L^2(B_{1})}$. Together with Wente's inequality \cite[Theorem 3.4.1]{Helein} and also \Cref{rem:various things} \ref{rem:typical argument harmonic functions}, \eqref{Poisson equation for lambda} gives
 \[\|\nabla \lambda\|_{L^{2,1}(B_{1/2})}\leq C \|\nabla \vec n\|_{L^2(B_1)} + C\]
 holds. We obtain that the right-hand side of \eqref{Delta Phi equation} is bounded in $L^{2,1}$, namely
\begin{equation}
\|2\nabla \lambda e^{2\lambda} \vec H + e^{2\lambda} \nabla \vec H\|_{L^{2,1}(B_{2^{-4}})} \leq C(\|\nabla \vec n \|_{L^2(B_1)} +|\beta|).\label{L21 estimates right hand side}
\end{equation}
Given $\Omega \subset \C$ open, bounded and with smooth boundary, $p\in (1,\infty)$, and $g\in L^p(\Omega)$, the solution $f$ to
\begin{equation}
\begin{cases}
\Delta f = g &\quad\text{in $\Omega$},\\
\;\;\; f = 0 &\quad\text{ on $\partial \Omega$},
\end{cases}
\end{equation}
satisfies $\|\nabla^2 f \|_{L^p(\Omega)}\leq C(p,\Omega) \|g\|_{L^p(\Omega)}$ by \cite[Theorem 3.29, 3.31]{AmbrosioLectureNotes}. By interpolation, see \cite[Theorem 3.3.3]{Helein}, we have $\|\nabla ^2 f\|_{L^{2,1}(\Omega)} \leq C(\Omega) \|g\|_{L^{2,1}(\Omega)}$. Applying this to \eqref{Delta Phi equation} and \eqref{L21 estimates right hand side}, we obtain with \Cref{rem:various things} \ref{rem:typical argument harmonic functions}
\begin{equation}
\|\nabla^3 \vec \Phi\|_{L^{2,1}(B_{2^{-5}})} \leq C(\|\nabla \vec n \|_{L^2(B_1)} +|\beta|).\label{estimate third derivatives Phi}
\end{equation}
It follows again from \cite[Theorem 3.3.4]{Helein} that $\nabla ^2 \vec \Phi \in L^{\infty}(B_{2^{-5}})$ with 
\begin{equation}
\|\nabla^2 \vec \Phi\|_{L^{\infty}(B_{2^{-5}})} \leq C(\|\nabla \vec n \|_{L^2(B_1)} +|\beta|).\label{L infinity bound for nabla 2 Phi}
\end{equation}
Since $|\nabla \vec n| \leq C |\nabla ^2 \vec \Phi|$, we obtain the same bound for $\|\nabla \vec n \|_{L^\infty(B_{2^{-5}})}$. From here, we can bootstrap the equations via the system \eqref{conservative system} to obtain \eqref{eps regularity 1} and \eqref{eps regularity 2}. Together with \eqref{c3 estimate}, we obtain \eqref{eps regularity 3}, which finishes the proof.
\end{proof}
\begin{remark}
The reason why we need to treat $\beta$ different from the other Lagrange multipliers is that $\delta \mc V = -\vec n$ does not depend on the curvature, whereas $\delta \mc M = -K\vec n$ and $\delta \mc A= -2\vec H$ do, and so an assumption like \eqref{small dirichlet energy} does not a priori give bounds for this term. But by solving \eqref{constrained Willmore surface} for $\beta$, we can express $\beta$ by terms that depend on the curvature and this argument allows us to show bounds for $\beta$.
\end{remark}
\section{Energy quantization for constrained Willmore surfaces}\label{sec:Energy quantization}
Our goal in the upcoming sections will be to prove \Cref{thm:Energy quantization}. For this, we will follow the strategy developed by \textsc{BernardRiviere} and \textsc{Rivi\`ere} \cite{BernardRiviere}, where they obtained an energy quantization in the case of unconstrained Willmore surfaces. The strategy is the following:
\begin{enumerate}[label = \arabic*)]
\item We consider a sequence $\vec \Phi_k:\Sigma \to \R^3$ of constrained Willmore surfaces with bounded Dirichlet energy $\mc E(\vec \Phi_k) \leq \Lambda$ such that the conformal classes induced by $\vec \Phi_k$ remain bounded in the moduli space. Upon reparametrization (see \cite[Corollary 4.4]{RiviereLectureNotes}), we may assume that $\vec \Phi_k$ are conformal with respect to the constant curvature metric $h_k$ in the same conformal class as $g_{\vec \Phi_k}$. Then $h_k$ converge smoothly to some limiting metric $h_\infty$ on $\Sigma$. Because of the $\eps$-regularity \Cref{thm:eps regularity}, the immersions are bounded in $C^l$ away from finitely many energy concentration points. In particular, up to a subsequence, they converge smoothly away from these points to some limiting immersion. We let $\eps_0>0$. We want to identify the bubble and neck domains around the $a_i$. Here, the bubbles are the domains where energy accumulates, which can be parametrized in $h_k$-conformal, converging coordinates in the form 
\[B(i,j,\alpha, k) \cqq B_{\alpha^{-1}\rho^{i,j}_k}(x^{i,j}_k) \setminus \bigcup_{j'\in K^{i,j}} B_{\alpha \rho^{i,j}_k}(x^{i,j'}_k),\]
see \eqref{bubble domains}. Here, $K^{i,j}$ is an index list of bubbles that are directly contained in the main bubble and $x^{i,j'}_k \in B_{2\rho^{i,j}_k}(x^{i,j}_k)$ are the centers of the interior bubbles. After rescaling, first letting $k$ tend to $\infty$ and then letting $\alpha$ tend to 0 allows us to define a limit immersion on this bubble. The necks are the regions connecting two adjacent bubbles and can be parametrized over diverging annuli of the form
\begin{equation}
B_{\alpha \rho^{i,j}_k}(x^{i,j'}_k) \setminus  B_{\alpha^{-1} \rho^{i,j'}_k}(x^{i,j'}_k),\label{neck region in proof strategy}
\end{equation}
where $\alpha \in (0,1)$, $j'\in K^{i,j}$ and $\frac{\rho^{i,j}_k}{\rho^{i,j'}_k}\to \infty$ as $k\to \infty$, see also \eqref{eq:neck annuli}. These neck regions satisfy 
\[\|d\vec n_{\vec \Phi_k}\|^2_{L^2(B_{2\rho}(x^{i,j'}_k)\setminus B_\rho(x^{i,j'}_k))} < \eps_0\]
for all $r \in (\rho^{i,j'}_k, \rho^{i,j}_k/2)$, see \eqref{doubling annulus condition}. This decomposition into bubble and neck regions was originally done in \cite[Proposition III.1]{BernardRiviere}. We recall and slightly adapt it in \Cref{prop:Bubble neck decomposition}.
\item Fix a neck region of the form \eqref{neck region in proof strategy} and consider the shifted immersion $\vec \xi_k(x) \cqq \vec \Phi_k\left (x^{i,j'}_k + x{\rho^{i,j}_k}\right )$. Let $r_k = \rho^{i,j'}_k/\rho^{i,j}_k$ and denote by $\lambda_k$ the conformal factor of $\vec \xi_k$. The goal is to show that the neck regions don't accumulate energy, namely
\begin{equation}
\lim_{\alpha \to 0} \lim_{k\to \infty}\int_{B_{\alpha }  \setminus  B_{\alpha^{-1} r_k} } |\nabla \vec n_{\vec \xi_k}|^{2}\dif{x} =0,\label{idea: nabla n L2 bound}
\end{equation}
see \eqref{no neck energy condition}. In \cite[Lemma V.1]{BernardRiviere}, it is shown that the Gaussian curvature doesn't accumulate in neck regions, namely (which holds true in general, not just for Willmore surfaces)
\[\lim_{\alpha \to 0} \lim_{k\to \infty} \int_{B_{\alpha }  \setminus  B_{\alpha^{-1} r_k} } K_{\vec \xi_k} e^{2\lambda_k}\dif{x}= 0.\]
Using the pointwise equality
\[|\nabla \vec n_{\vec \xi_k}|^2 = e^{2\lambda_k}(4 H_{\vec \xi_k}^2 - 2K_{\vec \xi_k}),\]
it suffices to show that the Willmore energy vanishes in neck regions, i.e.,
\begin{equation}
\lim_{\alpha \to 0} \lim_{k\to \infty} \int_{B_{\alpha }  \setminus  B_{\alpha^{-1} r_k} } H^2_{\vec \xi_k} e^{2\lambda_k} \dif{x}= 0.\label{idea:H L2 norm}
\end{equation}
\item We will show, up to a potential reflection of the neck region, that
\begin{equation}
\lim_{\alpha \to 0} \lim_{k\to \infty} \left \| e^{\lambda_k} \vec H _{\vec \xi_k}  \right \|_{L^{2,1}( B_{\alpha }  \setminus  B_{\alpha^{-1} r_k} )} \leq C.\label{idea L21 bound}
\end{equation}
This is the content of \Cref{lem: L21 control mean curvature} and the main ingredient in this paper. We will also show  
\begin{equation}
\lim_{\alpha \to 0} \lim_{k\to \infty} \|e^{\lambda_k}H_{\vec \xi_k}\|_{L^{2,\infty}(B_{\alpha }  \setminus  B_{\alpha^{-1} r_k} )} = 0, \label{idea: L2infty bound}
\end{equation}
which will be done in \Cref{lem: L2 weak estimate for gauss map annulus has small energy condition}. Recall the duality between $L^{2,1}$ and $L^{2,\infty}$, namely
\[\int f g \dif{x} \leq C \|f\|_{L^{2,1}} \|g\|_{L^{2,\infty}}.\]
The bounds \eqref{idea L21 bound} and \eqref{idea: L2infty bound} then prove the desired energy estimate \eqref{idea:H L2 norm} and \eqref{idea: nabla n L2 bound}, see \eqref{final estimate for L2 bound for H}.
\end{enumerate}
\begin{figure}[H] 
\centering

\tikzset{every picture/.style={line width=0.75pt}} 

\begin{tikzpicture}[x=0.75pt,y=0.75pt,yscale=-1,xscale=1]

\draw  [draw opacity=0][fill={rgb, 255:red, 255; green, 219; blue, 219 }  ,fill opacity=1 ] (208.39,127.41) .. controls (212.55,114.69) and (214.08,111.99) .. (217.47,107.97) .. controls (212.79,139.05) and (214.32,163.28) .. (217.47,183.52) .. controls (211.73,173.16) and (210.79,171.28) .. (208.39,163.19) .. controls (205.49,150.46) and (205.73,139.75) .. (208.39,127.41) -- cycle ;
\draw  [draw opacity=0] (208.39,127.41) .. controls (216.27,97.92) and (243.17,76.2) .. (275.15,76.2) .. controls (313.31,76.2) and (344.25,107.14) .. (344.25,145.3) .. controls (344.25,183.46) and (313.31,214.4) .. (275.15,214.4) .. controls (243.17,214.4) and (216.27,192.68) .. (208.39,163.19) -- (275.15,145.3) -- cycle ; \draw   (208.39,127.41) .. controls (216.27,97.92) and (243.17,76.2) .. (275.15,76.2) .. controls (313.31,76.2) and (344.25,107.14) .. (344.25,145.3) .. controls (344.25,183.46) and (313.31,214.4) .. (275.15,214.4) .. controls (243.17,214.4) and (216.27,192.68) .. (208.39,163.19) ;  
\draw    (208.39,163.19) .. controls (204.75,143.78) and (193.75,143.95) .. (189.06,163.19) ;
\draw  [draw opacity=0][fill={rgb, 255:red, 255; green, 219; blue, 219 }  ,fill opacity=1 ] (180.06,107.29) .. controls (183.2,112.75) and (187.1,119.15) .. (189.06,127.41) .. controls (185.4,138.3) and (184.4,150.9) .. (189.06,163.19) .. controls (186.9,170.85) and (184.4,177.55) .. (180.06,182.84) .. controls (168.2,154.9) and (174.8,123.9) .. (180.06,107.29) -- cycle ;
\draw  [dash pattern={on 0.84pt off 2.51pt}]  (189.06,163.19) .. controls (184.7,152.46) and (185.33,137.84) .. (189.06,127.41) ;
\draw  [dash pattern={on 0.84pt off 2.51pt}]  (180.06,183.53) .. controls (170.85,160.87) and (172.17,129.99) .. (180.06,107.97) ;
\draw  [draw opacity=0] (189.06,163.19) .. controls (181.18,192.68) and (154.28,214.4) .. (122.3,214.4) .. controls (84.14,214.4) and (53.2,183.46) .. (53.2,145.3) .. controls (53.2,107.14) and (84.14,76.2) .. (122.3,76.2) .. controls (154.28,76.2) and (181.18,97.92) .. (189.06,127.41) -- (122.3,145.3) -- cycle ; \draw   (189.06,163.19) .. controls (181.18,192.68) and (154.28,214.4) .. (122.3,214.4) .. controls (84.14,214.4) and (53.2,183.46) .. (53.2,145.3) .. controls (53.2,107.14) and (84.14,76.2) .. (122.3,76.2) .. controls (154.28,76.2) and (181.18,97.92) .. (189.06,127.41) ;  
\draw    (189.06,127.41) .. controls (192.7,146.82) and (203.7,146.65) .. (208.39,127.41) ;
\draw  [dash pattern={on 0.84pt off 2.51pt}]  (208.39,127.41) .. controls (205.53,138.65) and (205.7,150.48) .. (208.39,163.19) ;
\draw  [dash pattern={on 0.84pt off 2.51pt}]  (217.47,107.97) .. controls (213.53,130.65) and (213.53,157.65) .. (217.47,183.53) ;
\node at (53,144.58) [circle,fill,inner sep=0.7pt]{};

\draw  [dashed]  (182,159.83) -- (182,260.83) ;
\draw  [dashed]  (211.33,159.17) -- (211.33,226.17) ;

\draw (53,144.58) node [anchor=south east][inner sep=0.75pt]   [align=left] {$\vec \Phi_k( x_{k}^{i,1})\;$};
\draw (182,260.83) node [anchor=north west][inner sep=0.75pt]    {$\vec \Phi_k(B_{\alpha \rho^{i,2}_k}(x_{k}^{i,1})\setminus B_{\alpha^{-1} \rho^{i,1}_k}(x_{k}^{i,1}))$};
\draw (213.33,226.17) node [anchor=north west][inner sep=0.75pt]    {$\vec \Phi_k(B_{\alpha }(x_{k}^{i,1})\setminus B_{\alpha^{-1} \rho^{i,2}_k}(x_{k}^{i,1}))$};

\end{tikzpicture}
\caption{An example of the images of neck and bubble regions. The left sphere $\Phi_k(B_{\alpha^{-1} \rho^{i,1}_k}(x_{k}^{i,1}))$ and the catenoidal bridge $\Phi_k(B_{\alpha ^{-1} \rho^{i,2}_k}(x_{k}^{i,1})\setminus B_{\alpha \rho^{i,2}_k}(x_{k}^{i,1}))$ are the two resulting bubbles, connected by two neck regions (red) to the main immersion, which immerses the sphere on the right.}
\label{fig:bubble and neck regions}
\end{figure}
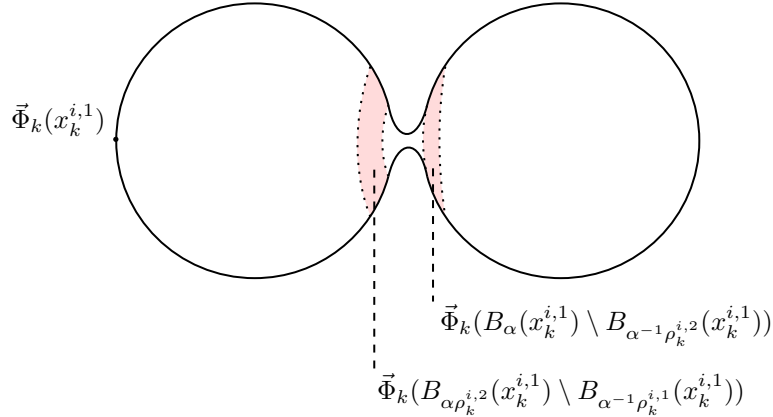

\subsection{Conservative system for constrained equation} \label{subsec:conservative system for constrained equation}
We will recall the conservative system corresponding to the constrained Willmore equation \eqref{constrained Willmore surface} as it was shown in \cite{BernardNoether}. We will again work in a conformal chart $\vec \Phi :\Omega\to \R^3$, where $\Omega$ is simply connected. Recall that $\vec T$ was defined in \eqref{definition T} as
\[\vec T =-\alpha \nabla \vec \Phi + \frac{\beta}{2} \vec \Phi \times \nabla ^\perp \vec \Phi- \frac{\gamma}{2} (\nabla \vec n + 2H \nabla \vec \Phi) .\]
We define $\vec X$ such that
\begin{equation}
\vec X = \frac{\beta}{4} |\vec \Phi|^2 \nabla ^\perp \vec \Phi,\label{definition X}
\end{equation}
and $Y$ being the solution of\footnote{The boundary conditions are specified for uniqueness of $Y$.}
\begin{equation}
\begin{cases}
\Delta Y =  e^{2\lambda} (-2\alpha  + \beta \langle \vec \Phi, \vec n\rangle-\gamma H) &\quad\text{in }\Omega,\\
\;\;\;Y = 0&\quad\text{on }\partial \Omega.
\end{cases}
\label{definition Y}
\end{equation}
By \cite{BernardNoether}, there exist quantities $\vec L, \vec R$ and $S$ satisfying
\begin{equation}
\begin{cases}
\nabla ^\perp \vec L &= \frac{1}{2} (\nabla \vec H - 3\pi_{\vec n}(\nabla \vec H) + \nabla ^\perp \vec n \times \vec H) -  \vec T,\\
\nabla ^\perp \vec R &= \vec L \times \nabla ^\perp \vec \Phi + \vec H \times \nabla \vec \Phi - \vec X,\\
\nabla ^\perp S &= \langle\vec L, \nabla ^\perp \vec \Phi\rangle - \nabla Y.
\end{cases}\label{definition L R and S}
\end{equation}
$S$ and $\vec R$ furthermore satisfy the equations
\begin{equation}
\begin{cases}

\nabla S &= - \langle \nabla ^\perp \vec R, \vec n \rangle + \nabla ^\perp Y,\\
\nabla \vec R &= \vec n \times \nabla ^\perp \vec R + (\nabla ^\perp S + \nabla Y )  \vec n.
\end{cases}\label{alternative system for S and R}
\end{equation}
From \eqref{alternative system for S and R} and \eqref{definition L R and S}, one can obtain
\begin{equation}
\begin{cases}
\Delta \vec R &=   \nabla ^\perp \vec n \cdot \nabla S + \nabla ^\perp \vec n \times \nabla \vec R + \Div \left [ \nabla Y \vec n \right ],\\
\Delta S &= \langle \nabla ^\perp \vec n , \nabla \vec R \rangle,\\
\Delta \vec \Phi &= -  \nabla ^\perp S\cdot \nabla \vec \Phi  - \nabla ^\perp \vec R \times \nabla \vec \Phi -  \nabla \vec \Phi \cdot \nabla Y  - \frac{\beta}{2}|\vec \Phi|^2 e^{2\lambda} \vec n.
\end{cases}
\label{conservative system}
\end{equation}
For the sake of completeness, we will rederive these conservation laws in \Cref{sec:Appendix Derivation Conservative System}.
\begin{remark}
The careful reader may notice that the equations above differ slightly from the ones obtained in \cite{BernardNoether}. The reasons are small mistakes in that paper, which we want to point out now for future reference:
\begin{enumerate}[label = $\bullet$]
\item In \cite[(2.15)]{BernardNoether} and the last line in the equation above that, the sign in front of the $\vec X$ term should be flipped.
\item In \cite[(3.13)]{BernardNoether}, the equation for $\Delta \vec R$ contains an additional term $\frac{\beta}{4} |\vec \Phi|^2 \nabla \vec \Phi$ which should not appear. The reason being that the choice $\nabla \vec X = \frac{\beta}{4} |\vec \Phi|^2 \nabla^\perp \vec \Phi$ right before \cite[(3.11)]{BernardNoether} is an abuse of notation as the right-hand side does not necessarily have non-vanishing curl (even though the right-hand side does have the correct divergence). This comes into play at \cite[(2.16)]{BernardNoether}, when the divergence of $\nabla^\perp \vec X$ is taken to be 0, whereas with our choice of $\nabla X$, this cannot be concluded. The resulting additional contribution cancels exactly with the remaining term involving $\vec X$.
\item In the first equation of \cite[(2.16)]{BernardNoether}, the $+$ in front of $|g|^{1/2}  \nabla_j ((\star \vec n) \cdot \nabla^j \vec X)$ should have been a minus. This however does not affect us as this term vanishes in our constrained system.
\item The differing signs in the equation for $\Delta \vec \Phi$ stem from several sign mistakes in the calculation between (2.16) and (2.17) in \cite{BernardNoether}.
\end{enumerate}
\end{remark}
\subsection{\texorpdfstring{$L^p$}{Lp}-estimate for the conformal parameters}\label{subsec:Lp estimates conformal parameter}
We recall the following theorem from \cite{MichelatPointwiseExpansion}, slightly generalizing \cite[Lemma V.3]{BernardRiviere}:
\begin{lemma}[See {\cite[Theorem 2.1]{MichelatPointwiseExpansion}}]\label{lem: original pointwise lambda estimate}
There exists a constant $\eps_2>0$ with the following property. Let $0<2^6 r<R<\infty$ and $\vec \Phi$ be a conformal immersion from $\Omega \cqq B_R \setminus B_r $ into $\R^3$ with $\|\nabla \vec n_{\vec \Phi}\|_{L^2(\Omega)}$ bounded, satisfying
\[\|\nabla \vec n_{\vec \Phi}\|_{L^{2,\infty}(\Omega)}\leq \eps_2.\]
Then, there exist $d, A\in \R$ such that
\begin{equation}
\|\lambda(x) - d\log|x| - A\|_{L^{\infty}(B_{R/2} \setminus B_{2r} )}\leq C\left (\|\nabla \lambda\|_{L^{2,\infty}(\Omega)}+ \int_{\Omega}|\nabla \vec n _{\vec \Phi}|^2 \dif{x}\right ),\label{original pointwise lambda estimate}
\end{equation}
where for all $r\leq \rho \leq R$, we have
\begin{equation}
\left | 2\pi d - \int_{\partial B_\rho}   \parpar{r}{\lambda} \dif{l}\right | \leq C \left (\int_{\Omega}|\nabla \vec n_{\vec \Phi}|^2 \dif{x} + \frac{1}{\log (R/\rho)} \int_{\Omega}|\nabla \vec n_{\vec \Phi}|^2\dif{x}\right ).\label{d estimate}
\end{equation}
\end{lemma}
\begin{remark} \label{rem:mean value argument}
By the mean value theorem, we can choose $\rho \in (r,2r)$ such that
\begin{equation}
\left | \int_{\partial B_{\rho}} \parpar{r}{\lambda} \dif{l}\right| \leq \frac{1}{r} \int_{B_{2r} \setminus B_{r} } |\nabla \lambda| \dif{x} \leq C \|\nabla \lambda \|_{L^{2,\infty}(\Omega)}.\label{mean value argument for d}
\end{equation}
\end{remark}
For the next lemma, we need to evaluate the following limit:
\begin{prop}
It holds that
\begin{equation}
\lim _{x\searrow 0, y\to 0} \frac{1-x^y}{y} = \infty.\label{1-x^y over y limit}
\end{equation}
\end{prop}
\begin{proof}
We rewrite
\[f(x,y)\cqq \frac{1-x^y}{y} = \frac{1-e^{\ln(x) y}}{\ln(x) y}\ln(x).\]
Let $\lim_{k\to \infty}(x_k,y_k)=(0,0)$ with $x_k>0$. Up to subsequence, $\ln(x_k)y_k \to C \in \R \cup \{\pm \infty\}$. If $C= -\infty$ then $y_k>0$ for large $n$ and also $1-e^{\ln(x_k)y_k} \geq 1/2$ for large $n$ and so $f(x_k,y_k)\geq \frac{1/2}{y_k}\to \infty$. Otherwise $\frac{1-e^{\ln(x_k)y_k}}{\ln(x_k)y_k} \to \frac{1-e^C}{C}<0$ (continuously extended through $C=0$) and so $f(x_k,y_k)\to \infty$. By the usual subsequence argument, this yields the assertion.
\end{proof}

In \cite[Section 3.1]{MichelatMondino}, it was shown that the conformal parameters $e^{\lambda_k}$ in neck regions of converging Willmore spheres satisfy an uniform $L^p$-estimate for some $p>2$ (notice that $p=2$ is equivalent to bounded area). We prove a similar result and extend it by dropping the assumption that $\|\lambda_k\|_{L^{\infty}(K)}< \infty$ for $K\subset \subset B_{1}\setminus \{0\}$, which is crucial for the bubbles that are not directly attached to a bubble or immersion carrying area in the limit. 
\begin{lemma}[$L^p$ bounds for conformal parameters in converging neck regions] \label{lem:Lp lemma}
Suppose that $r_k\searrow 0$ as $k\to \infty$. Furthermore, suppose $\vec \Phi_k :B_{1}\setminus B_{r_k}\to \R^3$ are immersions such that
\begin{equation}
\|\nabla \vec n_{k}\|_{L^{2,\infty}(B_{1}\setminus B_{r_k})}\leq \eps_2, \label{Lp lemma: L2infty estimate for gradient n}
\end{equation}
where $\eps_2$ is the constant from \Cref{lem: original pointwise lambda estimate}. Assume further that
\begin{equation}
\mc A(\vec \Phi_k) + \|\nabla \lambda_k\|_{L^{2,\infty}(B_{1}\setminus B_{r_k})} + \|\nabla \vec n_k\|_{L^2(B_{1}\setminus B_{r_k})}\leq \Lambda \label{Lp lemma: L2,infty bound nabla lambda and L2 estimate gradient n}
\end{equation}
for some fixed $\Lambda >0$. Denote by $d_k$ the constant obtained from \Cref{lem: original pointwise lambda estimate} applied in $B_{1}\setminus B_{r_k}$. Then up to a subsequence, there exists $p=p(\Lambda)>2$ and $\delta = \delta(\Lambda)>0$ depending only on $\Lambda$ such that either
\begin{equation}
\sup_{k\to \infty}\|e^{\lambda_k}\|_{L^p(B_{1/2}\setminus B_{2r_k})}<C(\Lambda)\quad\text{and}\quad \liminf _{k\to \infty} d_k>-1+\delta,\label{Lp lemma:Lp estimate}
\end{equation}	
or \eqref{Lp lemma:Lp estimate} holds for the reflected immersion $\hat{\vec{\Phi}}_k:B_1\setminus B_{r_k}\to \R^3$, $\hat{\vec{\Phi}}_k(z) = \vec{\Phi}_k\left ( \frac{r_k}{z}\right )$ instead.
\end{lemma}
\begin{proof}
Using \eqref{Lp lemma: L2infty estimate for gradient n} and \eqref{Lp lemma: L2,infty bound nabla lambda and L2 estimate gradient n}, we can apply \Cref{lem: original pointwise lambda estimate} to deduce that there exist $d_k\in \R$ and $A_k\in \R$ satisfying
\begin{equation}
\|\lambda_k(x) - d_k \log |x| -A_k\|_{L^{\infty}(B_{1/2}\setminus B_{2{r_k}})} \leq C(\Lambda).\label{lambda estimate on annulus}
\end{equation}
By \eqref{d estimate} and \eqref{mean value argument for d}, it holds
\[|d_k| \leq \left |d_k - \frac{1}{2\pi} \int_{\partial B_{{\rho_k}}} \parpar{r}{\lambda_k} \dif{l}  \right |+\left |\frac{1}{2\pi} \int_{\partial B_{{\rho_k}}} \parpar{r}{\lambda_k} \dif{l} \right | \leq C(\Lambda).\]
Thus, $d_k$ remains bounded and we may choose a subsequence such that $\lim_{k\to \infty} d_k = d\in \R$. 
\medbreak
\textbf{Case 1: $|d+1|\leq\delta$.} We will choose $\delta = \delta(\Lambda)\in (0,1/2)$ sufficiently small later and bring this case to a contradiction. First of all, notice that we can estimate the area by
\begin{equation}
\mc A(\vec \Phi_k\vert _{B_{1/2}\setminus B_{2{r_k}}}) \leq C e^{C(\Lambda)} \int_{2{r_k}}^{1/2} e^{2A_k} r ^{2d_k+1} \dif{r} = C e^{C(\Lambda)} e^{2A_k}\frac{1-(4{r_k})^{2d_k+2}}{d_k+1}. \label{a priori area estimate}
\end{equation}
Furthermore, we denote by $d_{\text{int}}$ the intrinsic diameter of $(B_{1/2}\setminus B_{2{r_k}}, g_{\vec \Phi_k})$. This is bounded from below by the intrinsic distance between the two boundary components $\partial B_{1/2}$ and $\partial B_{2{r_k}}$. Using \eqref{lambda estimate on annulus}, we estimate this by
\begin{equation}
d_{\text{int}} \geq e^{-C(\Lambda)} \int_{2{r_k}}^{1/2} e^{A_k} r^{d_k} \dif{r} =C e^{-C(\Lambda)} e^{A_k}\frac{1 - (4{r_k})^{d_k+1}}{1+d_k},\label{a priori diameter estimate}
\end{equation}
extended continuously to $d_k=-1$. Furthermore, we estimate the length of $\partial B_{1/2}$ and $\partial B_{2{r_k}}$ by
\begin{equation}
\int_{\partial B_{1/2}} e^{\lambda_k} \dif{l}\leq C e^{C(\Lambda)} e^{A_k} \quad\text{and}\quad\int_{\partial B_{2r_k}} e^{\lambda_k} \dif{l} \leq C e^{C(\Lambda)} e^{A_k} {r_k}^{d_k+1}.\label{a priori length estimate on boundary}
\end{equation}
We apply Miura's inequality \cite[Theorem 1.1]{MiuraDiameterBound} to $\vec \Phi \vert _{B_{1/2}\setminus B_{2{r_k}}}$, using \eqref{a priori diameter estimate} and \eqref{a priori length estimate on boundary}:
\[e^{A_k}\frac{1 - (4{r_k})^{d_k+1}}{1+d_k} \leq C(\Lambda)\left (\int_{B_{1/2}\setminus B_{2{r_k}}} |H_{\vec \Phi_k}| e^{2\lambda_k}\dif{x} + e^{A_k}+ e^{A_k}{r_k}^{d_k+1}\right ).\]
Using Cauchy-Schwarz and the uniform bound on $\|\nabla \vec n_{\vec \Phi_k}\|_{L^2}$ from \eqref{Lp lemma: L2,infty bound nabla lambda and L2 estimate gradient n}, we can bound the integral by $C(\Lambda) \sqrt{\mc A(\vec \Phi_k\vert _{B_{1/2}\setminus B_{2{r_k}}}) } $ and together with \eqref{a priori area estimate}, we obtain
\begin{equation}
\frac{1 - (4{r_k})^{d_k+1}}{1+d_k} \leq C(\Lambda)\left (\sqrt{\frac{1-(4{r_k})^{2d_k+2}}{d_k+1}} + 1+ {r_k}^{d_k+1}\right ).\label{inequality for contradiction}
\end{equation}
Suppose first that after taking a subsequence, $d_k>-1$ for all $k$. By \eqref{1-x^y over y limit} and the fact that $1+r_k^{d_k+1}$ remains bounded, there is $\delta = \delta(\Lambda)>0$ such that if $|d+1|<\delta$, then for sufficiently large $k$, the right-hand side is bounded from above by
\[C(\Lambda) \sqrt{\frac{1-(4{r_k})^{2d_k+2}}{d_k+1}}  + \frac{1-(4r_k)^{d_k+1}}{2(1+d_k)} \]
We obtain
\[\frac{1 - (4{r_k})^{d_k+1}}{1+d_k} \leq C(\Lambda) \sqrt{\frac{1-(4{r_k})^{2d_k+2}}{d_k+1}} .\]
It follows
\[\frac{1 - (4{r_k})^{d_k+1}}{1+d_k} \leq C(\Lambda)( 1+(4{r_k})^{d_k+1}) \leq C(\Lambda).\]
Possibly choosing $\delta$ smaller, this contradicts \eqref{1-x^y over y limit} for sufficiently large $k$. If instead $d_k<-1$ for some subsequence, we multiply \eqref{inequality for contradiction} by $(4{r_k})^{-(d_k+1)}$ to arrive at
\[\frac{1-(4{r_k})^{-(d_k+1)} }{-(1+d_k)} \leq C(\Lambda)\left (\sqrt{\frac{1-(4{r_k})^{-2(d_k+1)}}{-(d_k+1)}} + (4{r_k})^{-(d_k+1)}+1\right ).\]
Now the contradiction follows as for the $d_k>-1$ case. We conclude that $|d+1| > \delta$ and so $|d_k+1| > \delta/2$ for sufficiently large $k$.
 \medbreak
\textbf{Case 2: $d>-1+\delta$.} It holds that
\begin{equation}
\Lambda \geq \mc A(\Phi_k\vert _{B_{1/2}\setminus B_{2{r_k}}}) \geq 2\pi e^{-C(\Lambda)} \int_{2{r_k}}^{1/2} r^{2d_k+1} e^{2A_k} dr \geq C(\Lambda) \frac{e^{2A_k}}{1+d_k}((1/2)^{2d_k+2} - {(2r_k)}^{2d_k+2}).\label{Ak bound using bounded area}
\end{equation}
The right-hand side only remains uniformly bounded if
\begin{equation}
\sup_k A_k\leq C(\Lambda).\label{Ak remain bounded}
\end{equation}
 For $k$ large enough, $d_k>-1+\delta/2$. Hence, choosing $p= \frac{4-\delta}{2-\delta}>2$ and using \eqref{Ak remain bounded},
\begin{align*}
\int_{B_{{1}/2}\setminus B_{2{r_k}}} e^{p\lambda_k} \dif{x} &\leq e^{C(\Lambda)} e^{pA_k} \int_{2{r_k}}^{1/2} r^{p d_k+1} \dif{r} \\
&\leq  C(\Lambda)\frac{1}{pd_k+2} \left ((1/2)^{pd_k + 2} - (2{r_k})^{pd_k+2}\right ) \\
&\leq C(\Lambda).
\end{align*}
Here we used that $pd_k + 2 > \delta/2 > 0$.
\medbreak
\textbf{Case 3: $d<-1-\delta$.} Similar to \eqref{Ak bound using bounded area}, we deduce that 
\begin{equation}
\sup_k A_k + (d_k+1)\ln( r_k) \leq C(\Lambda).\label{Ak bound in inverted setting}
\end{equation}
We consider the inversion of the domain
\[\hat{\vec{\Phi}}_k:B_1 \setminus B_{{r_k}}\to \R^3,\quad \tilde{\vec{\Phi}}_k(z)= \vec \Phi_k\left (\frac{r_k}{z}\right ).\]
It follows that
\begin{equation}
\hat{\lambda}_k(z) = \ln({r_k}) - 2\ln |z| + \lambda_k\left (\frac{r_k}{z}\right ). \label{reflected immersion conformal factor}
\end{equation}
From this and \eqref{lambda estimate on annulus}, we derive
\[\|\hat{\lambda}_k(x) -(-(2+d_k) \ln |x|) -( A_k + (d_k+1)\ln({r_k}) )\|_{L^{\infty}(B_{1/2}\setminus B_{2{r_k}})} \leq C(\Lambda). \]
We have essentially replaced $d_k$ by $-(2+d_k)$ converging to $-(2+d)>-1 + \delta$, which allows us to apply Case 2. 
\end{proof}
\begin{remark}\label{rem:no area in neck regions}
This lemma implies that neck regions do not contain area and diameter in the limit. Namely, if \eqref{Lp lemma:Lp estimate} holds, then Hölders inequality implies
\begin{equation}
\int_{B_{\alpha}\setminus B_{r_k/\alpha}} e^{2\lambda_k} \dif{x} \leq C \alpha ^{1-2/p} \to 0\label{no area property in neck regions}
\end{equation}
as $\alpha \to 0$, where the constant does not depend on $k$. The intrinsic diameter $d_{\text{int}}$ can be estimated from above by 
\begin{equation}
d_{\text{int}} \leq C e^{C(\Lambda)}\int_{r_k/\alpha}^{\alpha} e^{A_k} r^{d_k}\dif{r} \leq C \int_{B_{\alpha}\setminus B_{r_k/\alpha}} \frac{e^{\lambda_k}}{|x|}\dif{x} \leq C \alpha ^{1-2/p} \to 0.\label{no diameter property in neck regions}
\end{equation}
The estimates for reflected immersion are the same. Finally, notice that the contradiction for $d=-1$ did not need the assumption that the area is uniformly bounded.
\end{remark}

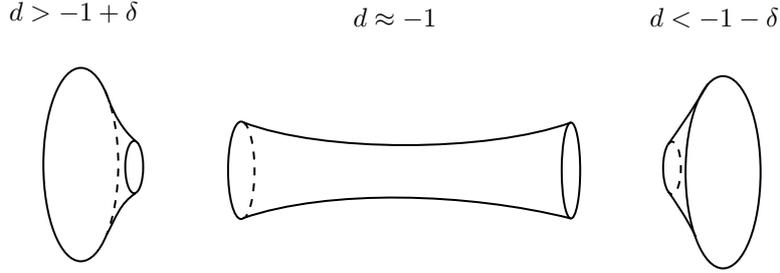
\begin{figure}[H] 
\centering

\tikzset{every picture/.style={line width=0.75pt}} 

\begin{tikzpicture}[x=0.75pt,y=0.75pt,yscale=-1,xscale=1]

\draw    (245,119.71) .. controls (284.57,135.07) and (366.61,134.96) .. (408.58,119.84) ;
\draw    (408.58,168.1) .. controls (357.62,155.19) and (287.71,153.07) .. (245,167.91) ;
\draw   (403.97,143.97) .. controls (403.97,130.64) and (406.03,119.84) .. (408.58,119.84) .. controls (411.13,119.84) and (413.2,130.64) .. (413.2,143.97) .. controls (413.2,157.3) and (411.13,168.1) .. (408.58,168.1) .. controls (406.03,168.1) and (403.97,157.3) .. (403.97,143.97) -- cycle ;
\draw  [draw opacity=0] (245.67,167.47) .. controls (245.12,168.06) and (244.53,168.37) .. (243.93,168.37) .. controls (240.34,168.37) and (237.43,157.4) .. (237.43,143.88) .. controls (237.43,130.35) and (240.34,119.38) .. (243.93,119.38) .. controls (244.3,119.38) and (244.66,119.5) .. (245.01,119.72) -- (243.93,143.88) -- cycle ; \draw   (245.67,167.47) .. controls (245.12,168.06) and (244.53,168.37) .. (243.93,168.37) .. controls (240.34,168.37) and (237.43,157.4) .. (237.43,143.88) .. controls (237.43,130.35) and (240.34,119.38) .. (243.93,119.38) .. controls (244.3,119.38) and (244.66,119.5) .. (245.01,119.72) ;  
\draw  [draw opacity=0][dashed] (245,119.71) .. controls (248.19,121.28) and (250.65,131.56) .. (250.65,144) .. controls (250.65,154.84) and (248.78,164.04) .. (246.19,167.26) -- (244.15,144) -- cycle ; \draw  [dashed] (245,119.71) .. controls (248.19,121.28) and (250.65,131.56) .. (250.65,144) .. controls (250.65,154.84) and (248.78,164.04) .. (246.19,167.26) ;  

\draw   (186.1,142.25) .. controls (186.1,134.98) and (188.11,129.08) .. (190.58,129.08) .. controls (193.06,129.08) and (195.06,134.98) .. (195.06,142.25) .. controls (195.06,149.51) and (193.06,155.41) .. (190.58,155.41) .. controls (188.11,155.41) and (186.1,149.51) .. (186.1,142.25) -- cycle ;
 \draw [draw opacity=0] (176.79,176.31) .. controls (173.44,184.57) and (168.92,189.63) .. (163.94,189.63) .. controls (153.61,189.63) and (145.25,167.86) .. (145.25,140.99) .. controls (145.25,114.13) and (153.61,92.35) .. (163.94,92.35) .. controls (169.63,92.35) and (174.72,98.97) .. (178.15,109.41) -- (163.94,140.99) -- cycle ; \draw (176.79,176.31) .. controls (173.44,184.57) and (168.92,189.63) .. (163.94,189.63) .. controls (153.61,189.63) and (145.25,167.86) .. (145.25,140.99) .. controls (145.25,114.13) and (153.61,92.35) .. (163.94,92.35) .. controls (169.63,92.35) and (174.72,98.97) .. (178.15,109.41) ; 
\draw    (178,108.82) .. controls (180.35,115.41) and (185.06,123.97) .. (191.41,129.29) ;
\draw    (176.79,176.31) .. controls (181.29,168) and (183.41,160.47) .. (191.29,155.41) ;
\draw  [draw opacity=0][dashed] (176.01,103.85) .. controls (180.06,112.77) and (182.63,126.1) .. (182.63,140.99) .. controls (182.63,154.9) and (180.38,167.44) .. (176.79,176.31) -- (163.94,140.99) -- cycle ; \draw  [dashed] (176.01,103.85) .. controls (180.06,112.77) and (182.63,126.1) .. (182.63,140.99) .. controls (182.63,154.9) and (180.38,167.44) .. (176.79,176.31) ;  
\draw   (465.8,144.85) .. controls (465.8,118.15) and (474.18,96.5) .. (484.51,96.5) .. controls (494.85,96.5) and (503.23,118.15) .. (503.23,144.85) .. controls (503.23,171.55) and (494.85,193.2) .. (484.51,193.2) .. controls (474.18,193.2) and (465.8,171.55) .. (465.8,144.85) -- cycle ;
\draw  [draw opacity=0] (456.98,154.31) .. controls (455.56,152.11) and (454.6,147.71) .. (454.6,142.65) .. controls (454.6,137.58) and (455.56,133.18) .. (456.98,130.98) -- (459.03,142.65) -- cycle ; \draw   (456.98,154.31) .. controls (455.56,152.11) and (454.6,147.71) .. (454.6,142.65) .. controls (454.6,137.58) and (455.56,133.18) .. (456.98,130.98) ;  
\draw    (456.98,130.98) .. controls (461.2,124.9) and (468.6,114.7) .. (477,100.3) ;
\draw    (456.98,154.31) .. controls (459.78,158.35) and (465.8,166.7) .. (471,177.5) ;
\draw  [draw opacity=0][dashed] (456.98,130.98) .. controls (457.59,130.03) and (458.29,129.48) .. (459.03,129.48) .. controls (461.48,129.48) and (463.46,135.38) .. (463.46,142.65) .. controls (463.46,149.91) and (461.48,155.81) .. (459.03,155.81) .. controls (458.29,155.81) and (457.59,155.27) .. (456.98,154.31) -- (459.03,142.65) -- cycle ; \draw  [dashed] (456.98,130.98) .. controls (457.59,130.03) and (458.29,129.48) .. (459.03,129.48) .. controls (461.48,129.48) and (463.46,135.38) .. (463.46,142.65) .. controls (463.46,149.91) and (461.48,155.81) .. (459.03,155.81) .. controls (458.29,155.81) and (457.59,155.27) .. (456.98,154.31) ;  

\draw (446.37,61) node [anchor=north west][inner sep=0.75pt]    {$d< -1-\delta $};
\draw (298.37,61) node [anchor=north west][inner sep=0.75pt]    {$d\approx -1$};
\draw (126.67,58) node [anchor=north west][inner sep=0.75pt]    {$d >-1+\delta $};

\end{tikzpicture}
\caption{The three different cases belong to three different types of neck regions. The case $d>-1+\delta$ corresponds to the case when the inner end of the annulus carries no area compared to the outer end. The case $d<-1-\delta$ is the reversed situation, i.e., the outer end of the annulus carries no area compared to the inner end. The case $d\approx -1$ occurs when the area is equidistributed throughout the annulus.}
\end{figure}

\begin{remark} \label{rem:conservative system for reflected immersion}

On first sight, the introduction of the reflected immmersion $\hat{\vec{\Phi}}_k$ could be problematic. The fact that the constrained Willmore equation \eqref{constrained Willmore surface} holds on the simply connected domain $B_1 $ in which $\vec \Phi_k$ is defined allows us to use the conservative system from \Cref{subsec:conservative system for constrained equation} without the introduction of residues, see e.g. \cite{LaurainRiviereEnergyQuantization} for such a problem when working with sequences whose conformal class is not bounded in the moduli space. However, the reflected immersion $\hat{\vec{\Phi}}_k(z) \cqq \vec \Phi_k(r_k/z)$ is defined only on $\C \setminus B_{r_k} $ which is not simply connected. 

Nevertheless, the conservative system on $B_1 $ directly translates to the conservative system on $\C\setminus B_{r_k}$ as the map $z\mapsto r_k/z$ is holomorphic, hence orientation preserving and conformal. More precisely, we can define the conservation quantities $\vec L_k$, $Y_k$, $\vec R_k$ and $S_k$ for $\vec \Phi_k$ on $B_1$ and then set
\begin{equation}
\hat{\vec{L}}_k(z) \cqq \vec L_k\left (\frac{r_k}{z}\right ),\quad\hat{Y}_k(z) \cqq Y_k\left (\frac{r_k}{z}\right ), \quad\hat{S}_k(z) \cqq S_k\left (\frac{r_k}{z}\right ),\quad\hat{\vec{R}}_k(z) \cqq \vec R_k\left (\frac{r_k}{z}\right ). \label{reflected quantities}
\end{equation}
As $\nabla \hat{\vec{L}}_k(z) = D(r_k/\cdot)(z)^\top \nabla \vec L_k (r_k/z)$ and $ \nabla^\perp \hat{\vec{L}}_k(z) = D(r_k/\cdot)(z)^\top \nabla^\perp \vec L_k (r_k/z)$ and similar formulas hold for the other quantities, \eqref{definition L R and S}, \eqref{alternative system for S and R} and \eqref{conservative system} still hold for the new quantities and so no residues show up for the reflected immersion.
\end{remark}

\subsection{Lorentz-type estimates on annuli}\label{subsec:Lorentz type estimates}
In this section, we want to quickly recall Lorentz-type estimates on annuli, which are independent of the conformal class of the annulus. First, we recall the estimates for harmonic function on annuli whose logarithmic component vanishes. The following lemma was obtained in \cite{MichelatPointwiseExpansion}, with ideas from \cite{LaurainRiviereAngularQuantization}.

\begin{lemma}[{\cite[Proposition 2.5]{MichelatPointwiseExpansion}}]\label{lem:harmonic function L2 weak to L2 strong}
Let $0<2^6 r<R< \infty$ and $u:B_{R}\setminus \overline{B}_r\to \R$ be harmonic such that for some $\rho_0\in (r,R)$
\[\int_{\partial B_{\rho_0}} \parpar{r}{u} \dif{l} = 0.\]
Then, there is a universal constant $C$ such that for all $\left (\frac{r}{R}\right )^{1/3} < \alpha < \frac{1}{4}$
\[\|\nabla u\|_{L^{2,1}(B_{\alpha R}\setminus \overline{B}_{\alpha^{-1}r})} \leq C \sqrt{\alpha} \|\nabla u\|_{L^{2,\infty}(B_{R}\setminus \overline{B}_{r})}.\]
\end{lemma}
We now consider the case when the right-hand side does not consist of Jacobians, but sufficiently regular data is assumed. It is important to obtain the estimates independent of the conformal class of the domain.
\begin{lemma}\label{lem:L21 estimate for Poisson in divergence form}
Suppose $0<\eps < 1/4$ and $p>1$. Let $g_1:B_1\setminus B_\eps\to \C$, $g_2:B_1\setminus B_{\eps} \to \R$ with $\|g_1\|_{L^{2,1}(B_1\setminus B_\eps)}<\infty$ and $\|g_2\|_{L^p(B_1\setminus B_\eps)} < \infty$.

\begin{enumerate}[label = \arabic*)]
\item  \label{lem:L21 estimate Poisson first part} Let $\phi_1 \in W^{1,2}_0(B_1\setminus B_\eps)$ be the unique solution to
\begin{equation}
\Delta \phi_1 = \Div(g_1) \quad\text{in }B_1\setminus B_\eps.\label{Poisson equation on annulus}
\end{equation}
Then there exists a constant $ C_\alpha>0$ depending only on $\alpha \in (\eps^{1/2}, 1/2)$ such that
\begin{equation}
\|\phi_1\|_{L^\infty(B_1\setminus B_\eps)} + \|\nabla \phi_1\|_{L^{2,1}(B_{\alpha}\setminus B_{\alpha^{-1} \eps})} \leq C_\alpha (1+\|g_1\|_{L^{2,1}(B_1\setminus B_\eps)}). \label{L 21 estimate for Poisson equation}
\end{equation}
\item \label{lem:L21 estimate Poisson second part} Let $\phi_2\in W^{1,2}_0(B_1\setminus B_\eps)$ be the unique solution to 
\begin{equation}
\Delta \phi_2 = g_2\quad\text{in }B_1\setminus B_\eps.\label{Poisson equation on annulus not in divergence form}
\end{equation} 
Then there exists a constant $C_{\alpha, p}>0$ depending only on $\alpha \in (\eps^{1/2}, 1/2)$ and $p$ such that
\begin{equation}
\|\phi_2\|_{L^\infty(B_1\setminus B_\eps)} + \|\nabla \phi_2\|_{L^{2,1}(B_{\alpha}\setminus B_{\alpha^{-1} \eps})} \leq C_{\alpha,p} (1+\|g_2\|_{L^{p}(B_1\setminus B_\eps)}). \label{L 21 estimate for Poisson equation not in divergence form}
\end{equation}
\end{enumerate} 
\end{lemma}
\begin{proof}
\begin{enumerate}[label = \arabic*)]
\item Suppose that $p\in (1,\infty)$ and $\Omega\subset \C$ is open and bounded with smooth boundary. For $f\in L^p(\Omega, \R^2)$, \cite[Theorem 3.29, 3.31]{AmbrosioLectureNotes} shows that there is a unique solution $u\in W^{1,p}_0(\Omega)$ to
\[\Delta u = \Div(f)\quad\text{in }\Omega\]
and the operator $T:L^p(\Omega)\to L^p(\Omega)$, $Tf\cqq \nabla u$ satisfies $\|T\| \leq C(p, \Omega)$. Using the interpolation theorem for Lorentz spaces, see \cite[Theorem 3.3.3]{Helein} shows that $T:L^{2,1}(\Omega)\to L^{2,1}(\Omega)$ is a bounded linear operator with
\begin{equation}
\|Tf\|_{L^{2,1}(\Omega)}\leq C(\Omega)\|f\|_{L^{2,1}(\Omega)}.\label{L21 interpolation for Poisson equation}
\end{equation}
Consider $\psi_1\in W^{1,2}_0(B_1)$ solving 
\begin{equation}
\Delta \psi_1 = \Div\left (\chi_{B_1\setminus B_\eps} g_1\right )\quad \text{in }B_1.\label{cut Poisson equation on annulus trick}
\end{equation}
By standard elliptic regularity for $L^2$-data, we have
\begin{equation}
\|\nabla \phi_1\|_{L^2(B_1\setminus B_\eps)} + \|\nabla \psi_1\|_{L^2(B_1\setminus B_\eps)}\leq 2\|g_1\|_{L^2(B_1\setminus B_\eps)}.\label{L2 gradient estimate for phi 1 and phi 3}
\end{equation}
\eqref{L21 interpolation for Poisson equation} shows that 
\begin{equation}
\|\nabla \psi_1\|_{L^{2,1}(B_1)}\leq  C \|g_1\|_{L^{2,1}(B_1\setminus B_\eps)},\label{L21 estimate on the full ball}
\end{equation}
where $C=C(B_1)$ does not depend on $\eps$. Furthermore, after extending $\psi_1$ by 0 to $\C$, \cite[Theorem 3.3.4]{Helein} implies 
\begin{equation}
\|\psi_1\|_{L^{\infty}(B_1)}\leq C\|g_1\|_{L^{2,1}(B_1\setminus B_\eps)}.\label{phi 1 L infinity estimate}
\end{equation}
Consider $v\cqq (\psi_1-\phi_1)\vert _{B_1\setminus B_\eps}$. By definition, $v$ is harmonic and it holds by \eqref{phi 1 L infinity estimate}
\begin{equation}
\|v\|_{L^\infty(\partial B_\eps)} \leq \|\psi_1\|_{L^\infty(B_1)} \leq C\|g_1\|_{L^{2,1}(B_1\setminus B_\eps)},\quad v\vert _{\partial B_1} = 0.\label{harmonic part bounded at boundary}
\end{equation}
Using the classification of harmonic functions on annuli, we can write in complex notation
\[v(z) = l_0 \log |z| + \Re \sum _{n=-\infty}^\infty c_n z^n\]
for coefficients $l_0\in \R$, $c_n \in \C$. It follows that $c_0 = 0$ and by \eqref{harmonic part bounded at boundary}
\[ |l_0| |\log(\eps)|=\left | \dashint _{\partial B_\eps} v \dif{l} \right | \leq  C\|g_1\|_{L^{2,1}(B_1\setminus B_\eps)}.\]
We obtain that $|l_0| \leq C\|g_1\|_{L^{2,1}(B_1\setminus B_\eps)}\frac{1}{|\log(\eps)|}$. From here, we copy the proof of \cite[Lemma A.2]{LaurainRiviereAngularQuantization} to see that
\begin{equation}
\|\nabla v\|_{L^{2,1}(B_{\alpha}\setminus B_{\alpha^{-1}\eps})}\leq C_\alpha (1+\|\nabla v\|_{L^2(B_1\setminus B_\eps)})\leq C_\alpha (1+\|g_1\|_{L^2(B_1\setminus B_\eps)}). \label{harmonic part gets L21 estimates}
\end{equation}
Finally, the maximum principle and \eqref{harmonic part bounded at boundary} imply that $\|v\|_{L^\infty(B_1\setminus B_\eps)} \leq C\|g_1\|_{L^{2,1}(B_1\setminus B_\eps)}$ and so together with \eqref{phi 1 L infinity estimate}, we conclude
\begin{equation}
\|\phi_1\|_{L^\infty(B_1\setminus B_\eps)}\leq C\|g_1\|_{L^{2,1}(B_1\setminus B_\eps)}.\label{phi L infinity bound}
\end{equation}
\eqref{L21 estimate on the full ball}, \eqref{harmonic part gets L21 estimates}, and \eqref{phi L infinity bound} give \eqref{L 21 estimate for Poisson equation}.
\item Denote by $p^* = \frac{p}{p-1}$ the dual exponent of $p$. Extend $\phi_2$ and $g_2$ by $0$ inside $B_{\eps}$. Then
\begin{equation}
\int _{B_1} |\nabla \phi_2|^2 \dif{x} = - \int _{B_1} \phi_2 g_2 \dif{x} \leq \|\phi_2\|_{L^{p^*}(B_1)} \|g_2\|_{L^p(B_1\setminus B_\eps)}.\label{Testing nabla phi2 with itself}
\end{equation}
Applying standard Sobolev estimates, it holds
\begin{equation}
\|\phi_2\|_{L^{p^*}(B_1)}\leq C_p \| \phi_2\|_{W^{1,2}(B_1)} \leq C_p \|\nabla \phi_2\|_{L^2(B_1)}.\label{Lp star estiamte for phi2}
\end{equation}
Plugging in \eqref{Lp star estiamte for phi2} into \eqref{Testing nabla phi2 with itself} yields
\begin{equation}
\|\nabla \phi_2\|_{L^2(B_1)} \leq C_p \|g_2\|_{L^p(B_1\setminus B_\eps)}.\label{L2 estimate for phi2}
\end{equation}
From here we follow the same idea as in the proof of \ref{lem:L21 estimate Poisson first part}, where we use that 
for the weak solution $\psi_2 \in W^{2,p}(B_1)\cap W^{1,p}_0(B_1)$ of $\Delta \psi_2 = \chi _{B_1\setminus B_\eps} g_2$, the estimate
\[\|\psi_2\|_{L^\infty(B_1)} + \|\nabla \psi_2\|_{L^{2,1}(B_1)}\leq C_{p} \|g_2\|_{L^p(B_1\setminus B_\eps)}\]
holds by \cite[Theorem 9.17]{GilbargTrudinger}, standard Sobolev embeddings and \eqref{Lp L21 inequality}.
\end{enumerate}
\end{proof}

\section{Improved estimates on the mean curvature in neck regions}\label{sec:L21 estimates on mean curvature in neck regions}

\begin{lemma}[$L^{2,1}$-control on the mean curvature in neck regions] \label{lem: L21 control mean curvature}
Let $\Lambda > 0$. There exists a constant $\eps_3 = \eps_3(\Lambda)>0$ depending only on $\Lambda$ and a universal constant $\alpha_0>0$ such that the following holds. Let $r_k \searrow 0$ as $k\to \infty$ and let $\vec \Phi_k:B_{1}\to \R^3$ be a sequence of conformal constrained Willmore immersions such that for all $k$,
\begin{equation}
\sup _{r_k<s<1/2} \int _{B_{2s} \setminus B_s } |\nabla \vec n _{k}|^2 \dif{x} \leq \eps_3(\Lambda) \label{double annulus small energy}
\end{equation}
and
\begin{equation}
 \|\nabla \lambda_k\|_{L^{2,\infty}(B_{1})} +  \int _{B_{1} } |\nabla \vec n _{k}|^2 \dif{x} +|\alpha_k|+|\beta_k|+|\gamma_k| + \mc A(\vec \Phi_k) + \|\vec \Phi_k\|_{L^\infty(B_{1})}\leq \Lambda,\label{H L21 Lemma assumptions}
\end{equation}
where $\vec n_k$ and $\lambda_k$ are again the unit normal vector and the conformal factor of $\vec \Phi_k$ and $\alpha_k$, $\beta_k$, $\gamma_k$ are the associated Lagrange multipliers. Then it holds that either
\begin{equation}
\limsup _{k\to \infty}\left \|e^{\lambda_k} \vec H _{k} \right \|_{L^{2,1}(B_{\alpha_0}\setminus B_{\alpha_0 ^{-1 }r_k})} \leq C(\Lambda),\label{l21 estimate for H}
\end{equation}
where $\vec H_{k}$ is the mean curvature of $\vec \Phi_k$ and $C(\Lambda)$ is a constant which only depends on $\Lambda$, or \eqref{l21 estimate for H} holds for the family $\hat{\vec{\Phi}}_k:B_{1}\setminus B_{r_k} \to \R^3$ defined by 
\begin{equation}
\hat{\vec{\Phi}}_k(z) \cqq \vec{\Phi}_k \left ( \frac{r_k}{z}\right ).\label{inverted immersion}
\end{equation}
\end{lemma}
\begin{proof}
We will adapt the proof from \cite[Lemma VI.1]{BernardRiviere} and deal with the additional terms that arise. All constants $C$, which may change from line to line, may depend on $\Lambda$, but not on $k$. We will break the proof into several steps.

\subsubsection*{Part 1. Establishing estimates for \texorpdfstring{$\vec n_k,\,  \lambda_k \text{ and }\vec H_k$}{nabla n, lambda and H}}
We choose $\eps_3(\Lambda)$ smaller than $\eps_1(\Lambda)$ from \Cref{thm:eps regularity}. By applying this $\eps$-regularity result to the ball $B_{|x|/4}(x)$, it holds for $x\in B_{1/2} \setminus B_{2r_k} $ that
\begin{equation}
|\nabla \vec n_{k}(x)|^2 \leq C |x|^{-2} \int _{B_{4|x|/3} \setminus B_{3|x|/4} } |\nabla \vec n_{k}|^2 \dif{x} \leq C \eps_3(\Lambda) |x|^{-2}.\label{pointwise nabla n bound}
\end{equation}
It follows from \eqref{L2infty norm for gradient of log} that
\begin{equation}
\|\nabla \vec n_{k}\|_{L^{2,\infty}(B_{1/2} \setminus B_{2r_k} )} \leq C\sqrt{\eps_3(\Lambda)}.\label{vec n weak L2 bound}
\end{equation}
Choosing $C \sqrt{\eps_3(\Lambda)}$ smaller than $\eps_2$ from \Cref{lem: original pointwise lambda estimate} allows us to apply this lemma and \Cref{rem:mean value argument} to obtain constants $A_k$ (depending on $r_k,\Lambda$ and $\vec \Phi_k$) and constants $d_k$ such that 
\begin{equation}
|d_k|\leq C\label{dk bounded}
\end{equation}
for all $k\in \N$ and 
\begin{equation}
\|\lambda_k(x) - d_k\log |x| -A_k\|_{L^\infty(B_{1/4}\setminus B_{4r_k})}\leq C. \label{logarithmic estimate conformal factor}
\end{equation}
Up to a subsequence, we may apply \Cref{lem:Lp lemma} to see that $\lim_{k\to \infty} d_k = d$, where $|d+1|>\delta$ for some $\delta=\delta(\Lambda)>0$. In the case $d>-1+\delta$, we keep working with the immersion $\vec \Phi_k$, whereas in the case $d<-1-\delta$, we work with the reflected immersion $\hat{\vec{\Phi}}_k$ defined in \eqref{inverted immersion} as
\[\hat{\vec{\Phi}}_k:\C\setminus B_{r_k}\to \R^3, \quad \hat{\vec{\Phi}}_k(z)= \vec \Phi_k \left (\frac{r_k}{z}\right ).\]
The corresponding conformal factor $\hat{\lambda}_k$ of $\hat{\vec{\Phi}}_k$ is given by \eqref{reflected immersion conformal factor} and in particular $\hat{d}_k = -(2+d_k)\to -(2+d)=\hat{d}>-1+\delta$ as $k\to \infty$. The analogous estimates for \eqref{pointwise nabla n bound}, \eqref{vec n weak L2 bound}, and \eqref{logarithmic estimate conformal factor} still hold for the reflected immersion $\hat{\vec{\Phi}}_k$. The arguments up to and including \eqref{Lk pointwise estimate} are the same for $\vec \Phi_k$ and $\hat{\vec{\Phi}}_k$ and so we do not distinguish between them until then.
Notice that since we assume bounded area, \eqref{logarithmic estimate conformal factor} together with Jensens inequality imply
\begin{equation}
\Lambda \geq \mc A(\vec \Phi_k) \geq \int _{B_{2s}\setminus B_{s}} e^{2\lambda_k} \dif{x} \geq C s^2 e^{2 \dashint_{B_{2s}\setminus B_s} \lambda_k \dif{x} } \implies e^{\lambda_k(x)}\leq \frac{C}{|x|}\label{Jensen in Lemma}
\end{equation}
for all $x\in B_{1/4}\setminus B_{4r_k}$. We define
\begin{align}
\delta_k'(s)&\cqq \left ( s^{-2} \int _{B_{2s} \setminus B_{s/2} } |\nabla \vec n_{k}|^2 \dif{x} \right )^{1/2} \leq \frac{C}{s} \quad\text{for $s\in (2r_k, 1/2)$},\\
\delta_k(s)&\cqq \left ( s^{-2} \int _{B_{4s} \setminus B_{s/4} } |\nabla \vec n_{k}|^2 \dif{x} \right )^{1/2}\leq \frac{ C}{s} \quad\text{for $s\in (4r_k, 1/4)$}.\label{delta_k bound}
\end{align}
It holds that $\delta_k'(s')\leq 2 \delta_k(s)$ for all $s' \in (s/2, 2s)$. Furthermore, \eqref{pointwise nabla n bound} implies
\begin{equation}
|e^{\lambda_k(x)} \vec H_{k}(x)| \leq |\nabla \vec n_{k}(x)| \leq C\delta_k'(|x|)\quad \text{for all $x\in B_{1/2} \setminus B_{2r_k} $.}\label{e to lambda H and nabla n bound by delta}
\end{equation}
As in \cite[(VI.11)]{BernardRiviere}, an application of Fubini shows
\begin{equation}
\int_{4r_k}^{1/4} s \delta_k(s)^2 \dif{s} \leq C \|\nabla \vec n_k\|^2_{L^{2}(B_{1}\setminus B_{r_k})}\leq C.\label{VI 11 equation}
\end{equation}
\subsubsection*{Part 2. Estimates for \texorpdfstring{$\vec L_k$}{L}}
Let us consider \eqref{extrinsic constrained system} for a moment.  Recall that in local conformal coordinates, the Laplace-Beltrami operator $\Delta_g$ can be written as $\Delta_g = e^{-2\lambda_k} \Delta$, where $\Delta$ is the standard Laplacian on $\R^2$. Recall from \eqref{extrinsic constrained system} that
\begin{equation}
\Delta H_{k} = e^{2\lambda_k}(-2H_{k}(H_{k}^2-K_{k}) + 2\alpha_k H_{k} + \beta_k + \gamma_k K_{k})\quad \text{in $B_{1} \setminus B_{r_k} $}.\label{recast Poisson equation for H}
\end{equation}
The right-hand side of \eqref{recast Poisson equation for H} can be bounded, using \eqref{pointwise nabla n bound}, \eqref{Jensen in Lemma}, \eqref{e to lambda H and nabla n bound by delta}, and $|e^{2\lambda_k}K_k| \leq |\nabla \vec n_k|^2$, by
\begin{align}
&\quad \, e^{2\lambda_k}(-2H_{k}(H_{k}^2-K_{k})+2\alpha_k H_{k} +\beta_k + \gamma_k K_{k})\notag\\
& \leq C(e^{-\lambda_k}\delta_k'(|x|) |\nabla \vec n_k|^2 + e^{\lambda_k} \delta_k'(|x|) + e^{2\lambda_k} + |\nabla \vec n_k|^2)\notag\\
&\leq C\frac{e^{-\lambda_k}}{|x|^2} (\delta_k'(|x|) + e^{\lambda_k} ) \label{estimate right side for pointwise gradient estimate}
\end{align}
for $x\in B_{2^{-3}} \setminus B_{2^{3}r_k} $. We apply \cite[Theorem 3.9]{GilbargTrudinger} to \eqref{recast Poisson equation for H} in the region $B_{2|x|} \setminus B_{|x|/2} (x)$ for $x\in B_{1/4} \setminus B_{4r_k} $. With \eqref{logarithmic estimate conformal factor}, \eqref{e to lambda H and nabla n bound by delta}, and \eqref{estimate right side for pointwise gradient estimate}, it follows
\begin{align}
|\nabla H_{k} (x)| &\leq \frac{C}{|x|}(\|H_{k}\|_{L^{\infty}(B_{2|x|}\setminus B_{|x|/2})} + \||x|^2 \Delta H_{k}\|_{L^{\infty}(B_{2|x|}\setminus B_{|x|/2})})\notag\\
&\leq \frac{C}{|x|} e^{-\lambda_k(x)} (\delta_k(|x|) + e^{\lambda_k(x)}) \label{nabla H pointwise estimate}. 
\end{align}
Recall that $\vec L_{k}$ is defined in \eqref{definition L R and S} on $B_{1}$ to satisfy
\begin{equation} \label{vec L definition}
\nabla^\perp \vec L_{k} =  \underbrace{\frac{1}{2} (\nabla \vec H_{k} - 3\pi_{\vec n_{k}}(\nabla \vec H_{k})+ \nabla ^\perp \vec n_{k} \times \vec H_{k})}_{=-\frac{1}{2}(2\nabla \vec H_{k} - 3 H_{k} \nabla \vec n_{k} +  \vec H_{k} \times \nabla^\perp \vec n_{k})} -  \left (-\alpha_k \nabla \vec \Phi_k + \frac{\beta_k}{2} \vec \Phi_k \times \nabla ^\perp \vec \Phi_k- \frac{\gamma_k}{2} (\nabla \vec n_{k} + 2H_{k} \nabla \vec \Phi_k)\right ).
\end{equation}
In the case that we work with the reflected immersion, this can still be done, see \Cref{rem:conservative system for reflected immersion}. We let for $r\in (r_k, 1)$
\begin{equation}
\vec L_{k,r} \cqq \dashint _{\partial B_r } \vec L_{k} \dif{l}.\label{mean value for L}
\end{equation}
Using the Poincar\' e inequality, \eqref{H L21 Lemma assumptions}, \eqref{Jensen in Lemma}, \eqref{delta_k bound}, \eqref{nabla H pointwise estimate}, and \eqref{vec L definition}, we deduce for $x\in B_{2^{-3}}\setminus B_{2^3r_k}$
\begin{align}
e^{\lambda_k(x)}|\vec L_k(x) - \vec L_{k,|x|}| &\leq e^{\lambda_k(x)} \int _{\partial B_{|x|}}|\nabla \vec L_k| \dif{l} \notag\\
&\leq C |x| \left ( \frac{\delta_k(|x|) + e^{\lambda_k(x)}}{|x|} + \delta_k(|x|)^2 + e^{2\lambda_k(x)} + \delta_k(|x|)e^{\lambda_k(x)}\right )\notag\\
&\leq C(\delta_k(|x|) + e^{\lambda_k(x)})\label{L minus average estimate}.
\end{align}
Integrating \eqref{L minus average estimate} and using \eqref{H L21 Lemma assumptions} and \eqref{VI 11 equation} yields
\begin{equation}
\int_{B_{2^{-3}}\setminus B_{2^3 r_k}} e^{2\lambda_k} |\vec L_k(x) - \vec L_{k,|x|}|^2 \dif{x} \leq C.\label{L minus average integral estimate}
\end{equation}
Expressing $x\in \partial B_r$ as a pair $(r,\theta)$ in polar coordinates gives by \eqref{vec L definition}
\begin{align}
\frac{\dd \vec L_{k,r}}{\dd r} = \dashint_0^{2\pi} \partial_r \vec L_k (r,\theta) \dif{\theta} &= \dashint_0^{2\pi} 	H_k\frac{3}{2r} \parpar{\theta}{\vec n_k} - \frac{1}{2} \vec H_k \times \parpar{r}{\vec n_k}\dif{\theta}\notag\\
&\quad + \dashint_0^{2\pi} -\frac{\beta_k}{2}\vec \Phi_k \times \parpar{r}{\vec \Phi_k} + \frac{1}{r}\gamma_k H_k \parpar{\theta}{\vec \Phi_k} \dif{\theta}.\label{derivative averages of L}
\end{align}
Here, we used that the integral of the tangential derivative over a circle vanishes, namely $\int _{0}^{2\pi} \partial_\theta \vec H_k \dif{\theta} = \int _{0}^{2\pi} \partial _{\theta} \vec n_k \dif{\theta} = 0$. Setting $a_k(r) = |\vec L_{k,r}|$ and estimating the terms in \eqref{derivative averages of L} shows
\[|a_k'(r)| \leq \left | \frac{\dd \vec L_{k,r}}{\dd r}\right | \leq C e^{-\lambda_k} \delta_k(r)^2 + C\delta_k(r) +C e^{\lambda_k},\]
where we write $\lambda_k(r) = \dashint _{\partial B_r} \lambda_k \dd{l}$. It follows, using \eqref{VI 11 equation} and also Cauchy-Schwarz together with $\mc A(\vec \Phi_k)\leq \Lambda$, that
\[\int _{2^3r_k}^{2^{-3}} s e^{\lambda_k} |a_k'(s)| \dif{s} \leq C.\]
Thus
\begin{equation}
\int _{2^3r_k}^{2^{-3}} s^{1+d_k} |a_k'(s)| \dif{s} \leq C e^{-A_k}.\label{VI 19 counterpart}
\end{equation}
Equations \eqref{L minus average integral estimate} and \eqref{VI 19 counterpart} are the counterparts of \cite[(VI.14),(VI.19)]{BernardRiviere}. We follow precisely the same arguments as in \cite{BernardRiviere} up to \cite[(VI.24)]{BernardRiviere} to conclude that
\begin{equation}
e^{\lambda_k(x)} |\vec L_k (x)| \leq \frac{C}{|x|}\label{Lk pointwise estimate}
\end{equation}
for all $x\in B_{2^{-3}} \setminus B_{2^3 r_k}$, after possibly shifting $\vec L_{k}$ to ensure that $\vec L_{k,2^{-3}}=0$.
\subsubsection*{Part 3. Obtaining estimates for the balancing term \texorpdfstring{$\mc B_k$}{B_k}}
We define in $B_{2^{-3}}$
\begin{equation}
f_k \cqq e^{2\lambda_k}(-2\alpha_k + \beta_k \langle \vec \Phi_k, \vec n_k\rangle- \gamma_k H_k )\label{fk definition}
\end{equation}
and define
\begin{equation}
\mc B_k \cqq \int _{B_{\sqrt{r_k}}} f_k \dif x.
\end{equation}

\begin{lemma}\label{lem: balancing term bound}
It holds that
\begin{equation}
\limsup _{k\to \infty}  |\mc B_k|\log(1/r_k) \leq C.\label{Bk bound}
\end{equation}
\end{lemma}
\begin{proof}
We adapt \cite[Lemma 4.2]{LaurainRiviereEnergyQuantization}. 
Recall from \eqref{divergence form for balancing integrand} that
\[-\Div(\langle \nabla^\perp \vec L_k, \vec \Phi_k\rangle) = \Div(\langle \vec L_k, \nabla^\perp \vec \Phi_k\rangle ) = f_k.\]
In particular, using the divergence theorem, we see for any $r\in (2^3r_k, 2^{-3})$
\begin{equation}
\int _{B_r}f_k \dif{x} = \int _{\partial B_{r}} \left \langle \frac{1}{r}\partial _{\theta} \vec L_k(x),  \vec \Phi_k(x) \right \rangle \dif{l}(x)=\int _{\partial B_{r}} \left \langle \frac{1}{r}\partial _{\theta} \vec L_k(x),  \vec \Phi_k(x) - \vec \Phi_k(y)\right \rangle \dif{l}(x),\label{Using the divergence theorem on the L nabla Phi term}
\end{equation}
where $y \in \partial B_{r}$ is arbitrary. Here we used that that $\int _{\partial B_r} \partial_{\theta} \vec L_k \dif l= 0$. Denoting by $\pi_T$ the projection onto the tangent space and using that $\pi_T + \pi_{\vec n_k} = \id$, it follows from \eqref{definition L R and S} that
\begin{align}
\frac{1}{r} \partial _{\theta} \vec L_k &= -\frac{1}{2}\bigg(\pi_{T}(\partial_{r} \vec H_k)  - \frac{1}{r}\partial_\theta \vec n_k \times \vec H_k\bigg) \notag\\
&\quad \, +\pi_{\vec n_k}(\partial_r \vec H_k)\notag\\
&\quad \, - \alpha_k \partial_r \vec \Phi_k - \frac{\beta_k}{2} \vec \Phi_k \times \frac{1}{r}\partial_{\theta} \vec \Phi_k - \frac{\gamma_k}{2}(\partial_r \vec n_k + 2 H_k \partial_r \vec \Phi_k).\label{decomposition of gradient of L}
\end{align}
We may write $\vec \Phi_k(x) - \vec \Phi_k(y) = \int _{\Gamma_{y,x}} \frac{1}{r} \partial _{\theta} \vec \Phi_k \dif l$, where $\Gamma_{y,x}$ is the circular arc from $y$ to $x$. This shows in particular that $|\vec \Phi_k(x) - \vec \Phi_k(y)|\leq C r e^{\lambda_k(r)}$. It holds $|\pi_T(\partial_r \vec H_k)| = |H_k \partial_r \vec n_k| \leq |\nabla \vec n_k|^2 e^{-\lambda_k}$ by \eqref{e to lambda H and nabla n bound by delta} and so
\begin{equation}
\left |\int _{\partial B_{r}} \left \langle \pi_{T}(\partial_r \vec H_k)-\frac{1}{r} \partial_{\theta} \vec n_k \times \vec H_k,  \vec \Phi_k(x) - \vec \Phi_k(y)\right \rangle \dif{l}(x)\right | \leq C r^2 \delta_k(r)^2.\label{balancing lemma first and third term estimate}
\end{equation}
Furthermore, as $\vec n_k$ is orthogonal to $\partial_\theta \vec \Phi_k$, it holds
\begin{align*}
&\quad\, \left |\int _{\partial B_{r}} \left \langle \pi_{\vec n_k(x)}(\partial_r \vec H_k(x)),  \vec \Phi_k(x) - \vec \Phi_k(y)\right \rangle \dif{l}(x)\right | \\
&=\left |\int _{\partial B_{r}} \int _{\Gamma_{y,x}}\left \langle \pi_{\vec n_k(x)}(\partial_r \vec H_k(x)) -\pi_{\vec n_k(z)}(\partial_r \vec H_k(x)),  \frac{1}{r} \partial _{\theta} \vec \Phi_k(z) \right \rangle \dif l(z) \dif{l}(x)\right |. 
\end{align*}
As $\pi _{\vec n_k} = \vec n_k \otimes \vec n_k$, we have $|\nabla \pi_{\vec n_k} |\leq C |\nabla \vec n_k|$ and so $|\pi_{\vec n_k(x)}(\partial_r \vec H_k(x)) -\pi_{\vec n_k(z)}(\partial_r \vec H_k(x))| \leq C |\nabla \vec H_k|(x)  r \delta_k(r)$. Using also \eqref{nabla H pointwise estimate}, it follows
\begin{align}
\left |\int _{\partial B_{r}} \left \langle \pi_{\vec n_k(x)}(\partial_r \vec H_k(x)),  \vec \Phi_k(x) - \vec \Phi_k(y)\right \rangle \dif{l}(x)\right |&\leq C r^2 e^{\lambda_k(r)} \delta_k(r) \int _{\partial B_r}   |\nabla \vec H_k|  \dif l(x)\notag\\
&\leq C r^2 \delta_k(r) (\delta_k(r) + e^{\lambda_k(r)}).\label{balancing lemma second term estimate}
\end{align}
Finally, 
\begin{align}
&\quad \,\int _{\partial B_{r}} \left \langle - \alpha_k \partial_r \vec \Phi_k - \frac{\beta_k}{2} \vec \Phi_k \times \frac{1}{r}\partial_{\theta} \vec \Phi_k - \frac{\gamma_k}{2}(\partial_r \vec n_k + 2 H_k \partial_r \vec \Phi_k),  \vec \Phi_k(x) - \vec \Phi_k(y)\right \rangle \dif{l}(x) \notag\\
&\leq C r^2 e^{\lambda_k(r)} (e^{\lambda_k(r)} + \delta_k(r)).\label{balancing lemma fourth term estimate}
\end{align}
Inserting \eqref{balancing lemma first and third term estimate}, \eqref{balancing lemma second term estimate}, \eqref{balancing lemma fourth term estimate} into \eqref{Using the divergence theorem on the L nabla Phi term}, using \eqref{decomposition of gradient of L}, we conclude
\begin{equation}
\left |\frac{1}{r} \int _{B_r} f_k \dif x \right | \leq C r (\delta_k(r)^2 + e^{\lambda_k(r)} \delta_k(r) + e^{2\lambda_k(r)}).\label{good bound for integral fk}
\end{equation}
From \eqref{H L21 Lemma assumptions} and \eqref{VI 11 equation}, we deduce that the right-hand side is uniformly integrable in $r$ over $(2^3r_k, 2^{-3})$. In particular, we can find $s_{1,k}\in (2^3r_k, \sqrt{r_k})$ and $s_{2,k} \in (\sqrt{r_k},2^{-3})$ such that
\begin{equation}
\left | \int _{B_{s_{1,k}}} f_k \dif x\right |+\left | \int _{B_{s_{2,k}}} f_k \dif x\right | \leq \frac{C}{\log(1/r_k)}.\label{s1 and s2 integral}
\end{equation}
If $d>-1+\delta$, then $d_k>-1 + \delta/2$ for sufficiently large $k$ and so the bound for $\vec \Phi_k$, $\mc A(\vec \Phi_k)$, and $\mc W(\vec \Phi_k)$ together with Cauchy-Schwarz implies
\begin{equation}
\int _{B_{\sqrt{r_k}}\setminus B_{2^3r_k}} |f_k| \dif x \leq  C \sqrt{\mc A(\vec \Phi_k \vert _{B_{\sqrt{r_k}}\setminus B_{2^3r_k}})} \leq C \left (\int _{2^3r_k}^{\sqrt{r_k}} e^{2A_k}s^{2d_k+1} \dif{s} \right)^{1/2} \leq C r_k ^{\delta/4}.\label{4rk to sqrt rk integral fk}
\end{equation}
As $r_k^{\delta/4}$ is controlled by $1/\log(1/r_k)$, \eqref{4rk to sqrt rk integral fk} and \eqref{s1 and s2 integral} show \eqref{Bk bound}. Likewise, if $d<-1-\delta$, we use $s_{2,k}$ instead of $s_{1,k}$ to prove \eqref{Bk bound}.
\end{proof}
We can now proceed with the proof of \Cref{lem: L21 control mean curvature}.
\subsubsection*{Part 4. Controlling  \texorpdfstring{$\nabla Y_k$}{nabla Y}}
Recall the quantities $\vec X_k$, $\vec R_k$, $ S_k$ and $Y_k$ from \Cref{subsec:conservative system for constrained equation}, defined on the simply connected domain $\Omega = B_{2^{-3}}$. In particular, $Y_k$ was defined in \eqref{definition Y} as the solution to
\begin{equation}
\begin{cases}
\Delta Y_k =  f_k&\quad \text{in }B_{2^{-3}}, \\
\;\;\;Y_k=0& \quad \text{on } \partial B_{2^{-3}}.
\end{cases}\label{Recalling Yk}
\end{equation}
We also define the corresponding reflected quantities $\hat{\vec{\Phi}}_k$, $\hat{\vec{X}}_k$, $\hat{\vec{R}}_k$, $\hat{S}_k$, and $\hat{Y}_k$ as in \eqref{reflected quantities}. \Cref{rem:conservative system for reflected immersion} yields that these still satisfy the conservative system \eqref{conservative system} and satisfy the same equations \eqref{definition L R and S} and \eqref{alternative system for S and R}. In the following, we will first assume that $d>-1+\delta$. In this case, we will work with the original immersions $\vec \Phi_k$ and finish the proof in this case. In the end, see Part \hyperref[part 8]{8} of the proof, we will consider the case $d<-1-\delta$ and work with the reflected quantities. We adapt the necessary arguments in a suitable way.
\subsubsection*{Case 1: \texorpdfstring{$d>-1+\delta$}{d>-1+delta}}
It holds $d_k > -1+\delta/2$ for sufficiently large $k$. \eqref{Ak remain bounded} shows that $A_k\leq C$ and so \eqref{logarithmic estimate conformal factor} yields
\begin{equation}
e^{\lambda_k(x)} \leq C |x|^{\delta/2-1} \label{conformal factor bounded by positive exponent}
\end{equation}
for all $x\in B_{1/4}\setminus B_{4r_k}$. This implies the uniform estimate
\begin{equation}
\sup_{k\in \N}\|e^{\lambda_k}\|_{L^{p_1}(B_{1/2}\setminus B_{2r_k})}<C \label{H L21 lemma L^p estimate}
\end{equation}
for $p_1=p_1(\Lambda)\cqq \frac{\delta/2 - 2}{\delta/2 -1}>2$.
The next goal will be to show that we can obtain $L^{2,1}$-estimates for $\nabla Y_k$, see \eqref{strong nabla Yk estimate}. Notice that by \cite[Theorem 3.3.6]{Helein} and \eqref{H L21 Lemma assumptions}, we have
\begin{equation}
\|\nabla Y_k\|_{L^{2,\infty}(B_{1/4})}\leq C \|f_k\|_{L^1(B_{1/4})}\leq C.\label{L2 weak estimate for Yk}
\end{equation}
 We define $Y^1_k$ to be the solution of
\begin{equation}
\begin{cases}
\Delta Y^1_k = \chi_{B_{2^{-3}}\setminus B_{2^3r_k}} f_k&\quad \text{in }B_{1/4} , \\
\;\;\;Y^1_k=0& \quad \text{on } \partial B_{2^{-3}} .
\end{cases}\label{Yk alternative definition}
\end{equation}
Notice that
\begin{equation}
\int _{\partial B_{2^3r_k}} \partial_r Y^1_k \dif{l} = \int _{B_{2^3r_k}} \Delta Y^1_k \dif{x} = 0.\label{Neumann term for tilde Yk}
\end{equation}
We obtain from \eqref{conformal factor bounded by positive exponent}, \eqref{Jensen in Lemma}, \eqref{e to lambda H and nabla n bound by delta}, \eqref{delta_k bound}, and \eqref{H L21 Lemma assumptions} together with $d_k \to d > -1+\delta$ that
\begin{equation}
|f_k| \leq C|x|^{\delta/2 - 2} \quad \text{for $x\in B_{2^{-3}}\setminus B_{2^3r_k}$} \label{Y_k Laplacian estimate}
\end{equation}
and
\begin{equation}
\|f_k\|_{L^{p_2}(B_{2^{-3}}\setminus B_{2^3r_k})} \leq C\label{fk Lp estimate}
\end{equation}
uniformly for $p_2=p_2(\Lambda) \cqq \frac{-2 + \delta/4}{-2 + \delta/2}>1$ depending only on $\Lambda$. Calderon-Zygmund estimates and standard Sobolev embeddings imply
\begin{equation}
\|Y^1_k\|_{L^\infty(B_{2^{-3}})}+\|\nabla Y^1_k\|_{L^{\frac{2p_2}{2-p_2}}(B_{2^{-3}})}\leq C \|Y^1_k\|_{W^{2,p_2}(B_{2^{-3}})}\leq C.\label{pointwise Yk tilde estimate}
\end{equation}
Using \cite[Theorem 3.9]{GilbargTrudinger} on the annulus $B_{2|x|} \setminus B_{|x|/2} $, \eqref{Y_k Laplacian estimate} and \eqref{pointwise Yk tilde estimate} give
\begin{equation}
|x| |\nabla Y^1_k(x)| \leq C\label{Yk L2 weak estimate}
\end{equation}
for $x\in B_{2^{-4}} \setminus B_{2^{4}r_k} $. 

The difference $u_k \cqq Y_k - Y^1_k$ satifies $\Delta u_k = 0$ on $B_{2^{-3}}\setminus B_{2^3r_k}$. By the classification of harmonic maps on annuli, we get (in complex notation)
\begin{equation}
u_k(z)=l_k \log |z| + \Re \sum_{n=-\infty}^\infty a_{n,k} z^n\label{harmonic representation on annulus}
\end{equation}
for coefficients $l_k\in  \R$, $a_{n,k} \in \C$. It follows directly by \eqref{Neumann term for tilde Yk}
\[2\pi l_k= \int _{\partial B_{2^3r_k}} \partial_r (Y_k - Y^1_k) \dif{l} = \int_{ \partial B_{2^3r_k} } \partial_r Y_k \dif{l} = \int _{B_{2^3r_k}} f_k \dif x. \]
Together with \eqref{4rk to sqrt rk integral fk} and \eqref{Bk bound}, we see
\begin{equation}
|l_k| \leq \frac{C}{\log(1/r_k)} \quad\text{for all $k\in \N$.}\label{bound for lk}
\end{equation}
In particular, \eqref{bound for lk} together with \eqref{L21 norm for gradient of log} yields that $\|\nabla (l_k \log |\cdot|)\|_{L^{2,1}(B_{2^{-4}}\setminus B_{2^{4}r_k})}\leq C$. Using \Cref{lem:harmonic function L2 weak to L2 strong} applied to the harmonic map $u_k(x) - l_k \log |x|$ together with \eqref{L2 weak estimate for Yk} and \eqref{pointwise Yk tilde estimate} yield
\begin{align}
\|\nabla u_k \|_{L^{2,1}(B_{2^{-4}}\setminus B_{2^{4}r_k})}&\leq |l_k| \|\nabla \log |\cdot|\|_{L^{2,1}(B_{2^{-4}}\setminus B_{2^{4}r_k})} + C\|\nabla(u_k - l_k \log |\cdot|)\|_{L^{2,\infty}(B_{2^{-4}}\setminus B_{2^{4}r_k})}  \notag\\
&\leq C(1+\|\nabla Y_k\|_{L^{2,\infty}(B_{2^{-4}}\setminus B_{2^{4}r_k})} + \|\nabla Y^1_k\|_{L^{2,\infty}(B_{2^{-4}}\setminus B_{2^{4}r_k})} + |l_k| )\notag\\
& \leq C.  \label{strong harmonic estimate for uk}
\end{align}
As $u_k$ is bounded on $\partial B_{2^{-3}}$ and $|l_k|\leq C$, it holds $\dashint _{\partial B_{2^{-4}}} u_k \dif l \leq C$. Thus, \cite[Proposition 2.7]{MichelatPointwiseExpansion} implies 
\begin{equation}
\|u_k\|_{L^\infty(B_{2^{-4}}\setminus B_{2^{4}r_k})} \leq C. \label{uk bounded}
\end{equation}
Applying again \cite[Theorem 3.9]{GilbargTrudinger} to $u_k$ shows $|x| |\nabla u_k(x)| \leq C$ on $B_{2^{-5}}\setminus B_{2^5r_k}$. Together with \eqref{Yk L2 weak estimate},
\begin{equation}
|x| |\nabla Y_k(x)| \leq C\quad \text{for all }x\in  B_{2^{-5}} \setminus B_{2^{5}r_k} . \label{pointwise Yk estimate}
\end{equation}
As the $L^{\frac{2p_2}{2-p_2}}$-norm bounds the $L^{2,1}$-norm on bounded domains (see \eqref{Lp L21 inequality}), \eqref{pointwise Yk tilde estimate} and \eqref{strong harmonic estimate for uk} imply
\begin{align}
\|\nabla Y_k\|_{L^{2,1}(B_{2^{-4}}\setminus B_{2^{4}r_k})}\leq C(\|\nabla Y^1_k\|_{L^{\frac{2p_2}{2-p_2}}(B_{2^{-4}}\setminus B_{2^{4}r_k})} + \|\nabla u_k\|_{L^{2,1}(B_{2^{-4}}\setminus B_{2^{4}r_k})}) \leq C.  \label{strong nabla Yk estimate}
\end{align}
\subsubsection*{Part 5. \texorpdfstring{$L^2$}{L^2} estimates for \texorpdfstring{$\nabla S_k$}{nabla S_k} and \texorpdfstring{$\nabla \vec R_k$}{nabla R_k}}
Notice that $\vec X_k = \frac{\beta_k}{4} |\vec \Phi_k|^2 \nabla^\perp \vec \Phi_k$ satisfies $\|\vec X_k\|_{L^{p_1}(B_{2^{-4}}\setminus B_{2^{4}r_k})}\leq C$ by the $L^\infty$-estimate for $\vec \Phi_k$ \eqref{H L21 Lemma assumptions} and \eqref{H L21 lemma L^p estimate}. \eqref{Jensen in Lemma} shows
\begin{equation}
|x| |\vec{X}_k|(x) \leq C\quad \text{for all }x\in B_{2^{-4}} \setminus B_{2^{4}r_k} .\label{Xk pointwise estimate}
\end{equation}
Hence, using \eqref{definition L R and S} together with \eqref{delta_k bound}, \eqref{e to lambda H and nabla n bound by delta}, \eqref{Lk pointwise estimate}, \eqref{pointwise Yk estimate}, and \eqref{Xk pointwise estimate}, we conclude
\begin{equation}
|\nabla S_k(x)| + |\nabla \vec R_k(x)| \leq \frac{C}{|x|} \label{pointwise nabla Sk and nabla Rk bound}
\end{equation}
for all $x\in B_{2^{-5}} \setminus B_{2^{5}r_k}$
and in particular by \eqref{L2infty norm for gradient of log}
\begin{equation}
\|\nabla S_k\|_{L^{2,\infty}(B_{2^{-5}} \setminus B_{2^{5}r_k})}+\|\nabla   \vec R_k\|_{L^{2,\infty}(B_{2^{-5}} \setminus B_{2^{5}r_k})} \leq C.\label{L2 weak bound nabla Sk and nabla Rk}
\end{equation}
Recall from \eqref{alternative system for S and R} that it holds
\begin{equation}
\begin{aligned}
\nabla S_k &= - \langle \nabla ^\perp \vec R_k, \vec n_k \rangle + \nabla ^\perp Y_k,\\
\nabla \vec R_k &= \vec n_k \times \nabla ^\perp \vec R_k + (\nabla ^\perp S_k + \nabla Y_k ) \vec n_k.
\end{aligned} \label{alternative versions for nabla S and nabla R}
\end{equation}
We set 
\[S_{k,r} \cqq \dashint_{\partial B_r } S_k \dif{l}, \quad \vec R_{k,r} \cqq \dashint _{\partial B_r } \vec R_k \dif{l}, \quad \vec n_{k,r} \cqq \dashint _{\partial B_r } \vec n_k \dif{l}.\]
Then it follows from \eqref{alternative versions for nabla S and nabla R}
\begin{align}
\frac{\dd S_{k,r}}{\dd r} &= \dashint_0^{2\pi} \partial _r S_k (r,\theta) \dif{\theta} = \dashint _0^{2\pi} \left \langle \frac{1}{r}\partial _{\theta}\vec R_k   , \vec n_k - \vec n_{k,r}\right \rangle \dif{\theta}, \label{calculation for estimate of S}\\
\frac{\dd \vec R_{k,r}}{\dd r} &= \dashint _0^{2\pi} \partial_r\vec R_k (r,\theta) \dif{\theta} \notag\\
&= \dashint_0^{2\pi} -(\vec n_k - \vec n_{k,r}) \times \frac{1}{r}\partial_{\theta} \vec R_k - \frac{1}{r} \partial_{\theta} S_k (\vec n_k - \vec n_{k,r}) \dif{\theta} + \dashint_0^{2\pi} \partial _{r}Y_k   \vec n _k  \dif{\theta}. \label{calculation for estimate of R}
\end{align}
Using the pointwise bound for $\nabla \vec n_k$ from \eqref{pointwise nabla n bound}, it holds
\begin{equation}
|\vec n_k(x) - \vec n_{k,|x|}| \leq  \int _{\partial B_{|x|}} |\nabla \vec n_k | \dif{l} \leq C \delta_k(|x|) |x|.\label{estimate of n minus average}
\end{equation}
In particular, \eqref{calculation for estimate of S} and \eqref{calculation for estimate of R} show
\begin{align}
\left | \frac{\dd \vec R_{k,r}}{\dd r}\right |& \leq C \delta_k(r) \int _{\partial B_r} (|\nabla \vec R_k| + |\nabla S_k|)\dif{l} + \frac{1}{2\pi} \int_{\partial B_r} \frac{1}{r} \left | \partial_r Y_k \right | \dif{l} , \label{derivative averages estimate for R}\\
\left | \frac{\dd S_{k,r}}{\dd r} \right | &\leq C \delta_k(r) \int _{\partial B_r} |\nabla \vec R_k|\dif{l}.\label{derivative averages estimate for S}
\end{align}
\eqref{pointwise nabla Sk and nabla Rk bound}, \eqref{calculation for estimate of S}, \eqref{calculation for estimate of R}, and \eqref{estimate of n minus average} imply
\begin{equation}
\left | \frac{\dd S_{k,r}}{\dd r}\right |^2 + \left | \frac{\dd \vec R_{k,r}}{\dd r} \right |^2 \leq C \delta_k(|x|)^2 + C \int _{0}^{2\pi} \left | \partial_r Y_k\right |^2\dif{\theta}.\label{derivative averages squared estimate}
\end{equation}
Integrating \eqref{derivative averages squared estimate} over $(2^{4}r_k, 2^{-4})$ and using \eqref{VI 11 equation} and \eqref{strong nabla Yk estimate} gives
\begin{equation}
\int _{2^4 r_k}^{2^{-4}}\left [\left | \frac{\dd S_{k,r}}{\dd r}\right |^2 + \left | \frac{\dd \vec R_{k,r}}{\dd r} \right |^2 \right ] r \dif{r} \leq C. \label{average L2 integral for R and S bounded}
\end{equation}
Consider $\vec R^1_k$ and $\vec R^2_k \cqq  \vec R_k - \vec R^1_k  $ satisfying
\begin{equation}
\begin{cases}
\Delta \vec R^1_k = \Div\left ( \nabla Y_k  \vec n_k\right )&\quad \text{in $B_{2^{-5}}\setminus B_{2^{5}r_k}$},\\
\;\;\; \vec R^1_k = 0 &\quad \text{on $\partial (B_{2^{-5}}\setminus B_{2^{5}r_k})$},\\
\Delta \vec R^2_k = \nabla^\perp \vec n_k \cdot \nabla S_k + \nabla^ \perp \vec n_k \times \nabla \vec R_k &\quad \text{in $B_{2^{-5}}\setminus B_{2^{5}r_k}$},\\
\;\;\; \vec R^2_k = \vec R_k &\quad \text{on $\partial (B_{2^{-5}}\setminus B_{2^{5}r_k})$},
\end{cases}\label{R1k auxiliary solution}
\end{equation}
where the equation for $\vec R^2_k$ follows from \eqref{conservative system}. With \eqref{strong nabla Yk estimate}, we can apply \Cref{lem:L21 estimate for Poisson in divergence form} \ref{lem:L21 estimate Poisson first part} to $\vec R^1_k$ to obtain
\begin{equation}
\|\vec R^1_k\|_{L^\infty(B_{2^{-5}}\setminus B_{2^{5}r_k})} + \|\nabla \vec R^1_k\|_{L^{2,1}(B_{2^{-7}}\setminus B_{2^{7}r_k})} \leq C.\label{strong L21 estimates for R1k}
\end{equation}
Notice that \eqref{average L2 integral for R and S bounded} and \eqref{strong L21 estimates for R1k} give
\begin{equation}
\int _{2^{7}r_k}^{2^{-7}} \left | \frac{\dd \vec R^2_{k,r}}{\dd r}\right |^2 r \dif{r} \leq C, \label{average L2 integral for R2k}
\end{equation}
where $\vec R^2_{k,r} = \dashint _{\partial B_r} \vec R^2_k \dif{l}$ again denotes the average. We apply \cite[Lemma 10]{LaurainRiviereAngularQuantization} to $\vec R^2_k$ and $S_k$, using \eqref{average L2 integral for R2k} and \eqref{R1k auxiliary solution} for $\vec R^2_k$ and \eqref{average L2 integral for R and S bounded} and \eqref{conservative system} for $S_k$ to obtain
\begin{equation}
\|\nabla S_k\|_{L^{2}(B_{2^{-6}} \setminus B_{2^{6}r_k} )} + \|\nabla \vec R^2_k\|_{L^{2}(B_{2^{-8}} \setminus B_{2^{8}r_k} )} \leq C.\label{L2 estimate for S and R}
\end{equation}
In particular, \eqref{strong L21 estimates for R1k}, \eqref{L2 estimate for S and R}, and \eqref{pointwise nabla Sk and nabla Rk bound} show
\begin{equation}
\|\nabla S_k\|_{L^{2}(B_{2^{-5}} \setminus B_{2^{5}r_k} )} +\|\nabla \vec R_k\|_{L^2(B_{2^{-5}} \setminus B_{2^{5}r_k} )} \leq C.\label{L2 estimates for nabla Rk and nabla Sk}
\end{equation}
\subsubsection*{Part 6. \texorpdfstring{$L^{2,1}$}{L^(2,1)}-estimates for \texorpdfstring{$\nabla S_k$}{nabla S_k} and \texorpdfstring{$\nabla \vec R_k$}{nabla R_k}}
Estimating the right-hand side in \eqref{derivative averages estimate for R} and \eqref{derivative averages estimate for S} with the help of \eqref{L2 estimate for S and R}, Fubini's theorem, \eqref{VI 11 equation}, the duality between $L^{2,\infty}$ and $L^{2,1}$, \eqref{strong nabla Yk estimate}, \eqref{L2infty norm for gradient of log}, and an application of Cauchy-Schwarz gives 
\begin{align}
\int _{2^5 r_k} ^{2^{-5}} \left | \frac{\dd \vec R_{k,r}}{\dd r} \right | \dif{r} &\leq \int _{2^5 r_k} ^{2^{-5}}  \int _{\partial B_r} C \delta_k(r)(|\nabla \vec R_k| + |\nabla S_k|) + \frac{1}{2\pi r} \left | \partial_r Y_k  \right | \dif{l} \dif r \notag\\
& \leq \left (\int _{2^5 r_k} ^{2^{-5}} r \delta_k^2(r) \dif{r}\right )^{1/2} \left ( \|\nabla S_k\|_{L^{2}(B_{2^{-5}}\setminus B_{2^5 r_k})} + \|\nabla \vec R_k\|_{L^{2}(B_{2^{-5}}\setminus B_{2^5 r_k})}\right )\notag\\
&\quad \, + \left \| \frac{1}{2\pi |x|}\right \|_{L^{2,\infty}(B_{2^{-5}}\setminus B_{2^5 r_k})}\left \| \partial_r Y_k  \right \|_{L^{2,1}(B_{2^{-5}}\setminus B_{2^5 r_k})}  \notag\\
&\leq C ,\label{modified supremum estimate for R}\\ 
\int _{2^5 r_k} ^{2^{-5}} \left | \frac{\dd S_{k,r}}{\dd r}\right | \dif{r} &\leq \int _{2^5 r_k} ^{2^{-5}}  \int _{\partial B_r} C \delta_k(r)(|\nabla \vec R_k|  ) \dif{l} \dif{r} \leq C. \label{modified supremum estimate for S}
\end{align}
By shifting $S_k$ and $\vec R_k$, we may assume $S_{k,2^5r_k} = \vec R_{k, 2^5r_k} = 0$. Together with the pointwise bound \eqref{pointwise nabla Sk and nabla Rk bound}, \eqref{modified supremum estimate for R} and \eqref{modified supremum estimate for S} imply
\begin{equation}
\|S_k\|_{L^\infty(B_{2^{-5}}\setminus B_{2^{5}r_k})} + \|\vec R_k\|_{L^\infty(B_{2^{-5}}\setminus B_{2^{5}r_k})} \leq C. \label{L infinity estimate for S and R}
\end{equation}
\eqref{L infinity estimate for S and R} and \eqref{strong L21 estimates for R1k} also show that $\|\vec R^2_k\|_{L^{\infty}(B_{2^{-5}}\setminus B_{2^{5}r_k})} \leq C$. We are now in the situation to apply \cite[Lemma 9]{LaurainRiviereAngularQuantization} to $S_k$ and $\vec R^2_k$, which shows 
\begin{equation}
\|\nabla S_k\|_{L^{2,1}(B_{2^{-6}}\setminus B_{2^{6}r_k})}+\|\nabla \vec R^2_k\|_{L^{2,1}(B_{2^{-6}}\setminus B_{2^{6}r_k})}\leq C.\label{strong L21 estimate for S and R2k}
\end{equation}
Finally, \eqref{strong L21 estimate for S and R2k} and \eqref{strong L21 estimates for R1k} show that
\begin{equation}
\|\nabla S_k\|_{L^{2,1}(B_{2^{-7}}\setminus B_{2^{7}r_k})}+\|\nabla \vec R_k\|_{L^{2,1}(B_{2^{-7}}\setminus B_{2^{7}r_k})}\leq C.\label{strong L21 estimate for S and Rk}
\end{equation}
\subsubsection*{Part 7. Conclusion in the first case}
We use $2 e^{2\lambda_k} \vec H_k = \Delta \vec \Phi_k$ together with \eqref{conservative system} to see that
\begin{align}
2 e^{2\lambda_k} \vec H_k &= \Delta \vec \Phi_k = -  \nabla^\perp S_k\cdot \nabla \vec \Phi_k - \nabla^\perp \vec R_k \times \nabla \vec \Phi_k -  \nabla \vec \Phi_k \cdot \nabla Y_k - \frac{\beta_k}{2} |\vec \Phi_k|^2 e^{2\lambda_k} \vec n_k. \label{putting everything together}
\end{align}
We divide both sides by $e^{\lambda_k}$. Using \eqref{strong L21 estimate for S and Rk}, \eqref{strong nabla Yk estimate}, \eqref{H L21 lemma L^p estimate} and \eqref{H L21 Lemma assumptions}, we conclude
\[\|e^{\lambda_k} \vec H_{k}\|_{L^{2,1}(B_{2^{-7}}\setminus B_{2^{7}r_k})} \leq C.\]
%
%
%
%
%
%
This finishes the proof in the case where $d>-1+\delta$. Let us now consider the case $d<-1-\delta$ where we work with the reflected immersion $\hat{\vec{\Phi}}_k$. The main issue is that we need to show $L^{2,\infty}$-bounds for $\nabla \hat Y_k$, where $\hat Y_k(z) = Y_k(r_k/z)$. We cannot directly apply \cite[Theorem 3.3.6]{Helein} as in Case 1, as the domain is now unbounded.
\subsubsection*{Part 8. Adressing the second case}\label{part 8}
We now assume that $d<-1-\delta$. As before, the quantities in the inverted setting are denoted with a hat. We set $\hat f_k \cqq e^{2 \hat\lambda_k}(-2\alpha_k  + \beta_k \langle  \hat{\vec{\Phi}}_k, \hat{\vec{n}}_k\rangle- \gamma_k  \hat H_k)$ which satisfies
\begin{equation}
\hat f_k(z) = \frac{r_k^2}{|z|^4} f_k(r_k/z).\label{reflected rhs for Y}
\end{equation}
We define $\hat Y^2_k$ to be the solution to
\begin{equation}
\begin{cases}
\Delta \hat Y^2_k =  \hat f_k&\quad \text{in }B_{2^{-3}}\setminus B_{2^3r_k}, \\
\;\;\;\hat Y^2_k=0& \quad \text{on } \partial (B_{2^{-3}}\setminus B_{2^3r_k}),
\end{cases} \label{hat Y1k definition}
\end{equation} 
Notice that as $d<-1-\delta$, $\hat f_k$ now satisfies the uniform $L^{p_2}$-estimates
\begin{equation}
\|\hat f_k\|_{L^{p_2}(B_{2^{-3}}\setminus B_{2^3r_k})}\leq C,\label{hat fk Lp estimate}
\end{equation}
which is the analogue to \eqref{fk Lp estimate}. As in \Cref{lem:L21 estimate for Poisson in divergence form} \ref{lem:L21 estimate Poisson second part}, see in particular \eqref{L2 estimate for phi2}, it holds
\begin{equation}
\|\nabla \hat Y^2_k\|_{L^2(B_{2^{-3}})} \leq C \|\hat f_k\|_{L^{p_2}(B_{2^{-3}}\setminus B_{2^3r_k})}\leq C.\label{nabla hat Y1k L2 estimate 2}
\end{equation}
We define $Y^2_k(z) = \hat Y^2_k(r_k/z)$, which then solves
\begin{equation}
\begin{cases}
\Delta Y^2_k =   f_k&\quad \text{in }B_{2^{-3}}\setminus B_{2^3r_k}, \\
\;\;\; Y^2_k=0& \quad \text{on } \partial (B_{2^{-3}}\setminus B_{2^3r_k}).
\end{cases} \label{Y1k definition}
\end{equation} 
As the Dirichlet energy is conformally invariant, we obtain from \eqref{nabla hat Y1k L2 estimate 2}
\begin{equation}
\|\nabla Y^2_k\|_{L^{2}(B_{2^{-3}}\setminus B_{2^3r_k})} = \|\nabla \hat Y^2_k\|_{L^{2}(B_{2^{-3}}\setminus B_{2^3r_k})}\leq C.\label{nabla Y1k L2 estimate}
\end{equation}
Recall from \eqref{L2 weak estimate for Yk} that $Y_k$ defined in \eqref{Recalling Yk} satisfies
\[\|\nabla Y_k\|_{L^{2,\infty}(B_{2^{-3}})} \leq C \|f_k\|_{L^1(B_{2^{-3}})} \leq C.\]
Let $v_k \cqq Y_k - Y^2_k$, which thus satisfies
\begin{equation}
\begin{cases}
\Delta v_k = 0\quad &\text{in $B_{2^{-3}}\setminus B_{2^3r_k}$},\\
\;\;\;v_k = 0 \quad&\text{on $\partial B_{2^{-3}}$},
\end{cases}  \quad \|\nabla v_k\|_{L^{2,\infty}(B_{2^{-3}}\setminus B_{2^3r_k})}\leq C.
\end{equation}
The uniform bound $\|f_k\|_{L^1(B_1)}\leq C$ implies for $\rho \in (2^3r_k, 2^{-3})$
\begin{equation}
\left |\int _{\partial B_{\rho}} \partial_r Y_k \dif{l}\right | =\left | \int _{B_{\rho}} f_k \dif{x}\right | \leq C.\label{line integral Yk}
\end{equation}
Notice further that since $Y^2_k = 0$ on $\partial B_{2^{-3}}$ and $\partial B_{2^3r_k}$, there is some $\rho_k \in (2^3r_k, 2^{-3})$ such that
\begin{equation}
 0 = \partial_\rho \bigg\vert _{\rho_k} \dashint_{\partial B_{\rho}} Y^2_k \dif{l} =  \dashint_{\partial B_{\rho_k}} \partial_r Y^2_k \dif{l} = \int_{\partial B_{\rho_k}} \partial_r Y^2_k \dif{l}.\label{mean value argument for line integral}
\end{equation}
From \eqref{line integral Yk} and \eqref{mean value argument for line integral}, we see that $\eta_k \cqq \frac{1}{2\pi}\int _{\partial B_{\rho_k}} \partial_r v_k \dif{l}$ satisfies $|\eta_k|\leq C$. Using \Cref{lem:harmonic function L2 weak to L2 strong} applied to $v_k(x) - \eta_k \log |x|$, we see
\begin{equation}
\left \|\nabla \left (v_k - \eta_k \log |\cdot| \right )\right \|_{L^2(B_{2^{-6}}\setminus B_{2^6 r_k})}\leq C\left \|\nabla \left (v_k - \eta_k \log |\cdot | \right )\right \|_{L^{2,\infty}(B_{2^{-3}}\setminus B_{2^3r_k})}\leq C.\label{L2 estimate for harmonic part uk minus log term}
\end{equation}
Using again the conformal invariance of the Dirichlet energy and the fact that the $L^{2,\infty}$-norm of the gradient of the logarithm is invariant under reflections, we deduce for the reflected harmonic map $\hat v_k(z) = v_k(r_k/z)$ that
\begin{equation}
\| \nabla \hat v_k\|_{L^{2,\infty}(B_{2^{-6}}\setminus B_{2^6 r_k})} \leq C\| \nabla \log |\cdot |\|_{L^{2,\infty}(B_{2^{-3}}\setminus B_{2^3r_k})} + \left \|\nabla \left (v_k - \eta_k \log |\cdot | \right )\right \|_{L^2(B_{2^{-6}}\setminus B_{2^6 r_k})} \leq C.\label{nabla hat uk L2 estimate}
\end{equation}
Finally, as $v_k = Y_k - Y^2_k$, it follows that $\hat Y_k = \hat Y^2_k + \hat v_k$, and so we obtain
\begin{equation}
\|\nabla \hat Y_k\|_{L^{2,\infty}(B_{2^{-6}}\setminus B_{2^6 r_k})} \leq C.\label{nabla hat Yk L2 weak estimate}
\end{equation}
This is the equivalent to \eqref{L2 weak estimate for Yk}. From here, the rest of the proof can be finished with the same arguments as in the first case, starting from \eqref{Yk alternative definition}. Notice that
\[\int _{\partial B_{\rho}} \partial _r \hat Y_k \dif l = - \int _{\partial B_{\frac{r_k}{\rho}}} \partial_r Y_k\dif l\]
and so in particular $\int _{\partial B_{\sqrt{r_k}}} \partial_r \hat Y_k \dif l = -\mc B_k$. Thus, we can prove \eqref{l21 estimate for H} for $\hat{\vec{\Phi}}_k$ analogously.
\end{proof}

\section{\texorpdfstring{$L^2$}{L2}-weak quantization of the Gauss map in neck regions}\label{sec:L2weak quantization in neck regions}

\begin{lemma}[$L^2$-weak energy quantization for the Gauss map in neck regions]\label{lem: L2 weak estimate for gauss map annulus has small energy condition}
Let $\Lambda > 0$. There exists $\eps_4 = \eps_4(\Lambda)>0$ depending only on $\Lambda$ with the following property. Let $\vec \Phi_k$ be a sequence of conformal constrained Willmore immersions with coefficients $\alpha_k,\beta_k$, $\gamma_k$, mapping from $B_{1} $ into $\R^3$. Suppose $r_k\searrow 0$ is such that
\[\sup _{r_k<s<1/2} \int _{B_{2s} \setminus B_s } |\nabla \vec n _{k}|^2 \dif{x} \leq \eps_4(\Lambda),\]
and suppose that
\begin{equation}
\|\nabla \lambda _k \|_{L^{2,\infty}(B_{1} )} + \int _{B_{1} } |\nabla \vec n _{k}|^2 \dif{x} +|\alpha_k| + |\beta_k| + |\gamma_k| + \mc A(\vec \Phi_k)+ \|\vec \Phi_k\|_{L^\infty(B_{1})} \leq \Lambda \label{lem: various stuff bound by Lambda}
\end{equation}
for all $k\in \N$. Then for all $\eps \in (0,\eps_4(\Lambda))$, there exists $\alpha \in (0,1)$ such that
\begin{equation}
\limsup_{k\to \infty} \||x| \nabla \vec n_{k}\|_{L^\infty(B_{\alpha } \setminus B_{\alpha ^{-1}r_k} )}\leq \eps.\label{x times nabla n is pointwise bounded by eps}
\end{equation}
In particular,
\begin{equation}
\limsup_{k\to \infty} \|\nabla \vec n_{k}\|_{L^{2,\infty}(B_{\alpha} \setminus B_{\alpha ^{-1}r_k} )} \leq C \eps.\label{implied L2 weak bound for nabla n by eps}
\end{equation}
\begin{proof}
We adapt the proof in \cite[Lemma VII.1]{BernardRiviere}. The main idea here is that due to \Cref{lem:Lp lemma}, the constraints for the rescaled immersions vanish in the limit, which then lets us use the same methods as in \cite[Lemma VII.1]{BernardRiviere}. Similar to the proof of \Cref{lem: L21 control mean curvature}, we may choose $\eps_4$ smaller than $\eps_1$ from \Cref{thm:eps regularity} to deduce that
\begin{equation}
|\nabla \vec n _{k}(x)|^2 \leq C|x|^{-2} \int _{B_{2|x|} \setminus B_{|x|/2} } |\nabla \vec n _{k}|^2 \dif{x} \leq C\eps_4 |x|^{-2} \label{pointwise estimate nabla n}
\end{equation}
for all $x\in B_{1/2} \setminus B_{2r_k} $. In particular,
\begin{equation}
\|\nabla \vec n_{k}\|_{L^{2,\infty}(B_{1/2} \setminus B_{2r_k} )} \leq C \sqrt{\eps_4}.\label{small L2 weak estimate nabla n}
\end{equation}
We assume that there exists a sequence of constrained Willmore immersions $\vec \Phi_k$ on $B_{1} $ satisfying the assumptions of \Cref{lem: L2 weak estimate for gauss map annulus has small energy condition}, but for which we find $\eps'>0$ and $x_k \in B_1 \setminus B_{r_k}$ such that
\begin{equation}
\frac{|x_k|}{r_k} \to \infty, \quad|x_k| \to 0\quad\text{as $k\to \infty$}\label{xk away from rk and Rk}
\end{equation}
and
\[|x_k| |\nabla \vec n_{k} (x_k)|\geq \eps'>0.\]
It follows from \eqref{pointwise estimate nabla n} that
\begin{equation}
\int _{B_{2|x_k|} \setminus B_{|x_k|/2} } |\nabla \vec n _{k} |^2 \dif{x} \geq  \frac{{\eps'}^2}{C}>0.\label{some energy in annulus}
\end{equation}
After a potential shift, we may assume that $\vec \Phi_k(x_k)=0$ for all $k\in \N$. Let us first show that the conformal factors around $x_k$ are small. More precisely, we show
\begin{equation}
\lambda_k(x_k) + \ln|x_k| \to -\infty \quad\text{ as $k\to \infty$.} \label{lambda k plus ln xk diverges}
\end{equation}
Using \Cref{lem: original pointwise lambda estimate}, which we may apply upon choosing $\eps_4$ such that $C \sqrt{\eps_4} $ is smaller than $ \eps_2$, together with \Cref{rem:mean value argument}, there are constants $A_k$ and $d_k$ such that
\begin{equation}
\|\lambda_k(x) - d_k\log |x| -A_k\|_{L^\infty(B_{1/4}\setminus B_{4r_k})}\leq C \quad \text{and}\quad |d_k| \leq C. \label{logarithmic estimate conformal factor L2 weak lemma}
\end{equation}
We obtain that
\begin{equation}
\mc A(\vec \Phi_k) \geq C e^{2A_k} \int_{4r_k}^{1/4} r^{2d_k+1}\dif{r} =C\frac{e^{2A_k}}{d_k+1}\left (4^{-2d_k-2} -(4r_k)^{2d_k+2}\right ).
\end{equation}
In particular,
\begin{equation}
A_k \leq C \quad \text{and}\quad A_k + (d_k+1)\ln(r_k)\leq C.\label{area bounds at the ends of the annulus}
\end{equation}
We apply \Cref{lem:Lp lemma} to deduce that $d_k \to d$ with $|d+1| > \delta$ for some $\delta = \delta (\Lambda)>0$. Suppose first that $d>-1+\delta$. Using \eqref{logarithmic estimate conformal factor L2 weak lemma}, \eqref{area bounds at the ends of the annulus}, \eqref{xk away from rk and Rk}, and $d_k+1 \to d+1 > \delta>0$, we see
\begin{equation}
e^{\lambda_k(x_k) + \ln|x_k|} \leq C |x_k|^{d_k+1} e^{A_k} \leq C |x_k|^{d_k+1} \to 0.\label{conformal factor at xk diverges in case 1}
\end{equation}
If $d<-1-\delta$, then we have
\begin{equation}
e^{\lambda_k(x_k) + \ln |x_k|} \leq C |x_k|^{d_k+1}e^{A_k} \leq C|x_k|^{d_k+1} r_k^{-(d_k+1)}\to 0.\label{conformal factor at xk diverges in case 2}
\end{equation}
\eqref{conformal factor at xk diverges in case 1} and \eqref{conformal factor at xk diverges in case 2} imply \eqref{lambda k plus ln xk diverges}.\medskip

We also choose $\eps_4$ smaller than $\eps_3$ from \Cref{lem: L21 control mean curvature}. Since all the assumptions are satisfied, we may apply \Cref{lem: L21 control mean curvature} to the sequence $\vec \Phi_k$ and apply all the ingredients from its proof. In the case that $d<-1-\delta$, we do the following arguments with $\hat{\vec{\Phi}}_k$ instead of $\vec{\Phi}_k$ and $x_k$ replaced by $r_k x_k^{-1}$, with all arguments being the same. 

From \eqref{Lk pointwise estimate} and \eqref{strong L21 estimate for S and Rk} we have bounds on the quantities $\vec L_{k},  S_{k}$ and $\vec R_{k}$:
\begin{equation}
\|e^{\lambda_k} \vec L_k\|_{L^{2,\infty}(B_{2^{-3} }\setminus B_{2^3r_k})} + \|\nabla S_k\|_{L^{2,1}(B_{2^{-7}}\setminus B_{2^7r_k})} +\|\nabla \vec R_k\|_{L^{2,1}(B_{2^{-7}}\setminus B_{2^7r_k})}  \leq C. \label{L2 weak lemma: L2 estimates L S and R}
\end{equation}
We consider the map $\tilde{\vec {\Phi}}_k$ defined by
\begin{equation}
\tilde{\vec {\Phi}}_k(y) \cqq e^{-\lambda_k(x_k) - \ln|x_k|} \vec \Phi_k(|x_k|y).\label{rescaled map xi k}
\end{equation}
All other quantities corresponding to $\tilde{\vec{\Phi}}_k$ are also denoted by a tilde. This is still a constrained Willmore immersion with coefficients $\tilde{\alpha}_k,\tilde{\beta}_k$ and $\tilde{\gamma}_k$. The scaling behavior of these coefficients is given by \eqref{rescalings of constants}, namely
\begin{align*}
\tilde{\alpha}_k &= e^{2\lambda_k(x_k) + 2\ln |x_k|} \alpha_k,\quad
\tilde{\beta}_k = e^{3\lambda_k(x_k) + 3\ln |x_k|} \beta_k,\quad
\tilde{\gamma}_k = e^{\lambda_k(x_k) + \ln |x_k|} \gamma_k.
\end{align*}
With the uniform bound \eqref{lem: various stuff bound by Lambda} for $\alpha_k$, $\beta_k$ and $\gamma_k$ and the divergence of the rescaling factor, see \eqref{lambda k plus ln xk diverges}, we see
\begin{equation}
|\tilde{\alpha}_k| + |\tilde{\beta}_k| + |\tilde{\gamma}_k| \to 0 \quad\text{as $k\to \infty$}.\label{rescaled constants go to 0}
\end{equation}
From \eqref{rescaled map xi k}, the gradient of the normal vector and the mean curvature scale as
\begin{equation}
\nabla _y \tilde{\vec{n}}_k(y) = |x_k| \nabla_x \vec n _{k}(|x_k|y)\quad \text{and}\quad\nabla_y \tilde{\vec{H}}_{k} (y) = e^{\lambda_k(x_k) + \ln |x_k|} |x_k| \nabla_x\vec{H}_{k} (|x_k|y).\label{rescaling of nabla n and nabla H}
\end{equation}
Notice further that the map $\vec T$, defined in \eqref{definition T}, scales as
\[\tilde{\vec{T}}_k(y) = e^{\lambda_k(x_k) + \ln |x_k|} |x_k| \vec T(|x_k| y). \]
The corresponding map $\tilde{\vec {L}}_k$, defined by \eqref{definition L R and S}, thus satisfies
\begin{equation}
\nabla _y \tilde{\vec {L}}_k(y) = e^{\lambda_k(x_k)  + \ln |x_k|} |x_k| \nabla_x \vec L_k (|x_k| y)\quad\text{so that}\quad\tilde{\vec {L}}_k(y) = e^{\lambda_k(x_k) + \ln |x_k|} \vec L_k(y).\label{rescaling of Lk}
\end{equation}
For the conformal factor, it holds
\begin{equation}
\tilde{\lambda}_k(y) = \lambda_k(|x_k|y) - \lambda_k(x_k).\label{rescaled conformal factor} 
\end{equation}
$Y_k$ is defined on $B_{2^{-3}}$ as in \eqref{definition Y}. The quantities $\vec X_k$ and $Y_k$ scale as
\begin{equation}
\tilde{\vec{X}}_k(y) = |x_k| \vec{X}_k(|x_k|y),\quad \tilde{Y}_k(y) = Y_k(|x_k|y).\label{scaling of Xk and Yk}
\end{equation}
From \eqref{rescaling of nabla n and nabla H}, \eqref{rescaling of Lk}, and \eqref{scaling of Xk and Yk}, the quantities $S_k$ and $\vec R_k$ are scaling invariant according to \eqref{definition L R and S}, namely
\begin{equation}
\tilde{S_k}(y) = S_k(|x_k|y) \quad \text{and}\quad \tilde{\vec{ R}}_k(y) = \vec R_k(|x_k|y).\label{rescaling of Sk and Rk}
\end{equation}
The $L^{2}$-norm of the gradient is invariant under dilations of the domain. Thus, \eqref{L2 weak lemma: L2 estimates L S and R} and \eqref{rescaling of Sk and Rk} show
\begin{equation}
\limsup_{k\to \infty} \Big [\|\nabla \tilde{S_k}\|_{L^{2,1}(K)} +\|\nabla \tilde{\vec{ R}}_k\|_{L^{2,1}(K)} \Big ]\leq C\label{tilde Sk and tilde Rk are bounded in L2}
\end{equation}
for any $K\subset \subset \C\setminus \{0\}$, where the constant $C$ does not depend on $K$. Furthermore, by \eqref{pointwise estimate nabla n}, \eqref{rescaled conformal factor}, and \eqref{logarithmic estimate conformal factor L2 weak lemma},
\begin{equation}
\limsup_{k\to \infty} \|\nabla \vec n _{\tilde{\vec{\Phi}}_k} \|_{L^{\infty}(K)} < C_K \quad \text{and}\quad\limsup_{k\to \infty} \| \tilde{\lambda}_k \|_{L^{\infty}(K)} <C_K.\label{gradient n and conformal factor are locally bounded}
\end{equation}
The $\eps$-regularity, \Cref{thm:eps regularity} shows that for any $l\in \N$, $\tilde{\vec{\Phi}}_k$ converges up to subsequences in $C^l_{\loc}(\C \setminus \{0\})$ to some map $\tilde{\vec{\Phi}}_\infty$, which, by \eqref{rescaled constants go to 0}, is an unconstrained, conformal Willmore immersion. Furthermore, sending \eqref{some energy in annulus} to the limit yields
\begin{equation}
\int _{B_{2} \setminus B_{1/2} } |\nabla \vec n _{\tilde{\vec{\Phi}}_\infty} |^2 \dif{x} \geq \frac{{\eps'}^2}{C}.\label{some energy in annulus in the limit}
\end{equation}
The $\eps$-regularity and \eqref{rescaled constants go to 0} directly yield that 
\begin{equation}
(\tilde{\vec{\Phi}}_k,\tilde{\vec{T}}_k, \tilde{\vec{X}}_k, \tilde{\vec{L}}_k) \to (\tilde{\vec{\Phi}}_\infty,0,0, \tilde{\vec{L}}_\infty)\quad \text{in $C^l(K)$ for $K\subset \subset \C \setminus \{0\}$ and all $l\in \N$.}\label{things passing to the limit}
\end{equation}
We will now obtain similar estimates for $\tilde{Y}_k$. Recall the harmonic function $u_k$ from \eqref{harmonic representation on annulus} satisfies \eqref{uk bounded} and in particular, $\tilde{u}_k(y) \cqq u_k(|x_k| y)$ satisfies
\begin{equation}
\limsup _{k\to \infty}\|\tilde u_k\|_{L^\infty(K)}\leq C \label{tilde u2k bounded}
\end{equation}
where $K\subset \subset \C\setminus \{0\}$ and the constant does not depend on $K$. Up to subsequence, $u_k \wto u_\infty$ in $W^{1,2}(K)$, thus pointwise a.e. up to a subsequence. \eqref{tilde u2k bounded} shows that $u^2_\infty$ is bounded and harmonic in $\C\setminus \{0\}$ and thus a constant $c_0$ by Liouville's theorem. Fix $s>0$. We deduce that $\lim _{k\to \infty}\|\tilde u^2_k-c_0\|_{L^2(B_{2^3s}\setminus B_{2^{-3}s})}=0$. The Poisson representation formula, see also \Cref{rem:various things} \ref{rem:typical argument harmonic functions}, implies 
\begin{equation}
\lim _{k\to \infty}\|\nabla \tilde u_k\|_{L^{2}(B_{4s}\setminus B_{s/4})} =0\label{harmonic part tilde uk L2 estimate}
\end{equation}
From \eqref{strong nabla Yk estimate} and $\lim _{k\to \infty }|x_k|= 0$, see \eqref{xk away from rk and Rk}, we deduce that the rescaled maps $\tilde{Y}^1_k(y) = Y^1_k(|x_k| y)$, where $Y^1_k$ was defined in \eqref{Yk alternative definition}, satisfy
\begin{align}
\|\nabla \tilde{Y}^1_k\|_{L^{2}(B_{4s}\setminus B_{s/4})} &= \|\nabla Y^1_k\|_{L^{2}(B_{4|x_k|s}\setminus B_{|x_k|s/4})} \notag\\
&\leq   \mc L^2(B_{4|x_k|s}\setminus B_{|x_k|s/4})^{\frac{1}{2} - \frac{2-p_2}{2p_2}} \|\nabla Y^1_k\|_{L^{\frac{2p_2}{2-p_2}}(B_{4|x_k|s}\setminus B_{|x_k|s/4})}\notag\\
 &\leq C(s|x_k|)^{2-\frac{2}{p_2}}\notag\\
 &\to 0\quad\text{as $k\to \infty$.}\label{L2 estimate for tilde Y1k}
\end{align}
From \eqref{harmonic part tilde uk L2 estimate} and \eqref{L2 estimate for tilde Y1k}, it follows 
\begin{equation}
\lim _{k\to\infty}\|\nabla \tilde Y_k\|_{L^{2}(B_{4s}\setminus B_{s/4})}= 0.\label{L2 estimate for tilde Yk}
\end{equation}
Consider $\tilde{Y}^2_k$ solving
\[\begin{cases}
\Delta \tilde{Y}^2_k = \Delta \tilde{Y}_k &\quad\text{in $B_{4s}\setminus B_{s/4}$,}\\
\;\;\; \tilde{Y}^2_k = 0 &\quad \text{on $\partial(B_{4s}\setminus B_{s/4})$.}
\end{cases}\]
As $\Delta \tilde{Y}_k = e^{2\tilde{\lambda}_k}(-2\tilde{\alpha}_k + \tilde{\beta}_k \langle \tilde{\vec{\Phi}}_k, \tilde{\vec{n}}_k\rangle - \tilde{\gamma}_k \tilde{H}_k)$, we deduce from \eqref{things passing to the limit}, \eqref{rescaled constants go to 0}, and standard regularity theory that $\tilde{Y}^2_k \to 0$ in $C^l(B_{4s}\setminus B_{s/4})$. Using again \Cref{rem:various things} \ref{rem:typical argument harmonic functions}, \eqref{L2 estimate for tilde Yk} shows
\[\nabla \tilde{Y}^2_k-\nabla \tilde{Y}_k \to 0 \quad \text{in $C^l(B_{2s}\setminus B_{s/2})$}.\]
The choice of $s$ was arbitrary, implying
\begin{equation}
\nabla \tilde{Y}_k \to 0 \quad \text{in $C^l_{\loc}(\C\setminus \{0\})$}\label{Yk to 0 smoothly}
\end{equation}
for all $l\in\N$. We may shift $S_k$ and $\vec R_k$ such that $S_k(x_k)=0$ and $\vec R_k(x_k)=0$. In particular, \eqref{definition L R and S} shows that up to a subsequence
\begin{equation}
\begin{aligned}
\tilde{S}_k &\to \tilde{S}_{\infty}\quad\text{in $C^l_{\loc}(\C\setminus \{0\})$,}\\
\tilde{\vec{R}}_k &\to \tilde{\vec{R}}_\infty \quad\text{in $C^l_{\loc}(\C\setminus \{0\})$.}
\end{aligned} 
\end{equation}
Notice further that \eqref{modified supremum estimate for R} and \eqref{modified supremum estimate for S} imply
\begin{equation}
\|\tilde{S}_\infty\|_{L^{\infty}(\C)}+\|\tilde{\vec{R}}_\infty\|_{L^{\infty}(\C)}\leq C .\label{L infinity estimate for S and R infinity}
\end{equation}
Recall that $\tilde{\vec{\Phi}}_k$ satisfies the conservative system
\begin{equation}
\begin{cases}
\Delta \tilde{\vec{R}}_k&=   \nabla ^\perp \tilde{\vec{n}}_k \cdot \nabla \tilde{S}_k + \nabla ^\perp \tilde{\vec{n}}_k\times \nabla \tilde{\vec{R}}_k  + \Div [\nabla \tilde{Y}_k \tilde{\vec{n}}_k ],\\
\Delta \tilde{S}_k &= \langle \nabla ^\perp \tilde{\vec{n}}_k , \nabla \tilde{\vec{R}}_k\rangle, \\
\Delta \tilde{\vec{\Phi}}_k &= -  \nabla ^\perp \tilde{S}_k\cdot \nabla \tilde{\vec{\Phi}}_k - \nabla ^\perp \tilde{\vec{R}}_k \times \nabla \tilde{\vec{\Phi}}_k - \frac{\tilde{\beta}_k}{2} |\tilde{\vec{\Phi}}_k|^2 e^{2\tilde{\lambda}_k } \tilde{\vec{n}}_k.
\end{cases}
\end{equation}
By \eqref{things passing to the limit} and \eqref{Yk to 0 smoothly}, the conservative system passes to the limit and all terms involving the constraints vanish:
\begin{equation}
\begin{cases}
\Delta \tilde{\vec{R}}_\infty&=   \nabla ^\perp \vec n_{\infty} \cdot \nabla \tilde{S}_\infty + \nabla ^\perp \vec n_{\infty} \times \nabla \tilde{\vec{R}}_\infty ,\\
\Delta \tilde{S}_\infty &= \langle \nabla ^\perp \vec n_{\infty} , \nabla \tilde{\vec{R}}_\infty\rangle, \\
\Delta \tilde{\vec{\Phi}}_\infty &= -  \nabla ^\perp \tilde{S}_\infty\cdot \nabla \tilde{\vec{\Phi}}_\infty - \nabla ^\perp \tilde{\vec{R}}_\infty \times \nabla \tilde{\vec{\Phi}}_\infty .
\end{cases} \label{conservative system in limit}
\end{equation}
This system a priori only holds in $\C \setminus \{0\}$. However, \eqref{tilde Sk and tilde Rk are bounded in L2} shows that
\[\|\nabla \tilde{S}_{\infty}\|_{L^{2,1}(\C)} + \|\nabla \tilde{\vec{R}}_{\infty}\|_{L^{2,1}(\C)}\leq C\]
and since $\nabla \vec n _{\infty} \in L^2(\C)$ as well, the conservative equation for $ \tilde{S}_{\infty}$ and $\tilde{\vec{R}}_\infty$ extend through $0$.
\eqref{pointwise estimate nabla n} passes to the limit as well, so $\|\nabla \vec n_{\infty}\|_{L^{2,\infty}(\C)}\leq C \sqrt{\eps_4}$. We use the argument from \Cref{rem:various things} \ref{rem:typical argument harmonic functions} to apply \cite[Theorem 3.4.5]{Helein} to $\tilde{\vec{R}}_{\infty}$ and $\tilde{S}_{\infty}$ on some ball $B_{s_2}$, where $s_2>0$. The resulting harmonic terms $\nu^S$ and $\vec \nu^R$ satisfy by \eqref{harmonic part is good inside half ball}
\[\|\nabla \nu^S\|_{L^\infty(B_{s_2/2})}+\|\nabla \vec \nu^R \|_{L^\infty(B_{s_2/2})} \leq \frac{C}{s_2}\]
and hence for $0<s_1<s_2/2$, it holds 
\[\|\nabla \nu^S\|_{L^{2,1}(B_{s_1})} + \|\nabla \vec \nu^R \|_{L^{2,1}(B_{s_1})} \leq C\frac{s_1}{s_2}.\]
Sending $s_2\to \infty$, we see
\begin{align}
\|\nabla \tilde{S}_{\infty}\|_{L^{2,1}(B_{s_1})} + \|\nabla \tilde{\vec{R}}_{\infty}\|_{L^{2,1}(B_{s_1})} &\leq \liminf _{s_2\to \infty}C \sqrt{\eps_4} \left (\|\nabla \tilde{S}_{\infty}\|_{L^{2,1}(B_{s_2})} + \|\nabla \tilde{\vec{R}}_{\infty}\|_{L^{2,1}(B_{s_2})}\right ) + C \frac{s_1}{s_2}   \notag\\
& \leq C \sqrt{\eps_4} \left (\|\nabla \tilde{S}_{\infty}\|_{L^{2,1}(\C)} + \|\nabla \tilde{\vec{R}}_{\infty}\|_{L^{2,1}(\C)}\right ).\label{Estimate L2 for S infty and R infty}
\end{align}
\eqref{Estimate L2 for S infty and R infty} holds for any $s_1>0$, and so
\begin{equation}
\|\nabla \tilde{S}_{\infty}\|_{L^{2,1}(\C)} + \|\nabla \tilde{\vec{R}}_{\infty}\|_{L^{2,1}(\C)} \leq C \sqrt{\eps_4} \left (\|\nabla \tilde{S}_{\infty}\|_{L^{2,1}(\C)} + \|\nabla \tilde{\vec{R}}_{\infty}\|_{L^{2,1}(\C)}\right ). \label{nabla S and nabla R have to be zero}
\end{equation}
Choosing $\eps_4$ sufficiently small such that $C \sqrt{\eps_4} < 1$ then yields that $\nabla \tilde{S}_{\infty}=0$ and $\nabla \tilde{\vec{R}}_{\infty}=0$. Inserting this into \eqref{conservative system in limit} then yields $2 e^{2\tilde{\lambda}_{\infty}} \tilde{\vec{H}}_\infty = \Delta \tilde{\vec{\Phi}}_{\infty} = 0$ and so $\tilde{\vec{\Phi}}_{\infty}$ is the conformal immersion of a complete \emph{minimal} surface. From here, we can finish the proof exactly as in \cite[Lemma VII.1]{BernardRiviere} and bring \eqref{some energy in annulus in the limit} to a contradiction.
\end{proof}
\end{lemma}
\section{Proof of main results}\label{sec:Proof main result}
\begin{proof}[Proof of \Cref{thm:Energy quantization} and \Cref{thm:continuation of energy quantization}]
Suppose $\vec \Phi_k:\Sigma \to \R^3$ denotes a sequence of constrained Willmore immersions with coefficients $\alpha_k,\beta_k,\gamma_k$, area $\mc A(\vec \Phi_k)$ and Willmore energy $\mc W(\vec \Phi_k)$ satisfying
\[\Lambda_1 \cqq \limsup_{k\to \infty} \Big[\mc W(\vec \Phi_k) + \mc A(\vec \Phi_k) +|\alpha_k| + |\beta_k| + |\gamma_k|\Big] < \infty.\]
We may assume that $\mc A(\vec \Phi_k)=1$ for all $k$ and $0\in \vec \Phi_k(\Sigma)$. Up to a subsequence, we may assume that the constraints $\alpha_k$, $\beta_k$, and $\gamma_k$ converge. We assume that the conformal structures induced by the metric $g_k = \vec \Phi_k^*g_{\R^3}$ remain in a compact subset of the moduli space. Using \cite[Corollary 4.4]{RiviereLectureNotes}, we find Lipschitz diffeomorphisms $f_k$ of $\Sigma$ such that $\vec \xi_k \cqq \vec \Phi_k \circ f_k$ are conformal as maps from $(\Sigma, h_k)$ into $\R^3$, where we denote by $h_k$ the constant scalar curvature metric of unit volume associated to $g_{k}$. These converge to a limiting metric $h_\infty$ in $C^l(\Sigma)$ for all $l\in \N$. We denote by $\tilde{\lambda}_k$ the conformal factor induced by $\vec \xi_k$, i.e.,
\[\vec \xi_k^* g_{\R^3} = g_k = e^{2\tilde{\lambda}_k} h_k.\]
By \cite[Lemma 1.1]{Simon}, we know that 
\[\sqrt{\frac{\mc A(\vec \xi_k)}{\mc W(\vec \xi_k)}}\leq \diam \vec \xi_k(\Sigma) \leq C\sqrt{\mc A(\vec \xi_k) \mc W(\vec \xi_k)} \]
and so
\begin{equation}
\vec \xi_k(\Sigma) \subset B_{C\sqrt{\Lambda_1}} \label{bounded area and bounded diameter}
\end{equation}
for all $k\in \N$. Let $\eps_0>0$ be the constant from \Cref{thm:eps regularity}. A Besicovitch covering argument, as done in \cite{RiviereCrelle}, shows that there exist finitely many energy concentration points $a_1,\ldots, a_n\in \Sigma$ away from which Dirichlet energy of at least $\eps_0$ accumulate. More precisely, for $x\not \in \{a_1,\ldots, a_n\}$, there is $\rho(x)>0$ such that
\begin{equation}
\limsup_{k\to \infty} \int_{B_{\rho(x)}(x)} |d \vec n _{\vec \Phi_k}|^2 _{h_k} \dif{\mu _{h_k}} < \eps_0,\label{radius for which no concentration happens}
\end{equation}
where balls in $\Sigma$ are measured with respect to $h_\infty$. By \cite[Theorem 5.4]{RiviereLectureNotes} it holds that 
\begin{equation}
\sup_{k\in \N}\|d \tilde{\lambda}_k\|_{L^{2,\infty}(\Sigma, h_k)} < \infty. \label{weak bound for gradients of conformal factor}
\end{equation}
Furthermore, by \cite[Lemma 3.3]{RiviereCrelle}, there are constants $C_k$ such that
\[\|\tilde{\lambda}_k - C_k\|_{L^{\infty}\left (\Sigma \setminus \bigcup _{i=1}^n B_{\eps}(a_i)\right )} \leq C_\eps \quad \text{for all $\eps>0$.}\]
It follows from the bounded area that $\limsup_{k\to \infty} C_k < \infty$ (however, $\lim_{k\to \infty} C_k = -\infty$ is not excluded by this).  Choose $z_0 \in \Sigma \setminus \bigcup _{i=1}^n B_{\delta}(a_i)$, where $\delta>0$ is sufficiently small such that this set is non-empty. Owing to the $\eps$-regularity theorem, \Cref{thm:eps regularity}, applied in local converging coordinates in the balls $B_{\rho(x)}(x)$ from \eqref{radius for which no concentration happens}, the rescaled immersions $e^{-C_k} (\vec \xi_k(\cdot) - \vec \xi_k(z_0))$ converge up to a subsequence in $C^l_{\loc}(\Sigma \setminus \{a_1,\ldots, a_n\})$ to a limiting immersion $\vec \zeta_\infty$ which is a constrained Willmore surface on $\Sigma \setminus \{a_1,\ldots, a_n\}$. Around the $a_i$, both branch points and ends may form. If an end forms anywhere, i.e.,
\[\mc A(\vec{\zeta}_\infty) = \infty,\]
this means that $C_k \to -\infty$ as $\limsup_{k\to \infty} e^{-2C_k}\mc A(\vec \xi_k) \leq \Lambda e^{-2\liminf_{k\to \infty}C_k}$ and that $\vec{\zeta}_\infty$ is an unconstrained Willmore immersion on $\Sigma \setminus \{a_1,\ldots, a_n\}$. Here, the fact that $\vec{\zeta}_\infty$ is unconstrained follows from the rescaling of the coefficients in \eqref{rescalings of constants}. 

Fix $i\in \{1,\ldots, n\}$. Choose conformal, converging coordinates $\phi_k:B_1 \to (\Sigma, h_k)$ containing some fixed neighborhood around $a_i$. By an abuse of notation, we will denote the maps $\vec \xi_k \circ \phi_k$ again by $\vec \xi_k$ and the conformal factor by $\lambda_k$. Then \eqref{weak bound for gradients of conformal factor} shows 
\begin{equation}
\Lambda_2 \cqq\sup _{k\in \N} \|\nabla \lambda_k\|_{L^{2,\infty}(B_1)} < \infty.
\end{equation}
Let $\eps>0$ be smaller than the constant $\eps_1(\Lambda_2)$ from \Cref{thm:eps regularity}, $\eps_3(\Lambda_1 +\Lambda_2)$ from \Cref{lem: L21 control mean curvature}, $\eps_4(\Lambda_1 + \Lambda_2)$ from \Cref{lem: L2 weak estimate for gauss map annulus has small energy condition} and smaller than $\eta$ from \cite[Lemma V.1]{BernardRiviere}. With this choice of $\eps$, we apply \Cref{prop:Bubble neck decomposition} to our sequence $\vec \xi_k$ to identify the bubbles and necks of our immersions. The region around concentration points splits into neck regions and bubbles in the sense of \eqref{neck regions and bubble domains disjoint cover}: Up to rescaling and shifting, the neck regions can be expressed as annuli of the form $B_{1} \setminus B_{r_k} $ with $ r_k \to 0$ such that 
\[\sup _{s \in (r_k, 1/2)}\int _{B_{2s} \setminus B_{s} } |\nabla \vec n_k|^2 \dif{x} < \eps.\]
We may apply \Cref{lem: L21 control mean curvature} in these neck regions to find a universal constant $\alpha_0>0$ such that either
\begin{equation}
\limsup _{k\to \infty} \left \|e^{\lambda_k} \vec H _{k} \right \|_{L^{2,1}(B_{\alpha_0 } \setminus B_{\alpha_0^{-1}r_k} )}\leq C(\Lambda_1+\Lambda_2)\label{L2 strong bound 1}
\end{equation}
or
\begin{equation}
\limsup _{k\to \infty} \left \|e^{\hat{\lambda}_k} \hat{\vec{H}} _{k}  \right \|_{L^{2,1}(B_{\alpha_0 } \setminus B_{\alpha_0^{-1}r_k} )}\leq C(\Lambda_1+\Lambda_2).\label{L2 strong bound 2}
\end{equation}
Here, the hat denotes the quantities corresponding to $\hat{\vec{\xi}}_k:B_{1} \setminus B_{r_k} $ which was defined in \eqref{inverted immersion} as
\[\hat{\vec{\xi}}_k(z) = \vec \xi_k\left (\frac{r_k}{z}\right ).\]
Further, we apply \Cref{lem: L2 weak estimate for gauss map annulus has small energy condition} to obtain that for all $\eps>0$, there is $\alpha_\eps>0$ such that for $0<\alpha < \alpha_\eps$
\begin{align}
&\quad \,\|e^{\lambda_k}\vec H_k\|_{L^{2,\infty}(B_{\alpha }\setminus B_{\alpha^{-1}r_k})} +\|e^{\hat{\lambda}_k} \hat{\vec{H}}_k\|_{L^{2,\infty}(B_{\alpha }\setminus B_{\alpha^{-1}r_k})}\notag \\
&\leq
\|\nabla \vec n_{k}\|_{L^{2,\infty}(B_{\alpha  }\setminus B_{\alpha^{-1}r_k})} +\|\nabla \hat{\vec{ n}}_k\|_{L^{2,\infty}(B_{\alpha  }\setminus B_{\alpha^{-1}r_k})}\notag\\
&< \eps.\label{L2 weak eps  bound}
\end{align}
Suppose without loss of generality that \eqref{L2 strong bound 1} holds, otherwise we work with $\hat{\vec{\xi}}_{k}$ instead. Using that $L^{2,\infty}$ is the dual space of $L^{2,1}$, we obtain by pairing \eqref{L2 strong bound 1} with \eqref{L2 weak eps  bound} on the annulus $B_{\alpha}  \setminus B_{\alpha^{-1} r_k} $ for $0<\alpha < \alpha_0$
\begin{align}
\|e^{\lambda_k} \vec H_k\|^2_{L^{2}(B_{\alpha  } \setminus B_{\alpha^{-1}r_k} )} &= \left \langle e^{ {\lambda}_k}  {\vec{H}} _{k}  , e^{ {\lambda}_k}  {\vec{H}} _{k}  \right \rangle _{L^2(B_{\alpha  } \setminus B_{\alpha^{-1}r_k} )}\leq \left \| e^{ {\lambda}_k}  {\vec{H}} _{k}   \right  \|_{L^{2,1}(B_{\alpha  } \setminus B_{\alpha^{-1}r_k} )} \left \|e^{ {\lambda}_k}  {\vec{H}} _{k} \right \|_{L^{2,\infty}(B_{\alpha  } \setminus B_{\alpha^{-1}r_k} )} \notag\\
&\leq C(\Lambda_1+ \Lambda_2) \cdot  \eps .\label{final estimate for L2 bound for H}
\end{align}
As $\eps$ was arbitrary, we see
\begin{align}
\limsup_{\alpha \to 0} \limsup_{k\to \infty} \|e^{\lambda_k} \vec H_k\|^2_{L^{2}(B_{\alpha  } \setminus B_{\alpha^{-1}r_k} )} = 0.\label{limsup bounded by balancing term}
\end{align}
Using further that $|\nabla \vec n_k|^2 = e^{2\lambda_k}(4 H_{k}^2 - 2K_{k})$ together with \cite[Lemma V.1]{BernardRiviere} yields
\begin{equation}
\lim _{\alpha \to 0} \lim_{k\to \infty} \int_{B_{\alpha } \setminus B_{\alpha^{-1}r_k}  } |\nabla \vec n_k|^2 \dif{x} = 0.\label{no neck energy condition}
\end{equation}
This holds for each neck region, so in particular for the union of the neck regions $\Omega_k(\alpha)$ from \eqref{union of all neck regions},
\begin{equation}
\lim _{\alpha \to 0} \lim_{k\to \infty} \int_{\Omega_k(\alpha)} |\nabla \vec n_k|^2 \dif{x} = 0.\label{no neck energy condition 2}
\end{equation}
We identify the bubbles in the same way as in \cite{BernardRiviere}. A bubble domain is of the form
\begin{equation}
B(i,j,\alpha, k) \cqq B_{\alpha^{-1}\rho^{i,j}_k}(x^{i,j}_k) \setminus \bigcup_{j'\in K^{i,j}} B_{\alpha \rho^{i,j}_k}(x^{i,j'}_k),\label{bubble domain}
\end{equation}
see \eqref{bubble domains}. By the construction from \eqref{definition index set Iij}, $x^{i,j'}_k \in B_{2\rho^{i,j}_k}(x^{i,j}_k)$ for all $j' \in K^{i,j}$. On each bubble domain, no further concentration of energy occurs by \eqref{eq:lower_bound_rho}. The bubbles are either true bubbles or pseudo bubbles in the sense of \eqref{true bubble}. The pseudo bubbles are those that don't carry any energy and immerse planes in the limit. As shown in \cite[(VIII.10)]{BernardRiviere}, the conformal factors are well-behaved in each bubble region, namely for all $\alpha \in (0,1)$, there is $C_\alpha$ such that
\begin{equation}
1\leq \frac{\sup_{x\in B(i,j,\alpha,k)} e^{\lambda_k(x)}}{\inf_{x\in B(i,j,\alpha,k)} e^{\lambda_k(x)}}\leq C_\alpha.\label{Harnack estimate in bubble region}
\end{equation}
We choose an arbitrary point 
\[z^{i,j}_k \in B(i,j,1/M,k) \cap B_{\rho^{i,j}_k}(x^{i,j}_k),\]
where $M>\sqrt{\# K^{i,j}}$ (to ensure that such a point $z^{i,j}_k$ exists) and we set $\lambda(i,j,k ) \cqq \lambda_k(z^{i,j}_k)$. We denote 
\[\tilde{\vec{\xi}}_k(y) \cqq e^{-\lambda(i,j ,k) - \ln(\rho^{i,j}_k)} (\vec{\xi}_k(\rho^{i,j}_ky + x^{i,j}_k) - \vec \xi_k(z^{i,j}_k)).\]
Furthermore, we extract a subsequence such that for all $j'\in K^{i,j}$, the limit
\[ \lim_{k\to \infty} \frac{x^{i,j'}_k - x^{i,j}_k}{\rho^{i,j}_k} = a^{j,j'}_i \in B_2  \]
exists. Since by \eqref{eq:lower_bound_rho}, no energy concentrates in $B(i,j,\alpha,k)$, the $\eps$-regularity yields smooth convergence up to a subsequence, i.e.,
\begin{equation}
\tilde{\vec{\xi}}_k \to \vec{\xi}^{i,j} \quad \text{in}\quad C^l_{\loc}\bigg(\C \setminus \bigcup_{j'\in K^{i,j}}\{a^{j,j'}_i\}\bigg) \label{convergence for the bubbles}
\end{equation}
for all $l\in \N$. Via the stereographic projection $\pi$ from $\mb S^2$ to $\C$, we obtain an immersion $\vec \zeta ^{i,j} \cqq \vec \xi^{i,j} \circ \pi$ which is constrained Willmore\footnote{Notice that the scaling factor $-\lambda(i,j,k) - \ln(\rho^{i,j}_k)$ is bounded from above, since we have bounded area, and so the coefficients of the original sequence cannot diverge.} on $\mb S^2 \setminus \left (\pi^{-1}(\infty)\cup \bigcup _{j'\in K^{i,j}}\{\pi^{-1}(a^{j,j'}_i)\}\right )$.
As before, if 
\[\mc A(\vec \xi^{i,j})=\infty,\]
then this implies that $e^{-\lambda(i,j,k) - \ln \rho^{i,j}_k} \to \infty$ as $k\to \infty$ and so the constraints corresponding to $\tilde{\vec{\xi}}_k$ tend to 0. In particular, $\vec \xi^{i,j}$ is a possibly branched, \emph{unconstrained} Willmore immersion. Using \cite[Lemma 5.13]{RiviereLectureNotes}, we are able to remove the ends via an inversion $\vec \Xi$ such that
\begin{equation}
\vec \Xi \circ \tilde{\vec{\xi}}_k \circ \pi \to \hat{\vec{\zeta}}^{i,j} \cqq \vec \Xi \circ \vec \zeta^{i,j} \quad \text{in}\quad C^l_{\loc}(\mb S^2\setminus \bigcup_{j'\in I^{i,j}}\{a^{j,j'}_i\})\label{inverted limit immersion}
\end{equation}
for all $l\in \N$. Here, $\hat{\vec{\zeta}}^{i,j}$ is an unconstrained Willmore immersion on $\mb S^2 \setminus \bigcup_{j'\in I^{i,j}}\{\pi^{-1}(a^{j,j'}_i)\}$. This finishes the energy quantization. Finally, \eqref{energy quantization: strong convergence everywhere} follows as $\mc W(\vec \Phi_k) \to \mc W(\vec \Phi_\infty)$ implies that no bubbles occur along the sequence, and so in particular the set of energy concentration points $\{a_1,\ldots, a_n\}$ is empty.
\end{proof}
\begin{proof}[Proof of \Cref{thm:Compactness of constrained Willmore surfaces}]
Let $\delta > 0$, $\vec \Phi_k:\Sigma \to \R^3$ be a sequence of constrained Willmore immersions with coefficients $\alpha_k,\beta_k,\gamma_k$ such that 
\begin{equation}
\max\{|\alpha_k|,|\beta_k|,|\gamma_k|\}<1/\delta, \label{compactness: bounded coefficients}
\end{equation}
\begin{equation}
\delta < \mc A(\vec \Phi_k)<1/\delta, \label{compactness: bounded area}
\end{equation}
and
\begin{equation}
\mc W(\vec \Phi_k)< \begin{cases}
\min\{8\pi,\bm{\beta}_p + \max\{\beta^{\mc T}_0(\mc T(\vec \Phi_k)), \beta_0^\iso(\iso (\vec \Phi_k))\} - 4\pi\}-\delta &\quad \text{if $p\geq 1$},\\
8\pi - \delta & \quad \text{if $p=0$}.
\end{cases}\label{compactness:bounded willmore}
\end{equation}
By \eqref{compactness:bounded willmore}, we can apply \cite{CompactnessKuwertSchatzle} to see that the conformal classes associated to $\vec \Phi_k ^* g_{\R^3} \qqc g_k$ remain in a compact subset of the moduli space. In particular, there exist constant scalar curvature metrics $h_k$ of unit volume conformally equivalent to $g_k$ converging in $C^l(\Sigma)$ to a limiting metric $h_\infty$ for all $l\in \N$. Furthermore, we find diffeomorphisms $f_k$ of $\Sigma$ such that $\vec \xi_k \cqq\vec \Phi_k \circ f_k$ are conformal. We denote by $\lambda_k$ the conformal factor of $\vec \xi_k$ with respect to $h_k$. Applying \eqref{energy quantization: strong convergence everywhere} from \Cref{thm:Energy quantization}, it is enough to show that 
\begin{equation}
\lim_{k\to \infty} \mc W(\vec \xi_k) = \mc W(\vec \zeta_\infty),\label{compactness:limit of willmore energy}
\end{equation}
where $\vec \zeta_\infty$ is defined in \eqref{limit immersion on Sigma}. 

Suppose first that $p\geq 1$. If \eqref{compactness:limit of willmore energy} fails, there must be at least one bubble. We choose conformal, converging coordinates around one of the energy concentration points $a_i$ and we choose a bubble that contains no further bubbles\footnote{This means that $K^{i,j} = \emptyset$, which can always be achieved by subsequently choosing more and more concentrated bubbles, which can be done only finitely many times.} centered at $x^{i,j}_k \in \C$ with a radius $\rho^{i,j}_k \to 0$. Then this bubble is a true bubble in the sense of \eqref{true bubble}, which can be seen from the construction \eqref{eq:bubble_energy}. The bubble converges in the sense of \eqref{convergence for the bubbles} to some constrained Willmore immersion $\vec{\zeta}^{i,j}$ on $\mb S^2$. Notice that we get local smooth convergence in all of $\mb S^2 \setminus \{\pi^{-1}(\infty)\}$ as we chose the most concentrated bubble.

Using \cite[Theorem 5.5]{RiviereLectureNotes}, the conformal factors $\lambda_k$ of $\vec \xi_k$ either diverge to $-\infty$ or remain bounded on $\Sigma$ away from the energy concentration points. We thus consider several cases:
\begin{enumerate}[wide=0pt,leftmargin=\parindent, labelsep=0.5em, label = \textbf{Case \arabic*:}]
\item $\lambda_k\to -\infty$ uniformly on $\Sigma \setminus \bigcup _{i=1}^n B_\eps(a_i)$ for any $\eps > 0$.\\
By the arguments from \cite{KMR}, we know that there is exactly one energy concentration point $a_i$ around which area can and must accumulate in the middle. Using the Cut and Fill procedure as done in \cite[Section 4.3]{KMR} for $\mc I$ and in \cite[Section 4.4.4]{MasterThesis} for $\mc T$, we can show that there are $\delta^1_k, \, \delta^2_k \to 0$ such that
\[\liminf _{k\to \infty} \mc W(\vec \xi_k) \geq \bm{\beta}_p +\liminf _{k\to \infty}\max\{\beta^{\mc T}_0(\mc T(\vec \xi_k)+\delta^1_k), \beta_0^\iso(\iso (\vec \xi_k)+\delta^2_k)\} - 4\pi .\]
$\mc T(\vec \xi_k)$ remain bounded by \cite[(1.19)]{MasterThesis} and $\mc I(\vec \xi_k)$ remain bounded by \cite[Theorem 1]{Schygulla}. Hence, up to a subsequence, $\mc T(\vec \xi_k) \to T_0$ and $\mc I(\vec \xi_k)\to I_0$ such that $\beta_0^{\mc T}(T_0)<8\pi$ and $\beta_0^{\mc I}(I_0)<8\pi$, using \cite[Theorem 4.3]{MasterThesis} and \cite[Theorem 3.2]{KMR}. Now, $\beta_0^{\mc T}$ is continuous at $T_0$ by the same arguments as in \cite{MasterThesis} and $\beta_0^{\mc I}$ is continuous at $I_0$ by \cite{Schygulla}. Thus, 
 \[\liminf _{k\to \infty} \mc W(\vec \xi_k) \geq \bm{\beta}_p +\liminf _{k\to \infty}\max\{\beta^{\mc T}_0(\mc T(\vec \xi_k)), \beta_0^\iso(\iso (\vec \xi_k))\} - 4\pi ,\]
 which contradicts \eqref{compactness:bounded willmore}.
\item $\lambda_k$ remains bounded on $\Sigma \setminus \bigcup_{i=1}^n B_{\eps}(a_i)$ for any fixed $\eps>0$. \\
In this case, $\vec \zeta_\infty$ is closed and contains at least $4\pi$ Willmore energy. 
\begin{enumerate}[wide=0pt,leftmargin=\parindent, labelsep=0.5em,label = \textbf{Case \arabic{enumi}\alph*):}]
\item At least one of the bubbles carries some area (before rescaling).\\
In this case, we can apply \cite[Lemma 4.4]{KMR} to see that this bubble contains at least $4\pi$ energy, which directly contradicts \eqref{compactness:bounded willmore}.
\item None of the bubbles carry any area: \\
In particular, the bubble around $x^{i,j}_k$ does not carry area in the limit. If the corresponding limit $\vec \zeta^{i,j}$ does not contain any ends, then $\vec \zeta^{i,j}$ is closed and contains at least $4\pi$ Willmore energy, which contradicts \eqref{compactness:bounded willmore}. If it does contain ends, then there can only be an end at $\infty$ and $\hat{\vec{\zeta}}^{i,j}$ from \eqref{inverted limit immersion} is an unconstrained Willmore immersion on $\mb S^2\setminus \{\pi^{-1}(\infty)\}$. If $\hat{\vec{\zeta}}^{i,j}$ has at least density 2 at $\pi^{-1}(\infty)$, then the Li-Yau inequality implies that $\mc W(\hat{\vec{\zeta}}^{i,j})\geq 8\pi$. It follows that
\begin{equation}
\lim _{k\to \infty} \mc W(\vec \xi_k) \geq \lim _{k\to \infty} \mc W( \vec \Xi_k \circ  (e^{-\lambda(i,j,k) - \ln(\rho^{i,j}_k)} (\vec{\xi}_k(\rho^{i,j}_k \cdot + x^{i,j}_k) - \vec \xi_k(z^{i,j}_k)))  ) \geq \mc W(\hat{\vec{\zeta}}^{i,j})\geq 8\pi,\label{compactness: limit 8pi}
\end{equation}
a contradiction. So the density at infinity is 1. Since $\hat{\vec{\zeta}}^{i,j}$ is unbranched in $\C$, the divergence theorem applied to \eqref{First variation of Willmore 1} implies that the first residue 
\[\frac{1}{2}\int_{B_{\rho} }   2\parpar{r}{\vec H_{\hat{\vec{\zeta}}^{i,j}}} - 3 H_{\hat{\vec{\zeta}}^{i,j}}\parpar{r}{\vec n_{\hat{\vec{\zeta}}^{i,j}}  } -\vec H_{\hat{\vec{\zeta}}^{i,j}}\times \frac{1}{r} \parpar{\theta}{\vec n_{\hat{\vec{\zeta}}^{i,j}}}  \dif{l} = 0\]
vanishes. This allows us to apply the point removability \cite[Corollary 1.5]{BernardRiviereSingularityRemovability} to see that $\hat{\vec{\zeta}}^{i,j}$ is smooth throughout $\pi^{-1}(\infty)$ and in particular is an unconstrained Willmore immersion on $\mb S^2$. As this bubble was a true bubble, it holds
\[\mc E(\vec{\zeta}^{i,j}) \geq \eps_0\]
and so $\vec{\zeta}^{i,j}$ cannot immerse a flat plane. Equivalently, $\hat{\vec{\zeta}}^{i,j}$ cannot immerse a round sphere.
 By the classification of Willmore spheres \cite{Bryant}, it follows that $\mc W(\hat{\vec{\zeta}}^{i,j})\geq 8\pi$ and so the conclusion follows again by \eqref{compactness: limit 8pi}.
\end{enumerate}
\end{enumerate}
In the case $p=0$, thanks to a result by \textcite{MondinoRiviere}, we can always choose the parametrization in such a way that Case 2 holds, i.e., that the conformal factors remain bounded away from the energy concentration points. Now, in Case 2 we showed that $\liminf _{k\to \infty} \mc W(\vec \xi_k)\geq 8\pi$, which yields the claim for $p=0$ and finishes the proof.
\end{proof}

\section{Bounding the Lagrange multipliers}\label{sec:Other stuff}
\begin{lemma}[Boundedness of Lagrange multipliers]\label{lem:boundedness of Lagrange multiplier for minimizers}
Let $\Sigma$ be a genus $p$ surface, $\{\tau_k,\, k\in \N\} \subset \subset I \cqq(0,\sqrt{8\pi})\setminus \{\sqrt{4\pi}\}$ and let $\vec \Phi_k:\Sigma \to \R^3$ be minimizers of the normalized curvature constrained Willmore problem \eqref{total mean curvature problem} with respect to the constraint $\mc T(\vec \Phi_k) = \tau_k$. Then the Lagrange multipliers $\lambda_k$ defined as in \eqref{explicit lagrange multipliers T} satisfy 
\begin{equation}
\sup_{n\in \N} |\lambda _k| < \infty. \label{uniform bound for Lagrange multipliers}
\end{equation}
\end{lemma}
\begin{proof}
Up to subsequence, $\tau_k \to \tau \in I$. Suppose by contradiction that the Lagrange multipliers diverge, i.e., $|\lambda_k| \to \infty$. Up to rescaling and shifting, we may assume that $\mc A(\vec \Phi_k)=1$ and $0\in \vec \Phi_k (\Sigma)$. Using the results from \cite{MasterThesis}, there are bilipschitz diffeomorphisms $f_k$ on $\Sigma$ such that there are finitely many points $\{a_1,\ldots, a_n\} \subset \Sigma$ and $\vec \xi_{\infty}:\Sigma \to \R^3$ such that
\[\vec \xi_k \cqq \vec \Phi_k \circ f_k \wto \vec\xi _{\infty}\quad\text{in $W^{2,2}_{\loc}(\Sigma \setminus \{a_1,\ldots, a_n\})$}.\]
The $\vec \xi_k$ are conformal with respect to constant scalar curvature metrics $h_k$ of unit volume which converge smoothly to a metric $h_\infty$. Furthermore, $\vec \xi _{\infty}$ is a minimizer of \eqref{total mean curvature problem}, $\mc W(\vec \xi _{\infty}) = \lim _{n\to \infty} \mc W(\vec \xi _{n}) < 8\pi$ and $\mc A(\vec \xi_{\infty})=1$. In particular, $\vec \xi _{\infty}$ is unbranched and smooth. Composing with converging charts away from the $a_i$, we may work with $\vec \xi_k,\, \vec \xi_\infty:B_1\to \R^3$ conformal. Let $\vec \phi \in C_c^\infty(B_1, \R^3)$. Testing the Euler-Lagrange equation for the Willmore energy \eqref{First variation of Willmore 1} yields
\begin{equation}
\sup _{k\in \N}\left |\delta \mc W(\vec \xi_k)(\vec \phi)\right | = \sup _{k\in \N}\left | \int_{B_1} \langle \vec H_k, \Delta \vec \phi \rangle -\frac{1}{2} \langle -3H_k \nabla \vec n_k + \vec H_k \times \nabla ^\perp \vec n_k, \nabla \vec \phi \rangle \dif{x}\right | <C<\infty\label{bounded Willmore variation}
\end{equation}
due to the weak $W^{2,2}$ convergence. Then the Euler Lagrange equation in divergence form implies
\begin{align}
|\delta \mc T(\vec \xi_k)(\vec \phi)| &=\left |\int_{B_1} \left \langle \nabla \vec \phi , \frac{\nabla \vec n_k}{2}  + H_k \nabla \vec \xi_k - \frac{\mc T(\vec \xi_k)}{2} \nabla \vec \xi_k\right \rangle \dif{x}\right |\notag\\
&\stackrel{k\to \infty}{\to} \left |\int_{B_1} \left \langle \nabla \vec \phi , \frac{\nabla \vec n_\infty}{2} + H_\infty \nabla \vec \xi_{\infty} - \frac{\mc T(\vec \xi_{\infty})}{2} \nabla \vec \xi_{\infty}\right \rangle \dif{x}\right |\notag\\
&=\left | \int_{B_1} \langle \vec \phi , \vec n_{\infty} \rangle \left ( -K_{\infty}  +  \mc T(\vec \xi _{\infty}) H_{\infty} \right )e^{2\lambda_{\infty}} \dif{x}\right | .\label{explicit T variation}
\end{align}
Here we used $\nabla \vec n_k \wto \nabla \vec n_{\infty}$ in $L^2(B_1)$, $H_k \wto H_{\infty}$ in $L^2(B_1)$ and $\nabla \vec \xi_k \to \nabla \vec \xi_{\infty}$ in $L^2(B_1)$ and the equivalence of the weak and strong formulation of the first variation as $\vec \xi _{\infty}$ is smooth. On the other hand, the assumption $|\lambda_k|\to \infty$ together with \eqref{bounded Willmore variation} yields
\[|\delta \mc T(\vec \xi_k)(\vec \phi)|=  \frac{|\delta \mc W(\vec \xi_k)(\vec \phi)|}{|\lambda_k|} \to 0.\]
As $\vec \phi$ was arbitrary, it follows from \eqref{explicit T variation} that $  K_{\infty}  =  \mc T(\vec \xi _\infty) H_{\infty} $ for all $x\in \Sigma \setminus \{a_1,\ldots, a_n\}$. Furthermore, since $\vec \xi_{\infty}$ is unbranched, the Gauss-Bonnet theorem still holds (see \cite[(2.6)]{MondinoScharrer}) and so the argument of \cite[Proposition 3.6]{MasterThesis} yields that all points in $\Sigma\setminus \{a_1,\ldots, a_n\}$ are umbilic. Thus, $\vec \xi_{\infty}$ immerses a round sphere and $\tau=\mc T(\vec \xi_{\infty}) = \sqrt{4\pi}$, which is a contradiction.
\end{proof}
\begin{remark}\label{rem:iso lagrange multipliers remain bounded}
A similar argument was used in \cite[Lemma 3.2]{KuwertAsymptoticsSmallIso} to show that the Lagrange multiplier for the isoperimetrically constrained problem \eqref{isoperimetric problem} remain bounded in $(\sqrt[3]{36\pi}+\delta,\infty)$ for any $\delta>0$.
\end{remark}
A direct corollary of the previous result is that the profile curve of $\beta_p^{\mc T}$ is Lipschitz away from the end points.
\begin{cor}[Lipschitz continuity of $\beta^{\mc T}_p$]\label{cor:Lipschitz continuity of graph}
The profile curves $\beta^{\mc T}_p$ are locally Lipschitz continuous on $I=(0,\sqrt{8\pi})\setminus \{\sqrt{4\pi}\}$. 
\end{cor}
\begin{proof}
Let $[a,b]\subset\subset I$. Recall that $\beta^{\mc T}_p$ is continuous on $I$. Using \Cref{lem:boundedness of Lagrange multiplier for minimizers}, we see that
\[C\cqq \sup \left \{|\lambda|, \, \text{$\vec \Phi$ is a minimizer of $\beta_p^{\mc T}$ with $\mc T(\vec \Phi) =\tau\in [a,b]$ and $\delta \mc W(\vec \Phi) = \lambda \delta \mc T(\vec \Phi)$}\right \}<\infty.\]
Testing each such minimizer with a variation as in \cite[Lemma 3.7]{MasterThesis} for each $\tau\in [a,b)$, we see that there is a function $y:[a,b) \to [a,b]$ with $y(\tau)>\tau$ for all $\tau\in [a,b)$ such that
\begin{equation}
\left |\frac{\beta^{\mc T}_p(y(\tau))-\beta^{\mc T}_p(\tau)}{y(\tau)-\tau}\right | < 2C.\label{small variation with bounded lagrange multiplier}
\end{equation}
Let 
\[E \cqq \left \{\tau'\in [a,b],\,  \left |\frac{\beta^{\mc T}_p(\tau')-\beta^{\mc T}_p(a)}{\tau'-a}\right |\leq 2C\right \}.\]
$E$ is nonempty because it contains $y(a)$. Let $\tau_n \in E$ such that $\tau_n \nearrow \tau \cqq \sup (E)\in [a,b]$. Then by continuity
\[\left |\frac{\beta^{\mc T}_p(\tau) - \beta^{\mc T}_p(a)}{\tau-a}\right |\leq 2C\]
and if $\tau<b$, then $y(\tau)\in E$ by \eqref{small variation with bounded lagrange multiplier}, which is a contradiction to $\tau = \sup(E)$. It follows that $b\in E$. Because $a$ and $b$ were arbitrary, the conclusion follows.
\end{proof}

\begin{ack}
The second author acknowledges support by the Studienstiftung des deutschen Volkes.
\end{ack}

\appendix
\renewcommand{\thesection}{Appendix \Alph{section}}
\tocless\section{The bubble-neck decomposition}
\addcontentsline{toc}{section}{Appendix A. The bubble-neck decomposition}
\renewcommand{\thesection}{\Alph{section}}
\label{sec:Appendix Bubble neck decomposition}
In this appendix, we want to revisit the bubble-neck decomposition originally done in \cite[Proposition III.1]{BernardRiviere} and slightly change the way the bubbles are chosen. There are also some slight alterations in the statement. The reason for this is that \cite[(III.14)]{BernardRiviere} cannot be guaranteed in general. It states that each bubble has at least $\eps_0$ energy, excluding the energy of the bubbles contained within. Namely, it might happen that all of the energy in a bubble is concentrated within the bubbles that are contained in it. We call bubbles in which this happens a \emph{pseudo bubble}. The remaining bubbles that do contain some remaining energy will be referred to as \emph{true bubbles}.
\begin{prop}[Bubble-neck decomposition, see {\cite[Proposition III.1]{BernardRiviere}}] \label{prop:Bubble neck decomposition}
Let $\Sigma$ be a closed two-dimensional manifold. Let $\vec{\Phi}_k \in \mc E_{\Sigma}$ be a sequence of weak immersions (see \cite{RiviereCrelle} for a definition) such that
\begin{equation} \label{eq:bounded_second_fundamental}
\int_{\Sigma} |d\vec n_{\vec{\Phi}_k}|^2_{h_k}\dif{\mu_{h_k}} \leq \Lambda,
\end{equation}
where $g_k \cqq \vec{\Phi}_k^* g_{\R^3}$.

We suppose that $g_k$ is conformally equivalent to a constant scalar curvature metric $h_k$ of unit volume such that for all $l \in \N$,
\[
h_k \rightarrow h_{\infty} \text{ in } C^l(\Sigma),
\]
where $h_{\infty}$ is a constant scalar curvature metric of unit volume on $\Sigma$.
Suppose $0<\eps_0<8\pi/3$ and suppose there exist $n$ points $\{a_1, \ldots, a_n\} \subset \Sigma$ such that for $x\in \Sigma$, $x\in \Sigma \setminus \{a_1,\ldots, a_n\}$ if and only if there exists $\rho(x)>0$ such that 
\begin{equation}
\limsup_{k\to \infty} \int_{B_{\rho(x)}(x)} |d \vec n _{\vec \Phi_k}|^2 _{h_k} \dif{\mu _{h_k}} < \eps_0. \label{rho of x radius}
\end{equation}
Here and in the following, the balls $B_{\rho(x)}(x)$ are measured with respect to $h_\infty$. Then there exist a subsequence, still denoted $\vec{\Phi}_k$, $n$ integers $\{Q^1, \dots, Q^n\}$, $n$ sequences of points $(x^{i,j}_k)_{j=1,\dots,Q^i} \subset \Sigma^{Q^i}$, and $n$ sequences of radii $(\rho^{i,j}_k)_{j=1,\dots,Q^i}$ satisfying
\begin{equation} \label{eq:point_convergence}
\lim_{k \to  \infty} x^{i,j}_k = a_i,
\end{equation}
\begin{equation} \label{eq:radius_convergence}
\lim_{k \to  \infty} \rho^{i,j}_k = 0,
\end{equation}
and for all $i\in \{1,\ldots, n\}$ and $j\neq j' \in \{1,\ldots, Q^i\}$
\begin{equation} \label{eq:rho_ratio}
\text{either}\quad \lim_{k \to  \infty} \frac{\rho^{i,j}_k}{\rho^{i,j'}_k} + \frac{\rho^{i,j'}_k}{\rho^{i,j}_k} = \infty \quad \text{or}\quad 
\lim_{k \to  \infty} \frac{|x^{i,j}_k - x^{i,j'}_k|}{\rho^{i,j}_k + \rho^{i,j'}_k} =  \infty.
\end{equation}

Moreover, for all $i\in \{1,\ldots, n\}$ and $j\in \{1,\ldots, Q^i\}$
\begin{equation} \label{eq:bubble_energy}
\int_{B_{\rho^{i,j}_k}(x^{i,j}_k)} |d\vec{n}_{\vec{\Phi}_k}|^2_{h_k} \dif{\mu_{h_k}} \geq \eps_0.
\end{equation}
The set of balls $B_{\rho^{i,j}_k}(x^{i,j}_k)$ are called ``bubbles associated to the sequence $\vec{\Phi}_k$''.

Furthermore, for any $i \in \{1, \dots, n\}$ and for any $j \in \{1, \dots, Q^i\}$, the set of indices\footnote{The doubled radius $2\rho^{i,j}_k$ in the definition is there to catch bubbles which might form around $\partial B_{\rho^{i,j}_k}(x^{i,j}_k)$.}
\begin{equation}
I^{i,j} \cqq \left\{ j' ,\, x^{i,j'}_k \in B_{2\rho^{i,j}_k}(x^{i,j}_k), \, \frac{\rho^{i,j}_k}{\rho^{i,j'}_k} \to \infty \right\}
\label{definition index set Iij}
\end{equation}
is independent of $k$. It is called the set of ``bubbles contained in the bubble $B_{\rho^{i,j}_k}(x^{i,j}_k)$''. Similarly, the ``bubbles directly contained in the bubble $B_{\rho^{i,j}_k}(x^{i,j}_k)$'' is the set of indices 
\begin{equation}
K^{i,j} \cqq I^{i,j}\setminus \bigcup_{j'\in I^{i,j}} I^{i,j'}.\label{definition index set Kij}
\end{equation}
There is a unique bubble $r^i \in \{1,\ldots, Q^i\}$ that includes all the others, i.e., for $\alpha < 1/2$
\begin{equation}
\bigcup_{j=1}^{Q^i} B_{ \alpha^{-1}\rho^{i,j}_k}(x^{i,j}_k) = B_{ \alpha^{-1}\rho^{i,r^i}_k}(x^{i,r^i}_k).\label{root bubble}
\end{equation}
For any $\alpha \in (0,1/2)$, we denote for $i\in \{1,\ldots, n\}$
\[
\Omega^i_k(\alpha) \cqq B_{\alpha}(x^{i,r^i}_k) \setminus B_{\alpha^{-1} \rho^{i,r^i}_k}(x^{i,r^i}_k),
\]
and for $j\in \{1,\ldots, Q^i\}$
\begin{equation}
\Omega^{i,j}_k(\alpha) \cqq \bigcup _{j'\in K^{i,j}} \left (B_{\alpha \rho^{i,j}_k}(x^{i,j'}_k) \setminus  B_{\alpha^{-1} \rho^{i,j'}_k}(x^{i,j'}_k)\right ).\label{eq:neck annuli}
\end{equation}
The sets $\Omega^i_k(\alpha)$ and $\Omega^{i,j}_k(\alpha)$ are the ``$\alpha$-neck regions'' of $\vec{\Phi}_k$. They are the disjoint union of diverging annuli. Let
\begin{equation}
\Omega_k(\alpha) \cqq \bigcup_{i=1}^n \left (\Omega^i_k(\alpha) \cup \bigcup_{j=1}^{Q^i} \Omega^{i,j}_k(\alpha)\right ).
\label{union of all neck regions}
\end{equation}
We define the ``bubble domains'' $B(i,j,\alpha, k) $ as
\begin{equation}
B(i,j,\alpha, k) \cqq B_{\alpha^{-1}\rho^{i,j}_k}(x^{i,j}_k) \setminus \bigcup_{j'\in K^{i,j}} B_{\alpha \rho^{i,j}_k}(x^{i,j'}_k).\label{bubble domains}
\end{equation}
The neck regions and bubble domains form a disjoint union of the energy concentration regions in the sense that
\begin{equation}
\bigcup _{i=1}^n B_{\alpha}(x^{i,r^i}_k) =\Omega _{k}(\alpha) \cup \bigcup _{i=1}^n \bigcup _{j=1}^{Q^i} B(i,j,\alpha,k). \label{neck regions and bubble domains disjoint cover}
\end{equation}
There exists $\alpha_0$ with $0 < \alpha_0 < 1$ independent of $k$ such that for all $\alpha$ with $0<\alpha<\alpha_0$, $i\in \{1,\ldots, n\}$, $j\in \{1,\ldots, Q^i\}$, all $k\in \N$, and all $\rho>0$ such that $B_{2\rho}(x^{i,j}_k)\setminus B_{\rho}(x^{i,j}_k)\subset \Omega_k(\alpha)$, it holds
\begin{equation} \label{eq:double_neck_bound}
\int_{B_{2\rho}(x^{i,j}_k) \setminus B_{\rho}(x^{i,j}_k)} |d\vec{n}_{\vec{\Phi}_k}|^2_{h_k}  \dif{\mu_{h_k}} < \eps_0.
\end{equation}
We call a bubble $B_{\rho^{i,j}_k}(x^{i,j}_k)$ a ``true bubble'' if it holds
\begin{equation}
\limsup_{\alpha \to 0}\limsup_{k\to \infty} \int_{B(i,j,\alpha,k)} |d \vec n_{\vec \Phi_k}|^2 _{h_k} \dif{\mu_{h_k}} >0 .\label{true bubble}
\end{equation}
Otherwise, we call $B_{\rho^{i,j}_k}(x^{i,j}_k)$ a ``pseudo bubble''. We define
  \begin{equation}\label{eq:smallest_r}
  r^{i,j}_k \cqq \inf \left \{0 < r < \rho^{i,j}_k ,\, \begin{aligned}
  &B_r(x) \subset B_{\alpha^{-1}\rho^{i,j}_k}(x^{i,j}_k) \setminus \bigcup_{j' \in K^{i,j}} B_{\alpha \rho^{i,j}_k}(x^{i,j'}_k) \\
&\int_{B_{r}(x)} |d\vec{n}_{\vec{\Phi}_k}|^2_{h_k} \, \dif{\mu_{h_k}} \geq \eps_0
\end{aligned} 
\right \}.
\end{equation}
Then for all $i\in \{1,\ldots, n\}$ and $j\in \{1,\ldots, Q^i\}$, there holds
\begin{equation} \label{eq:lower_bound_rho}
\liminf_{k \to \infty} \frac{r^{i,j}_k}{\rho^{i,j}_k} > 0.
\end{equation}
\end{prop}

\begin{proof}
Let $0<\eps_0 < 8\pi/3$. Consider
\[\rho_k \cqq \inf \left \{ \rho>0, \, \int_{B_{\rho}(x)} |d \vec n _{\vec \Phi_k}|^2 _{h_k} \dif{\mu_{h_k}} \geq \eps_0 \quad \text{for some }x\in \Sigma\right \},\]
where $B_{\rho}(x)$ denotes the geodesic ball of center $x$ and radius $\rho$ for the constant scalar curvature metric $h_\infty$. If
\[\limsup_{k\to \infty} \rho_k > 0,\]
we can extract a subsequence such that $\rho_k$ converges to a positive constant. By \cite[Theorem 1.1]{RiviereCrelle}, the conformal factors remain bounded and there is no concentration of curvature. Then there is no concentration anywhere and no points $a_i$. 

Hence, in the following, we will assume that $\limsup_{k\to \infty} \rho_k = 0$. This means that there is some $x^{i,1}_k \in \Sigma$ such that
\[\int_{B_{\rho_k}(x^{i,1}_k)} |d\vec n_{\vec \Phi_k}|_{h_k}^2 \dif{\mu_{h_k}} = \eps_0.\]
Up to subsequence, due to \eqref{rho of x radius}, the $x^{i,1}_k$ converge to one of the $a_i$. We set $\rho^{i,1}_k \cqq \rho_k$. By choosing conformal, converging coordinates
$\omega^i_k:B_1\to (\Sigma,h_k)$ containing some fixed neighborhood of $a_i$ and mapping $0$ to $a_i$, we may work with the immersions $\vec \Phi_k \vert _{\omega^i_k(B_1)} \circ {\omega^i_k}$, which we will, by an abuse of notation, relabel from now on as $\vec \Phi_k$. Then $x^{i,1}_k \to 0$ and $\rho^{i,1}_k\to 0$. We choose $\alpha^{i,1}>0$ sufficiently small such that up to subsequence\footnote{Up to subsequence, $|\nabla \vec n_{\vec \Phi_k}|^2 \mc L^2\vert_{B_1}$ converge weakly as measures to some measure $\mu$. Choose $\alpha^{i,1}>0$ sufficiently small such that $\mu(B_{\alpha^{i,1}} ) < \mu(\{0\}) + \eps_0$ and $\mu(\partial B_{\alpha^{i,1}} ) = 0$. The weak convergence yields upper semicontinuity on closed sets, so that for $r\in (0,{\alpha^{i,1}})$
\[\limsup_{k\to \infty} \int_{B_{\alpha^{i,1}}(x^{i,1}_k) \setminus B_r(x^{i,1}_k) } |\nabla \vec n_{\vec \Phi_k}|^2 \dif{x}  \leq \mu(\overline{B_{\alpha^{i,1}}(0) }\setminus B_r(0) ) <\eps_0.\]}
\begin{equation}
\lim_{k\to \infty} \sup  \left \{r\in (0,{\alpha^{i,1}}),\, \int_{B_{{\alpha^{i,1}}}(x^{i,1}_k) \setminus B_{r}(x^{i,1}_k) } |\nabla \vec n_{\vec \Phi_k}|^2 \dif{x} \geq \eps_0\right \}=0.\label{choice of alpha for first bubble}
\end{equation}
We apply \cite[Lemma III.2]{BernardRiviere} to $B_{{\alpha^{i,1}}}(x^{i,1}_k)\setminus B_{\rho^{i,1}_k}(x^{i,1}_k)$ to obtain finitely many radii\footnote{Notice that $N$ is bounded in terms of $\Lambda$ and $\eps_0$.}
\[{\alpha^{i,1}} = R^0_k > R^1_k > \ldots >R^N_k = \rho^{i,1}_k\]
satisfying \cite[(III.18)-(III.20)]{BernardRiviere}. 
We merge adjacent annuli of the same type to get a subsequence of radii, which we will not relabel, in which annuli of type $I_0$ and $I_1$ alternate: For every annulus $l\in I_0$, it holds
\begin{equation}
\lim_{k\to \infty} \frac{R^{l}_k}{R^{l+1}_k} < \infty \quad \text{and}\quad \int_{B_{R^{l}_k}   \setminus B_{R^{l+1}_k} } |\nabla \vec n_{\vec \Phi_k}|^2 \dif{x} \geq \eps_0\label{type I0 annuli}
\end{equation}
and for $l \in I_1$, it holds
\begin{equation}
\lim_{k\to \infty} \frac{R^{l}_k}{R^{l+1}_k} = \infty \label{radius divergence}
\end{equation}
and  
\begin{equation}
\forall \rho \in (R^{l+1}_k, R^{l}_k/2):\quad \int _{B_{2\rho}(x^{i,1}_k)\setminus B_{\rho}(x^{i,1}_k)} |\nabla \vec n_{\vec \Phi_k}|^2 \dif{x} <\eps_0.\label{doubling annulus condition}
\end{equation}
Then \eqref{choice of alpha for first bubble} implies
\[\lim _{k\to \infty}\frac{R^0_k}{R^1_k} = \infty.\]
We will now construct more bubbles. For this, we will look through all annuli of type $I_0$. To start, let $\Omega^{l}_k \cqq B_{R^{l}_k}(x^{i,1}_k)\setminus B_{R^{l+1}_k}(x^{i,1}_k)$ be the innermost annulus of type $I_0$. 
\begin{enumerate}[wide=0pt,leftmargin=\parindent, labelsep=0.5em, label = \textbf{Case \arabic*:}]
\item If $l+1=N$, i.e., the innermost annulus is already of type $I_0$, we move on to the next annulus and do not declare this as a bubble region.
\item $l+1<N$. We declare $B_{R^{l}_k}(x^{i,1}_k)$ as a bubble. Furthermore, we define for $\Omega\subset B_{\alpha^{i,1}} $
\[\rho_k(\Omega) \cqq \inf \left \{\rho>0, \, \int_{B_{\rho}(x)}|\nabla \vec n_{\vec \Phi_k}|^2 \dif{x} = \eps_0 \quad\text{for some $x\in \Omega$}\right \},\]
where this infimum could be $\infty$. 
\begin{enumerate}[wide=0pt,leftmargin=\parindent, labelsep=0.5em,label = \textbf{Case \arabic{enumi}\alph*):}]
\item Suppose that
\begin{equation}
\liminf _{k\to \infty} \frac{\rho_k(\Omega^{l}_k)}{R^{l}_k} > 0.\label{no concentration in the I0 annulus}
\end{equation}
In this case, we choose a subsequence such that $\lim _{k\to \infty} \frac{\rho_k(\Omega^{l}_k)}{R^{l}_k} > 0$. Then the bubble $B_{R^{l}_k}(x^{i,1}_k)$ is a true bubble which contains the bubble $B_{\rho^{i,1}_k}(x^{i,1}_k)$. Additionally, there is no further concentration of energy in this bubble domain due to \eqref{no concentration in the I0 annulus}. Finally, we move on to the next annulus of type $I_0$.
\item Suppose
\begin{equation}
\lim _{k\to \infty} \frac{\rho_k(\Omega^{l}_k)}{R^{l}_k} = 0.\label{Some concentration in the I0 annulus}
\end{equation}
\eqref{Some concentration in the I0 annulus} means there is some concentration of energy inside $\Omega^{l}_k$. Via induction over $j$, we choose finitely many bubbles $B_{\rho^{i,l,j}_k}(x^{i,l,j}_k)$ and radii $\alpha^{i,l,j}>0$ such that
\begin{equation}
\rho^{i,l,j}_k \cqq \rho_k\left (\Omega^{l}_k\setminus \bigcup _{j'=1}^{j-1}B_{2\alpha^{i,l,j'}R^l_k}(x^{i,l,j'}_k)\right ) \quad\text{satisfies}\quad\lim_{k\to \infty} \frac{\rho^{i,l,j}_k}{R^{l}_k} = 0,\label{in I0 annulus there is more concentration to explore}
\end{equation}
\begin{equation}
x^{i,l,j}_k \in \Omega^{l}_k\setminus \bigcup _{j'=1}^{j-1}B_{2\alpha^{i,l,j'}}(x^{i,l,j'}_k)\quad \text{such that}\quad  \int_{B_{\rho^{i,l,j}_k}(x^{i,l,j}_k)} |\nabla \vec n_{\vec \Phi_k}|^2 \dif{x} = \eps_0.\label{choice of concentration points in I0 annulus}
\end{equation}
The $\alpha^{i,l,j}$ are chosen to satisfy, using the same arguments as in \eqref{choice of alpha for first bubble},
\begin{equation}
\lim_{k\to \infty} \sup  \left \{r\in (0,{2\alpha^{i,l,j}}),\, \int_{B_{{2\alpha^{i,l,j}R^{l}_k}}(x^{i,l,j}_k)\setminus B_{r R^{l}_k}(x^{i,l,j}_k)} |\nabla \vec n_{\vec \Phi_k}|^2 \dif{x} \geq \eps_0\right \}=0.\label{choice of alphas for the I0 bubble}
\end{equation}
For $j\neq j'$, the balls $B_{\alpha^{i,l,j}R^l_k}(x^{i,l,j}_k)$ are pairwise disjoint.
\eqref{in I0 annulus there is more concentration to explore} can only be satisfied finitely many often because at each iteration, we exhaust at least $\eps_0$ of the Dirichlet energy. 
\eqref{in I0 annulus there is more concentration to explore} ensures that there is no further concentration of energy in $B_{\rho^{i,l,j}_k}(x^{i,l,j}_k)$. \eqref{choice of concentration points in I0 annulus} ensures that each bubble has an energy of $\eps_0$. We apply again \cite[Lemma III.2]{BernardRiviere}, this time to the disjoint annuli $B_{\alpha^{i,l,j}R^{l}_k}(x^{i,l,j}_k)\setminus B_{\rho^{i,l,j}_k}(x^{i,l,j}_k)$ and repeat the procedure. We do this for each annulus of type $I_0$ and each neighborhood of the $a_i$. The process has to terminate as each $I_0$ annulus exhausts at least $\eps_0$ energy. 
\end{enumerate}
\end{enumerate} 
In this construction, the bubble $r^i$ from \eqref{root bubble} is the one satisfying $\rho^{i,r^i}_k = R^1_k$ (unless $N=2$, which corresponds to Case 1, for which $\rho^{i,r^i}_k = R^2_k$). We choose a subsequence and reindex the family of bubbles around each $a_i$ as $\{B_{\rho^{i,j}_k}(x^{i,j}_k),\, j\in \{1,\ldots, Q^i\}\},$ where $Q^i$ is a fixed integer. It remains to check that this satisfies the statements from the proposition. 
\begin{enumerate}[wide=0pt,leftmargin=\parindent, labelsep=0.5em,label = ]
\item[\textbf{Eq. \eqref{eq:point_convergence}, \eqref{eq:radius_convergence}, \eqref{eq:bubble_energy} and \eqref{root bubble}:}] These are immediate from the construction and \eqref{radius divergence}.
\item[\textbf{Eq. \eqref{eq:rho_ratio}:}] We need to check that for any two bubbles $B_{\rho^{i,j}_k}(x^{i,j}_k), B_{\rho^{i,j'}_k}(x^{i,j'}_k)$, $j\neq j'$, it holds
\begin{equation}
\frac{|x^{i,j}_k - x^{i,j'}_k|}{\rho^{i,j}_k + \rho^{i,j'}_k} \to \infty\label{centers far apart}
\end{equation} 
or
\begin{equation}
 \frac{\rho^{i,j}_k}{\rho^{i,j'}_k}+\frac{\rho^{i,j'}_k}{\rho^{i,j}_k}\to \infty\quad\text{as $k\to \infty$}.\label{radii on different scales}
\end{equation}
\eqref{centers far apart} occurs precisely when the bubbles are disjoint. \eqref{choice of alpha for first bubble}, \eqref{radius divergence}, and \eqref{choice of alphas for the I0 bubble} ensure that for the remaining cases, \eqref{radii on different scales} holds. 
\item[\textbf{Eq. \eqref{eq:double_neck_bound}:}] We choose $\alpha$ to be smaller than all the $\alpha^{i,j}$ and $\alpha^{i,l,j}$ chosen in \eqref{choice of alpha for first bubble} and \eqref{choice of alphas for the I0 bubble} and smaller than all ratios $\frac{R^{l+1}_k}{R^{l}_k}$, where $B_{R^{l+1}_k}(x^{i,j}_k)\setminus B_{R^l_k}(x^{i,j}_k)$ was some $I_0$ annulus during the construction process. This then ensures that all annuli $B_{2\rho}(x^{i,j}_k)\setminus B_{\rho}(x^{i,j}_k)\subset \Omega_k(\alpha_0)$ are contained in the union of the $I_1$ annuli and the union of the regions with small energy coming from \eqref{choice of alpha for first bubble}, \eqref{choice of alphas for the I0 bubble}, so that \eqref{eq:double_neck_bound} is implied by \eqref{doubling annulus condition}.
\item[\textbf{Eq. \eqref{eq:lower_bound_rho}:}] This statement follows as the induction process in Case 2b) is stopped only after \eqref{in I0 annulus there is more concentration to explore} fails, which ensures that there is no concentration of energy left\footnote{Notice that there is also no concentration in any of the $I_1$ annuli due to \eqref{doubling annulus condition}.}.
\end{enumerate}
\end{proof}

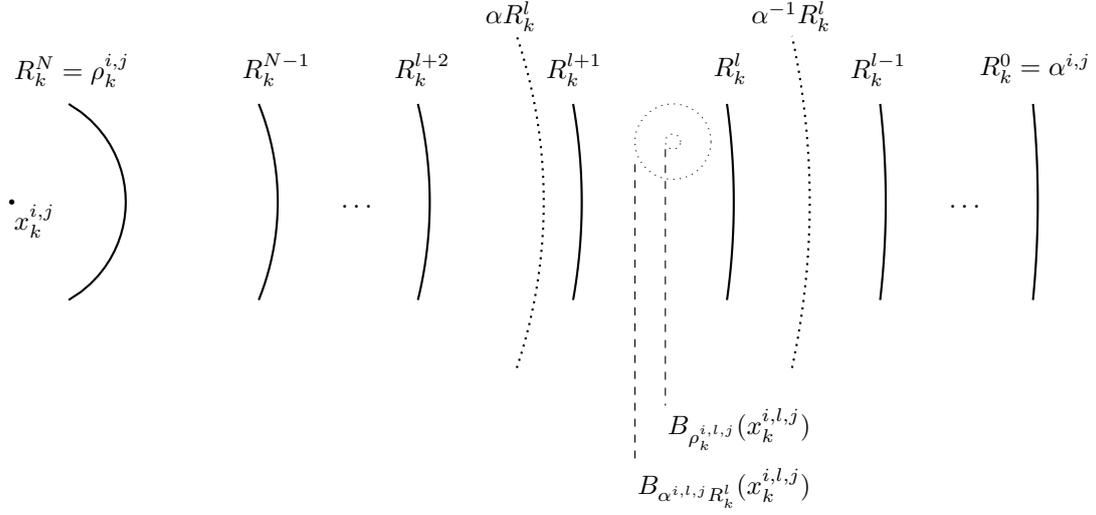
\begin{figure}[H] 
\centering 

\begin{tikzpicture}
\foreach \R in {0,...,6} {
   \draw [black, thick, domain=-1.3:1.3] plot ({sqrt((2*\R+1.5)^2 - (\x)^2)}, \x);
}
\node at (0,0) [circle,fill,inner sep=0.7pt]{};

\draw (0,2) node [anchor=north west][inner sep=0.75pt]    {$R^N_k = \rho^{i,j}_k$};

\draw (3,2) node [anchor=north west][inner sep=0.75pt]    {$R^{N-1}_k$};

\draw (5,2) node [anchor=north west][inner sep=0.75pt]    {$R^{l+2}_k $};
\draw (7,2) node [anchor=north west][inner sep=0.75pt]    {$R^{l+1}_k $};
\draw (9.2,2) node [anchor=north west][inner sep=0.75pt]    {$R^{l}_k $};
\draw (11,2) node [anchor=north west][inner sep=0.75pt]    {$R^{l-1}_k $};
\draw (12.7,2) node [anchor=north west][inner sep=0.75pt]    {$R^0_k = \alpha^{i,j}$};
\draw (4.3,0) node [anchor=north west][inner sep=0.75pt]    {$\ldots$};
\draw (12.3,0) node [anchor=north west][inner sep=0.75pt]    {$\ldots$};
\draw (0,0) node [anchor=north west][inner sep=0.75pt]    {$x^{i,j}_k$};
\draw (6.2,2.7) node [anchor=north west][inner sep=0.75pt]    {$\alpha R^{l}_k$};
\draw (9.7,2.7) node [anchor=north west][inner sep=0.75pt]    {$\alpha^{-1} R^{l}_k$};

\draw (8.6,-2.7) node [anchor=north west][inner sep=0.75pt]    {$B_{\rho^{i,l,j}_k}(x^{i,l,j}_k)$};
\draw (8.2,-3.5) node [anchor=north west][inner sep=0.75pt]    {$B_{\alpha^{i,l,j}R^{l}_k}(x^{i,l,j}_k)$};
   \draw [black, thick, dotted, domain=-2.2:2.2] plot ({sqrt(7^2 - (\x)^2)}, \x);
   \draw [black, thick, dotted, domain=-2.2:2.2] plot ({sqrt(10.5^2 - (\x)^2)}, \x);
   
   \draw[dotted] (8.7,0.8) circle (0.5);
   
   \draw[dotted] (8.7,0.8) circle (0.1);
   \draw[dashed] (8.6,0.7) -- (8.6, -2.7);
   \draw[dashed] (8.2,0.5) -- (8.2, -3.5);
   
\end{tikzpicture}
\caption{The construction of the bubbles.}

\end{figure}

\renewcommand{\thesection}{Appendix \Alph{section}}
\tocless\section{Derivation of conservative system}
\addcontentsline{toc}{section}{Appendix B. Derivation of conservative system}
\renewcommand{\thesection}{\Alph{section}}
\label{sec:Appendix Derivation Conservative System}
Suppose that $\vec \Phi:B_1 \to \R^3$ is a constrained Willmore surface with coefficients $\alpha,\beta,\gamma$ in a conformal parametrization. Denote 
\begin{equation}
\vec {\mc W} = -\frac{1}{2}(\nabla \vec H - 3\pi_{\vec n}(\nabla \vec H) + \nabla^\perp \vec n \times \vec H)\label{gradient of Willmore}
\end{equation}
and recall from \eqref{definition T} that
\begin{equation}
\vec T = - \alpha \nabla \vec \Phi + \frac{\beta}{2} \vec \Phi \times \nabla^\perp \vec \Phi - \frac{\gamma}{2} (\nabla \vec n + 2H \nabla \vec \Phi).\label{gradient of constraints}
\end{equation}
By \eqref{constrained Willmore surface}, it holds
\begin{equation}
\Div(\vec {\mc W} + \vec T) = 0.\label{constrained Willmore equation in appendix}
\end{equation}
By the Poincar\'e lemma applied to \eqref{constrained Willmore equation in appendix}, there is $\vec L:B_1\to \R^3$ satisfying
\begin{equation}
\nabla^\perp \vec L = -\vec {\mc W} - \vec T.\label{definition L in appendix}
\end{equation}
Then it holds, using also \cite[(7.19)+(7.20)]{RiviereLectureNotes} for the fifth equality
\begin{align}
\Div(\langle \vec L, \nabla^\perp \vec \Phi\rangle ) &= - \langle \nabla ^\perp \vec L, \nabla \vec \Phi\rangle\notag\\
&= -\left \langle \frac{1}{2}(\nabla \vec H - 3\pi_{\vec n}(\nabla \vec H) + \nabla^\perp \vec n \times \vec H)- \vec T, \nabla \vec \Phi\right \rangle\notag\\
&= -\left \langle \frac{1}{2}( H \nabla \vec n + \nabla^\perp \vec n \times \vec H) - \vec T, \nabla \vec \Phi\right \rangle\notag\\
 &= \langle \vec T, \nabla \vec \Phi \rangle - \frac{1}{2} \langle - (H \nabla \vec n + \vec H \times \nabla ^\perp \vec n) + 2 H \nabla \vec n, \nabla \vec \Phi \rangle \notag\\
 &=\langle \vec T, \nabla \vec \Phi \rangle - \frac{1}{2} (4 e^{2\lambda} H^2 - 4e^{2\lambda}H^2) = \langle \vec T, \nabla \vec \Phi\rangle\notag\\
 &= e^{2\lambda}(-2\alpha) + \frac{\beta}{2} \langle\nabla ^\perp \vec \Phi \times \nabla \vec \Phi , \vec \Phi \rangle - \frac{\gamma}{2} \langle \nabla \vec n + 2H \nabla \vec \Phi, \nabla \vec \Phi \rangle\notag\\
 &= e^{2\lambda} ( -2\alpha + \beta \langle \vec n, \vec \Phi\rangle - \gamma H ). \label{divergence form for balancing integrand}
\end{align}
We set $Y:B_1\to \R$ to be the solution to 
\begin{equation}
\begin{cases}
\Delta Y = e^{2\lambda} ( -2\alpha + \beta \langle \vec n, \vec \Phi\rangle - \gamma H )&\quad\text{in $B_1$,}\\
\;\;\;Y=0&\text{on $\partial B_1$.}
\end{cases}\label{definition Y in appendix}
\end{equation}
It follows
\begin{equation}
\Div(\langle \vec L, \nabla^\perp \vec \Phi\rangle - \nabla Y) = 0.\label{divergence freeness for S in appendix}
\end{equation}
Applying again the Poincar\'e lemma to \eqref{divergence freeness for S in appendix}, there is $S:B_1\to \R$ such that
\begin{equation}
\langle \vec L , \nabla^\perp \vec \Phi \rangle - \nabla Y = \nabla^\perp S.\label{definition S in appendix}
\end{equation}
Next, it holds
\begin{equation}
\Div(\vec L \times \nabla ^\perp \vec \Phi) = \nabla \vec \Phi \times \nabla^\perp \vec L = \nabla \vec \Phi \times (- \vec{\mc W} - \vec T).\label{divergence form for R in appendix}
\end{equation}
It holds by \cite[(7.22)]{RiviereLectureNotes}
\begin{align}
\nabla \vec \Phi \times (-\vec{\mc W}) &= \frac{1}{2} \nabla \vec \Phi \times (\nabla \vec H - 3\pi_{\vec n}(\nabla \vec H) + \nabla^\perp \vec n \times \vec H) = \frac{1}{2} \nabla \vec \Phi \times (\nabla \vec H - 3\nabla H \vec n + \nabla^\perp \vec n \times \vec H)\notag\\
&= \frac{1}{2} \nabla \vec \Phi \times (-2 \nabla \vec H + 3H \nabla \vec n + \nabla^\perp \vec n \times \vec H)=\frac{1}{2} \nabla \vec \Phi \times ((\nabla^\perp \vec n \times \vec H - H \nabla \vec n) + (4H \nabla \vec n - 2\nabla \vec H))\notag\\
&= \frac{1}{2} \nabla \vec \Phi \times (2H \nabla \vec n - 2\nabla H \vec n)= \underbrace{\nabla \vec \Phi \times \nabla \vec n }_{=0} H -  (\nabla \vec \Phi \times \vec n) \cdot \nabla H\notag\\
&= \nabla ^\perp \vec \Phi \cdot  \nabla  H .\label{setup for R in appendix}
\end{align}
It follows from \eqref{divergence form for R in appendix} and \eqref{setup for R in appendix}
\begin{align}
\Div(\vec L \times \nabla^\perp \vec \Phi - \nabla ^\perp \vec \Phi H) &= - \nabla \vec \Phi \times \vec T \notag\\
&=- \nabla \vec \Phi \times \left (- \alpha \nabla \vec \Phi + \frac{\beta}{2} \vec \Phi \times \nabla^\perp \vec \Phi - \frac{\gamma}{2} (\nabla \vec n + 2H \nabla \vec \Phi)\right )\notag\\
&= -\frac{\beta}{2} \vec \Phi \langle \nabla \vec \Phi, \nabla^\perp \vec \Phi\rangle +\frac{\beta}{2}\nabla^\perp \vec \Phi \cdot \langle \nabla \vec \Phi,  \vec \Phi\rangle\notag \\
 &=\frac{\beta}{4} \nabla (|\vec \Phi|^2) \cdot \nabla ^\perp \vec \Phi.\label{final setup for R in appendix}
\end{align}
We set $\vec X =\frac{\beta}{4} |\vec \Phi|^2 \nabla^\perp \vec \Phi$. Then the right-hand side of \eqref{final setup for R in appendix} equals $\Div(\vec X)$ and so applying the Poincar\'e lemma a third time to \eqref{final setup for R in appendix} yields that there is $\vec R:B_1\to \R$ such that
\begin{align}
\nabla^\perp \vec R &= \vec L \times \nabla^\perp \vec \Phi - \nabla ^\perp \vec \Phi H - \vec X =\vec L \times \nabla^\perp \vec \Phi + \vec H \times \nabla  \vec \Phi  - \vec X. \label{definition R in appendix}
\end{align}
Furthermore, as in \cite[(7.28)]{RiviereLectureNotes}, using the definition \eqref{definition S in appendix} and \eqref{definition R in appendix},
\begin{align}
\vec n \times \nabla^\perp \vec R &= \vec n \times (\vec L \times \nabla^\perp \vec \Phi + \vec H \times \nabla  \vec \Phi - \vec X) = -\langle \vec L ,\nabla^\perp \vec \Phi \rangle \vec n + \vec L \times \nabla \vec \Phi - \vec H \times \nabla^\perp  \vec \Phi- \frac{\beta}{4} |\vec \Phi|^2 \nabla \vec \Phi\notag\\
&= - (\nabla^\perp S + \nabla Y)  \vec n + \vec L \times \nabla  \vec \Phi - \vec H \times \nabla^\perp  \vec \Phi - \frac{\beta}{4} |\vec \Phi|^2 \nabla \vec \Phi\notag\\
&= - (\nabla^\perp S + \nabla Y)  \vec n  + \nabla \vec R.\label{setup alternative definition for nabla R in appendix}
\end{align}
Rewriting \eqref{setup alternative definition for nabla R in appendix} shows
\begin{equation}
\nabla \vec R = \vec n \times \nabla ^\perp \vec R+\nabla ^\perp S \vec n + \nabla Y \vec n.\label{definition nabla R alternative in appendix}
\end{equation}
\eqref{definition nabla R alternative in appendix} also yields
\begin{equation}
\langle \vec n, \nabla^\perp \vec R\rangle = - \langle \vec n \times \nabla \vec R , \vec n\rangle - \nabla S + \nabla^\perp Y \label{setup alternative definition nabla S in appendix}
\end{equation}
and so
\begin{equation}
\nabla S = - \langle \vec n, \nabla^\perp \vec R\rangle + \nabla^\perp Y.\label{nabla S alternative definition in appendix}
\end{equation}
Finally, as $\Delta \vec \Phi  = 2 e^{2\lambda} \vec H$, we see
\begin{align}
\nabla \vec \Phi \times \nabla^\perp \vec R &= \nabla \vec \Phi \times (\vec L \times \nabla^\perp \vec \Phi) + \nabla \vec \Phi \times (\vec H \times \nabla \vec \Phi) - \nabla \vec \Phi \times \vec X = - \langle \nabla \vec \Phi, \vec L \rangle \cdot \nabla ^\perp \vec \Phi + 2e^{2\lambda}\vec H - \nabla \vec \Phi \times \vec X\notag\\
&= - (\nabla S - \nabla^\perp Y)\cdot \nabla^\perp \vec \Phi + \Delta \vec \Phi - \nabla \vec \Phi \times \vec X.\label{setup Laplace Phi}
\end{align}
From \eqref{setup Laplace Phi}, we conclude
\begin{equation}
\Delta \vec \Phi = \nabla \vec \Phi \times \nabla^\perp \vec R + \nabla S \cdot \nabla^\perp \vec \Phi - \nabla Y \cdot \nabla \vec \Phi - \frac{\beta}{2} |\vec \Phi| ^2 e^{2\lambda }\vec n.\label{Laplace Phi in appendix}
\end{equation}
\eqref{definition nabla R alternative in appendix}, \eqref{nabla S alternative definition in appendix}, and \eqref{Laplace Phi in appendix} yield the conservative system \eqref{conservative system}.

\printbibliography

(C. Scharrer) \textsc{Institute for applied mathematics, University of Bonn, Endenicher Allee 60, 53115 Bonn, Germany.}\\
\emph{Email address:} \href{mailto:scharrer@iam.uni-bonn.de}{scharrer@iam.uni-bonn.de} 
\medskip\newline
(A. West) \textsc{Institute for applied mathematics, University of Bonn, Endenicher Allee 60, 53115 Bonn, Germany.}\\
\emph{Email address:} \href{mailto:west@iam.uni-bonn.de}{west@iam.uni-bonn.de} 
\end{document}

%% file: preamble.tex
\usepackage{tcolorbox}
\usepackage{enumitem}
\usepackage[utf8]{inputenc}
\usepackage{dsfont}
\usepackage{amssymb,amsthm,amsmath}
\usepackage{amsfonts}
\usepackage{comment} 
\usepackage[usestackEOL]{stackengine} 
\let\oldmathchoice\mathchoice
\usepackage{mathtools} 
\usepackage{mathrsfs} 
\usepackage{scalerel}
\usepackage{breqn}
\let\newmathchoice\mathchoice
\usepackage{microtype}
\usepackage{graphicx}
\usepackage{float}
\usepackage[USenglish]{babel}
\usepackage[T1]{fontenc}
\usepackage{lmodern}
\usepackage{pdfpages}
\usepackage{caption}
\usepackage[titletoc]{appendix} 
\usepackage{subcaption}
\usepackage{physics} 
\usepackage{setspace} 
\usepackage{parskip}
\usepackage{tikz}
\usepackage{pgfplots} 
\usepackage[textsize=tiny]{todonotes}
\usepackage{bm} 
\usepackage{marginnote}

\usepackage{titling} 
\usepackage[final]{showlabels} 
\usepackage{cancel}
\usetikzlibrary{patterns}
\usetikzlibrary{cd, babel}
\usepackage{csquotes}
\usepackage[%
  style=numeric, 
  sortcites,
  backend=biber,
  giveninits=true, 
  isbn= false,
  url=false,
  doi = false,
  eprint=false,
  date=year,
  ]{biblatex}
\usepackage[colorlinks=true,linkcolor=blue,citecolor=blue]{hyperref}

\usepackage{cleveref}

\theoremstyle{plain}

\newtheorem{prop}{Proposition}[section]
\newtheorem*{prop*}{Proposition}
\newtheorem{lemma}[prop]{Lemma}

\newtheorem{thm}[prop]{Theorem}

\newtheorem{cor}[prop]{Corollary}

\theoremstyle{definition}

\theoremstyle{remark}
\newtheorem{remark}[prop]{Remark}
\newtheorem*{ack}{Acknowledgement}

%% file: makros.tex
\newcommand{\nocontentsline}[3]{}
\newcommand{\tocless}[2]{\bgroup\let\addcontentsline=\nocontentsline#1{#2}\egroup}

\let\temp\phi
\let\phi\varphi
\let\varphi\temp
\newcommand*\dif{\,\mathrm{d}}

\DeclareMathOperator{\cqq}{\coloneqq}
\DeclareMathOperator{\qqc}{\eqqcolon}
\DeclareMathOperator{\wto}{\rightharpoonup}

\newcommand{\vertiii}[1]{{\left\vert\kern-0.25ex\left\vert\kern-0.25ex\left\vert #1 
    \right\vert\kern-0.25ex\right\vert\kern-0.25ex\right\vert}}

\newcommand\N{\mathbb{N}}

\newcommand\R{\mathbb{R}}

\newcommand\C{\mathbb{C}}

\newcommand{\mc}[1]{\mathcal #1}

\newcommand{\mb}[1]{\mathbb #1}
\newcommand{\ms}[1]{\mathscr #1}

\DeclareMathOperator{\id}{Id}

\newcommand{\vecto}[2]{\begin{pmatrix} #1\\#2\end{pmatrix}}

\newcommand{\parpar}[2]{\partial_{#1} #2}

\newcommand{\eps}{\varepsilon}

\DeclareMathOperator{\Div}{div}
\DeclareMathOperator{\loc}{loc}

\def\dashint{\let\mathchoice\oldmathchoice\,\ThisStyle{\ensurestackMath{%
            \stackinset{c}{.2\LMpt}{c}{.5\LMpt}{\SavedStyle-}{%
            \SavedStyle\phantom{\int}}}%
        \setbox0=\hbox{$\SavedStyle\int\,$}\kern-\wd0}\int%
        \let\mathchoice\newmathchoice}

\DeclareMathOperator{\diam}{diam}

\DeclareMathOperator{\iso}{\mc{I}}
